\documentclass[11pt]{amsart}

\allowdisplaybreaks
\usepackage[normalem]{ulem}

\usepackage{parskip}
\usepackage{geometry} 
\usepackage{framed} 

\usepackage{tikz}
\tikzset{every picture/.style={line width=1pt}} 

\usepackage{setspace}
\setdisplayskipstretch{2.5}

\usepackage{amsthm} 
\usepackage{amsmath}
\usepackage{amssymb}

\usepackage{hyperref}

\usepackage{amsfonts} 
\newcommand\A{\mathcal{A}}
\newcommand\B{\mathcal{B}}
\newcommand\FB{\mathfrak{B}}

\newcommand\CC{\mathcal{C}}

\newcommand\BC{\mathbf{C}}
\newcommand\CF{\mathcal{F}}

\newcommand\G{\mathcal{G}}
\renewcommand\H{\mathcal{H}}

\newcommand\N{\mathbb{N}}

\newcommand\R{\mathbb{R}}
\renewcommand\S{\mathcal{S}}
\newcommand\Z{\mathbb{Z}}

\newcommand\one{\mathbf{1}}

\newcommand\w{\omega}
\newcommand\vphi{\varphi}
\renewcommand\phi{\vphi}

\usepackage{stmaryrd} 

\newcommand\sing{\textnormal{sing}}
\newcommand\reg{\textnormal{reg}}
\newcommand\spt{\textnormal{spt}}
\newcommand\Lip{\textnormal{Lip}}
\newcommand\dist{\textnormal{dist}}
\newcommand\graph{\textnormal{graph}}
\newcommand\ext{\mathrm{d}}
\newcommand\del{\partial}

\newcommand{\res}{\mathbin{\hspace{0.1em}\vrule height 1.3ex depth 0pt width 0.13ex\vrule height 0.13ex depth 0pt width 1.0ex}} 

\newcommand{\weakly}{\rightharpoonup}
\renewcommand{\div}{\textnormal{div}}

\makeatletter
\def\@tocline#1#2#3#4#5#6#7{\relax
  \ifnum #1>\c@tocdepth 
  \else
    \par \addpenalty\@secpenalty\addvspace{#2}%
    \begingroup \hyphenpenalty\@M
    \@ifempty{#4}{%
      \@tempdima\csname r@tocindent\number#1\endcsname\relax
    }{%
      \@tempdima#4\relax
    }%
    \parindent\z@ \leftskip#3\relax \advance\leftskip\@tempdima\relax
    \rightskip\@pnumwidth plus4em \parfillskip-\@pnumwidth
    #5\leavevmode\hskip-\@tempdima
      \ifcase #1
       \or\or \hskip 1em \or \hskip 2em \else \hskip 3em \fi%
      #6\nobreak\relax
    \dotfill\hbox to\@pnumwidth{\@tocpagenum{#7}}\par
    \nobreak
    \endgroup
  \fi}
\makeatother

\hypersetup{
           breaklinks=true,   
           pdfusetitle=true,  
        }

\newtheoremstyle{newtheoremstyle}
{3pt}
{3pt}
{\itshape}
{\parindent}
{\bfseries}
{.}
{0.5em}
{} 
\newtheoremstyle{newtheoremstyledefn}
{3pt}
{3pt}
{}
{\parindent}
{\bfseries}
{.}
{0.5em}
{} 

\theoremstyle{newtheoremstyle}
\newtheorem{theorem}{Theorem}
\newtheorem*{theorem*}{Theorem}
\newtheorem{lemma}[theorem]{Lemma}
\newtheorem{prop}[theorem]{Proposition}

\newtheorem{thmx}{Theorem}

\theoremstyle{newtheoremstyledefn}
\newtheorem{defn}[theorem]{Definition}

\numberwithin{equation}{part} 
\numberwithin{theorem}{part}

\newtheorem{theoremA}{Theorem}

\newtheorem{lemmaA}[theoremA]{Lemma}

\usepackage{fancyhdr}
\begin{document}

\title{A structure theory for stable codimension 1 integral varifolds with applications to area minimising hypersurfaces mod $p$}

\author{
	Paul Minter
	\and
	Neshan Wickramasekera
}
\thanks{\SMALL The first author was supported by the UK Engineering and Physical Sciences Research Council (EPSRC) grant EP/L016516/1 for the University of Cambridge Centre for Doctoral Training, the Cambridge Centre for Analysis.}

\address{\textnormal{Department of Pure Mathematics and Mathematical Statistics, University of Cambridge}}
\email{pdtwm2@cam.ac.uk\\
\and
N.Wickramasekera@dpmms.cam.ac.uk}

\begin{abstract}
For any $Q\in\{\frac{3}{2},2,\frac{5}{2},3,\dotsc\}$, we establish a structure theory for the class $\S_Q$ of stable codimension 1 stationary integral varifolds admitting no classical singularities of density $<Q$. This theory comprises three main theorems which describe the nature of a varifold $V\in \S_Q$ when: (i) $V$ is close to a flat disk of multiplicity $Q$ (for integer $Q$); (ii) $V$ is close to a flat disk of integer multiplicity $<Q$; and (iii) $V$ is close to a stationary cone with vertex density $Q$ and support the union of 3 or more half-hyperplanes meeting along a common axis. The main new result concerns (i) and gives in particular a description of $V\in \S_Q$ near branch points of density $Q$. Results concerning (ii) and (iii) directly follow from parts of the work \cite{wickramasekera2014general} (and are reproduced in Part~\ref{other-results}).

These three theorems, taken with $Q=p/2$, are readily applicable to codimension 1 rectifiable area minimising currents mod $p$ for any integer $p\geq 2$, establishing local structure properties of such a current $T$ as consequences of little, readily checked, information. Specifically, applying case (i) it follows that, for even $p$, if $T$ has one tangent cone at an interior point $y$ equal to an (oriented) hyperplane $P$ of multiplicity $p/2$, then $P$ is the unique tangent cone at $y$, and $T$ near $y$ is given by the graph of a $\frac{p}{2}$-valued function with $C^{1,\alpha}$ regularity in a certain generalised sense. This settles a basic remaining open question in the study of the local structure of such currents near points with planar tangent cones, extending the cases $p=2$ and $p=4$ of the result  which have been known since the 1970's from the De~Giorgi--Allard regularity theory  (\cite{allard1972first}) and the structure theory of White (\cite{white1979structure}) respectively. If $P$ has multiplicity $< p/2$ (for $p$ even or odd), it follows from case (ii) that $T$ is smoothly embedded near $y$, recovering a second well-known theorem of White (\cite{white1984regularity}). Finally, the main structure results obtained recently by De~Lellis--Hirsch--Marchese--Spolaor--Stuvard (\cite{de2021area}) for such currents $T$ all follow from case (iii).
\end{abstract} 

\maketitle

\tableofcontents

\part{Introduction, the main theorem, and notation} 
\setcounter{section}{1}
From the work of B.~White in the late 1970's and the early 1980's, it is known that an area minimising hypersurface (i.e.\ a codimension 1 rectifiable (representative) current) mod $p$ in $\R^{n+1}$ (or more generally, in an $(n+1)$-dimensional Riemannian manifold) near an interior point $y$ where one tangent cone is a multiplicity $q$ $(> 0)$ hyperplane is smoothly embedded if $q < p/2$ (\cite{white1984regularity}), and  smoothly immersed if $p=4$ and 
$q = p/2 = 2$ (\cite{white1979structure}). If $p=2$, then a planar tangent cone has multiplicity $q= 1$, and so in this case embeddedness holds near $y$ as a consequence of the De~Giorgi--Allard regularity theory (\cite{allard1972first}, \cite{simon1983lectures}). For even $p >4$, it has remained an open question as to what can be said about the structure of the hypersurface near such a point $y$ when $q = p/2$, including whether the tangent cone at $y$ must be unique. In this direction, a recent result of De~Lellis--Hirsch--Marchese--Spolaor--Stuvard (\cite{de2021area}) gives that all tangent cones at $y$ must be supported on hyperplanes. (In fact, this result does not require the minimising mod $p$ hypothesis and follows from simple connectivity properties of the set of tangent cones at a point combined with regularity theorems for stable hypersurfaces, see \cite[Chapter 7, Corollary 7]{minterthesis}).

Here we settle this uniqueness question affirmatively, and show in fact that when $q=p/2$ the current near $y$ is given by a multi-valued (in fact $\frac{p}{2}$-valued) Lipschitz graph over the tangent hyperplane, with $C^{1, \alpha}$ regularity in a certain generalised sense (Theorem~\ref{thm:D}(i) in Section~\ref{application} below).  Since the multiplicity of a (representative) mod $p$ minimiser is $\leq p/2$ a.e., understanding the case $q=p/2$ has been arguably the most basic remaining question concerning regularity of area minimising hypersurfaces mod $p$. 

As it turns out, this and a number of other results for mod $p$ minimising hypersurfaces are in fact very direct consequences of a much more general structure theory that is applicable to stable codimension 1 integral varifolds. Indeed,  for any given $Q \in \{\frac{3}{2}, 2, \frac{5}{2}, 3, \ldots\}$, we here consider the class $\S_{Q}$ of stable codimension 1 integral varifolds admitting no \textit{classical singularities} of density $< Q$ (i.e.\  no singularities near which the varifold is expressible as a sum of $< 2Q$ embedded $C^{1,\alpha}$ hypersurfaces-with-boundary, counted with multiplicity, meeting only along a common $(n-1)$-dimensional $C^{1,\alpha}$ boundary; see the definition in Section~\ref{notation}).  Our main result, Theorem~\ref{thm:A} (in Section~\ref{mainresult} below), shows that for $Q$ a positive integer, if a varifold $V \in \S_{Q}$ lies close to a multiplicity $Q$ hyperplane in an open cylinder, then in a smaller cylinder, it is given by a possibly branched $Q$-valued graph with generalised-$C^{1, \alpha}$ regularity (in the sense of Definition~\ref{genC1alpha}). We note that unless $Q=2$, $C^{1, \alpha}$ regularity in the usual sense can fail since for $Q \geq 3$, as a density $Q$ classical singularity need not be an immersed point (see Figure \ref{fig:0}). 

Further information on $\S_{Q}$ is provided by two additional results, both of which follow very directly from the work \cite{wickramasekera2014general}: Theorem~\ref{sheeting} (in Section~\ref{wic-sheeting} below), which says that a varifold $V \in \S_{Q}$ is embedded in the interior if $V$ is close to a flat disk of multiplicity $< Q$; and Theorem~\ref{thm:B} (in Section~\ref{other} below), which says that $V \in \S_{Q}$ has the structure of a density $Q$ classical singularity  in the interior whenever $V$ is, in a ball, sufficiently close to a stationary cone having vertex density $Q$ and support the union of 3 or more half-hyperplanes meeting along a common axis.  

Turning again to area minimising hypersurfaces mod $p$, it is easily seen, by a standard ``wedge-replacement'' comparison, that a 1-dimensional singular cone which is a locally length minimising rectifiable current mod $p$ with zero mod $p$ boundary must have density at the origin $\geq p/2$; consequently, any classical singularity of an $n$-dimensional area minimising hypersurface mod $p$ must have density $\geq p/2$. Thus for area minimising hypersurfaces mod $p$, Theorem~\ref{thm:A} (taken with $Q = p/2$) readily implies the aforementioned uniqueness-of-tangent-cone and structure results for even $p$, and Theorem~\ref{sheeting} (also taken with $Q = p/2$) readily implies the aforementioned embeddedness result of White for arbitrary $p$ (giving a new proof of that result). Likewise, Theorem~\ref{thm:B} (taken with $Q = p/2$) readily implies the structure results obtained recently in \cite{de2021area}\footnote{The proof given in \cite{de2021area} for the main structure results therein is similar to the proof (found in \cite[Section~16]{wickramasekera2014general}) of Theorem~\ref{thm:B}, although the former makes explicit use of the minimising hypothesis in places of the argument where the latter uses stability; Theorem~\ref{thm:B} is not stated in \cite{wickramasekera2014general} in its explicit form given in (Section~\ref{other} of) the present work, but the main decay result from which it follows (\cite[Lemma 16.9]{wickramasekera2014general}) is (see the proof of Theorem \ref{thm:B}); the authors of  \cite{de2021area} have informed us that their work had been carried out unaware of this.}; this is so in light of the additional fact that the density at a classical singularity of a mod $p$ area minimising hypersurface must in fact be equal to $p/2$---which, as observed in \cite[Proposition~3.2]{de2021area}, is seen by a further simple 1-dimensional comparison argument. These consequences to mod $p$ area minimising hypersurfaces are the content of Theorem~\ref{thm:D} in Section~\ref{application} below. 

Theorem~\ref{thm:A}, Theorem~\ref{sheeting}, and Theorem~\ref{thm:B} thus provide a unified regularity theory, which in particular shows, when taken with $Q = p/2$, that the regularity properties of codimension 1 locally area minimising rectifiable currents mod $p$ are in fact consequences of a small amount of readily extracted information: the stability of the regular part of the current and the absence of certain classical singularities. Though perhaps surprising at a first glance, this is entirely analogous to the fact that the regularity theory for codimension 1 area minimising \emph{integral} currents is an immediate consequence of the regularity theory for $\S_{\infty} \equiv \cap_{Q} \S_{Q}$, i.e.\ for stable codimension 1 integral varifolds with \emph{no} classical singularities (\cite{wickramasekera2014general}).

Our proof of Theorem~\ref{thm:A} relies in part on certain results and techniques developed in  \cite{wickramasekera2014general} (including Theorem~\ref{sheeting}) which in turn are based on the foundational work  of L.~Simon (\cite{simon1993cylindrical}), R.~Schoen \& L.~Simon (\cite{schoen1981regularity}), and of F.~J.~Almgren~Jr (\cite{almgrenalmgren}). Additionally,  we prove and use monotonicity of the frequency function associated with certain $Q$-valued functions arising from varifolds in $\S_{Q}.$ Almgren originally introduced the frequency function 
in \cite{almgrenalmgren} to study the branching behaviour of area minimising integral currents (of codimension $\geq 2$) and of the closely related $Q$-valued Dirichlet energy minimising functions. In the present context, there is no minimising hypothesis on the varifolds nor is 
there an a priori readily verifiable variational principle satisfied by the associated $Q$-valued functions; moreover, our use of the frequency function is for establishing uniform regularity estimates at branch points, a purpose different from its main use in Almgren's work, which was to bound the Hausdorff dimension of the branch set.

{\bf Guide to the paper:} The statement of the main new structure result (Theorem~\ref{thm:A}) is contained in Section~\ref{mainresult}; its corollary, with justification, to area minimising hypersurfaces mod $p$ (Theorem~\ref{thm:D}(i)) is contained in Section~\ref{application}. The readers primarily interested in these results, who are also familiar with the work \cite{wickramasekera2014general}, may wish to read the statements of Theorem~\ref{thm:A} and Theorem~\ref{thm:D}(i) and proceed directly to the proof of Theorem~\ref{thm:A} in Part~\ref{main-thm-proof} (page \pageref{main-thm-proof}). For the benefit of the general reader and the reader who wishes to gain a more comprehensive understanding, we have attempted to motivate this work  and to minimise the need to refer to \cite{wickramasekera2014general} by including a significant amount of introductory and preliminary material, all of which is contained in the rest of this introduction and in Part~\ref{other-results}. 

\subsection{Stable varifolds, branch points, and classical singularities} \label{branch-classical} Consider a stationary integral $n$-varifold $V$ on an open ball $B$ in ${\mathbb R}^{n+1}$ or, more generally,  on an $(n+1)$-dimensional Riemannian manifold. 
It is a well known consequence of the Allard regularity theory (\cite{allard1972first}, \cite{simon1983lectures}) that the regular part 
${\rm reg} \, V$ of the varifold (i.e.\ the smoothly embedded part of ${\rm spt} \, \|V\|$, where $\|V\|$ is the weight measure associated with $V$) is a dense open subset of ${\rm spt} \, \|V\|.$ It is not known whether the (interior) singular set 
${\rm sing} \, V =( {\rm spt} \, \|V\| \setminus {\rm reg} \, V) \cap B$ must have zero $n$-dimensional Hausdorff measure; nor is there much understanding about the behaviour of $V$ on approach to the singular set beyond the fact that non-trivial tangent cones can be produced at every singular point as limits of sequences of scalings of $V$ about that point, with different sequences of scalings conceivably producing different limits. The difficulty of these questions  lies primarily in issues arising from the occurrence of higher (i.e.\ $\geq 2$) multiplicity on sets of positive $n$-dimensional Hausdorff measure in weak limits of such varifolds (including in the given varifold $V$ itself or in its tangent cones); in particular, it remains a basic open question to understand the nature of $V$ in the vicinity of a \emph{branch point singularity}, i.e.\ a singular point where one tangent cone is a hyperplane with (constant) integer multiplicity $\geq 2$ and yet in no neighbourhood about that point is ${\rm spt} \, \|V\|$ the union of smoothly embedded minimal hypersurfaces.  

If however $V$ is \emph{stable}, i.e.\ if every two-sided portion of 
${\rm reg} \, V$ has non-negative second variation with respect to the mass functional for compactly supported normal deformations, then it is known (see \cite{wickramasekera2014general})  
that either there is a \emph{classical singularity} $y$ in $V$---that is to say, there is a point $y \in {\rm spt} \, \|V\| \cap B$, an ambient open ball $B(y)  \subset B$ centred at $y,$ and a number $\alpha \in (0, 1)$ such that 
${\rm spt} \, \|V\| \cap B(y)$ is made up of a finite number of at least 3 embedded $C^{1, \alpha}$ hypersurfaces-with-boundary coming together smoothly and transversely only along a common $C^{1, \alpha}$ embedded $(n-1)$-dimensional submanifold containing $y$---or else the entire singular set ${\rm sing} \, V$ has Hausdorff dimension at most $n-7$. (By \cite{krummel2014regularity}, the ``$C^{1, \alpha}$'' in the definition of classical singularity can be replaced by ``smooth'' if the ambient metric is smooth, or by ``real analytic'' if the ambient metric is so.) A central part of the proof of this result involves ruling out branch points altogether in the absence of classical singularities. This is achieved via a certain ``sheeting theorem'', \cite[Theorem~3.3]{wickramasekera2014general} (see Theorem~\ref{sheeting} below). 

In the presence of classical singularities, stable codimension 1 varifolds can develop branch points (\cite{simon2007stable}, \cite{krummel2019existence}). Since the absence of 
classical singularities in a punctured ball $B^{n+1}_{\rho}(y) \setminus \{y\}$ implies (by the definition of classical singularity) the absence of classical singularities in $B^{n+1}_{\rho}(y),$ 
the sheeting theorem \cite[Theorem~3.3]{wickramasekera2014general} implies that every branch point of a stable codimension 1 varifold must be a limit point of classical singularities. 

En route to this sheeting theorem, there are several intermediate results that have been established in \cite{wickramasekera2014general} that readily provide structural information, including uniqueness of certain non-planar tangent cones, for stable codimension 1 varifolds without the need to rule out \emph{all} classical singularities. For instance, the following two results (a) and (b) can be extracted from Section~16 and Section~14 of \cite{wickramasekera2014general} respectively: let $Q \in \{\frac{3}{2}, 2, \frac{5}{2}, 3, \ldots\}$ and let $V$ be a stable codimension 1 integral varifold with no classical singularities of density $< Q$.

\begin{itemize}
\item[(a)] If one tangent cone to $V$ at a point $Z$ is a classical cone of density (at the origin) equal to $Q,$ then it is the unique tangent cone at $Z$ and in fact $Z$ is a classical singularity of $V$ (\cite[Lemma~16.9 and the proof of Theorem~16.1]{wickramasekera2014general}; see Theorem~\ref{thm:B} in Section~\ref{other} below); 

\item[(b)] If $Z$ is a point of $V$ with density equal to $Q$ such that, at some small scale about $Z$, $V$ is close to a multiplicity $Q$ hyperplane $\mathbf{L}$ and is significantly closer to a (not necessarily stationary) classical cone $\BC$,  then $V$ has a (unique) classical tangent cone at $Z$  and hence (by (a) above) $Z$ is a classical singularity of $V$; here the degree of closeness of $V$ to ${\mathbf L}$ and to $\BC$ are determined by fixed thresholds, depending only on $n$ and $Q$, for the $L^{2}$ distance (height excess) ${E}(V, {\mathbf L})$ between $V$ and ${\mathbf L}$ and for the 
ratio  $\frac{E(V, \BC)}{\inf_{{\mathbf L}'} \, {E}(V, {\mathbf L}')}$ respectively (\cite[proof of Lemma~14.1]{wickramasekera2014general}; see Theorem~\ref{thm:fine_reg} below, which is a strengthening of this assertion giving uniform estimates where the constants are independent of $\BC$).
\end{itemize}
Here and subsequently, a \emph{classical cone} means a cone supported on at least 3 distinct half-hyperplanes meeting along a common $(n-1)$-dimensional axis, and a \textit{classical tangent cone} is a tangent cone which is a classical cone.
Both these results concern singularities with classical tangent cones (a priori or a posteriori, respectively), and they naturally raise the following question: what must a stable codimension 1 integral varifold $V$ with no classical singularities of density $< Q$ look like near a branch point of density $Q$?

\subsection{Varifold class $\S_{Q}$ and the main result: Theorem~A} \label{mainresult} 
We give an answer to the above question in Theorem~\ref{thm:A} below, which implies that, near a branch point, a varifold $V$ as above has the structure of a $Q$-valued graph with ``generalised-$C^{1, \alpha}$'' regularity, and in particular has unique tangent cones at every nearby 
point. 

First we define the class of varifolds we shall be concerned with in Theorem~\ref{thm:A} and subsequently. 

{\bf Definition:} Let $Q \in \{\frac{3}{2}, 2, \frac{5}{2}, 3, \ldots \}$. Let $\S_{Q}$ denote that class of integral $n$-varifolds $V$ on the open ball $B^{n+1}_2(0) \subset \R^{n+1}$ with $0 \in {\rm spt} \, \|V\|$, 
$\|V\|(B_{2}^{n+1}(0)) < \infty$ and which satisfy the following conditions: 

\begin{enumerate}
	\item [$(\S1)$\;\;\,] $V$ is stationary in $B^{n+1}_2(0)$ with respect to the area functional, in the following (usual) sense: for any given vector field $\psi\in C^1_c(B^{n+1}_2(0);\R^{n+1})$, $\epsilon>0$, and $C^2$ map $\phi:(-\epsilon,\epsilon)\times B^{n+1}_2(0)\to B^{n+1}_2(0)$ such that:
	\begin{enumerate}
		\item [(i)] $\phi(t,\cdot):B^{n+1}_2(0)\to B^{n+1}_2(0)$ is a $C^2$ diffeomorphism for each $t\in (-\epsilon,\epsilon)$ with $\phi(0,\cdot)$ equal to the identity map on $B^{n+1}_2(0)$,
		\item [(ii)] $\phi(t,x) = x$ for each $(t,x)\in (-\epsilon,\epsilon)\times\left(B^{n+1}_2(0)\setminus\spt(\psi)\right)$, and
		\item [(iii)] $\left.\del\phi(t,\cdot)/\del t\right|_{t=0} = \psi$,
	\end{enumerate}
	we have that
	$$\left.\frac{\ext }{\ext t}\right|_{t=0} \|\phi(t,\cdot)_\# V\|(B^{n+1}_2(0)) = 0;$$
	equivalently (see \cite[Section 39]{simon1983lectures}),
	$$\int_{B^{n+1}_2(0)\times G_n}\div_S\psi(X)\ \ext V(X,S) = 0,$$
	for every vector field $\psi\in C^1_c(B^{n+1}_2(0);\R^{n+1})$.
	\item [$(\S2)$\;\;\,] $\reg\,V$ is stable in $B^{n+1}_2(0)$ in the following (usual) sense: for each open ball $\Omega\subset B^{n+1}_2(0)$ with $\sing\,V\cap \Omega = \emptyset$ in the case $2\leq n\leq 6$ or $\H^{n-7+\gamma}(\sing\,V\cap \Omega) = 0$ for every $\gamma>0$ in the case $n\geq 7$, given any vector field $\psi\in C^1_c(\Omega\setminus \sing\,V;\R^{n+1})$ with $\psi(X)\perp T_X\reg\,V$ for each $X\in \reg\,V\cap \Omega$, we have
	$$\left.\frac{\ext^2}{\ext t^2}\right|_{t=0}\|\phi(t,\cdot)_\#V\|(B^{n+1}_2(0))\geq 0,$$
	where $\phi(t,\cdot)$, $t\in (-\epsilon,\epsilon)$, are the $C^2$ diffeomorphisms of $B^{n+1}_2(0)$ associated with $\psi$, described in $(\S1)$ above; equivalently (see \cite[Section 9]{simon1983lectures})\footnote{This equivalence requires two-sidedness of ${\rm reg} \, V$, which holds in a ball $\Omega$ as above in view of the smallness assumption on the singular set in $\Omega$ (see e.g.\ \cite{samelson}, where the proof assumes absence of singularities but carries over to the case of a small singular set).} for every such $\Omega$ we have
	$$\int_{\reg\,V\cap \Omega}|A|^2\zeta^2\ \ext\H^n \leq \int_{\reg\,V\cap \Omega}|\nabla\zeta|^2\ \ext\H^n\ \ \ \ \text{for all }\zeta\in C^1_c(\reg\,V\cap\Omega),$$
	where $A$ denotes the second fundamental form of $\reg\,V$, $|A|$ the length of $A$, and $\nabla$ the gradient operator on $\reg\,V$.
	\item[$(\S3)_Q$] No singular point of $V$ with density $< Q$ is a classical singularity of $V$ (see Definition~\ref{classical-sing}). 
\end{enumerate}

Note that $(\S3)_Q$ replaces \cite[Section~3, Condition $(\S3)$]{wickramasekera2014general}  (which rules out \emph{all} classical singularities), and it only rules out classical singularities with density $\in \{\frac{3}{2},2,\dotsc,Q-\frac{1}{2}\}$.

\begin{thmx}\label{thm:A}
	Let $Q\in \Z_{\geq 2}$. There is a number $\epsilon = \epsilon (n, Q) \in (0, 1)$ such that if $V \in \S_{Q}$, 
	$(\w_n 2^n)^{-1}\|V\|(B^{n+1}_2(0))<Q+1/2$, $Q-1/2 \leq \w_n^{-1}\|V\|(\R\times B^n_1(0))<Q+1/2$ and $\int_{\R \times B_{1}^{n}(0)} |x^{1}|^{2} \, \ext\|V\| < \epsilon$, then we have the following: 
there is a generalised-$C^{1,\alpha}$ ($Q$-valued) function $u:B^n_{1/2}(0)\to \A_Q(\R)$ such that: 
\begin{itemize}
\item[(i)] ${\mathcal R}_{u} = \pi \, (\{X \in \R \times B_{1/2}^{n}(0) \, : \, \Theta_{V}(X) < Q\});$ $\B_{u} \, \cup \, 
\CC_{u} = \pi \, (\{X \in \R \times B_{1/2}^{n}(0) \, : \, \Theta_{V}(X) \geq Q\}) = \pi \, (\{X \in \R \times B_{1/2}^{n}(0) \, : \, \Theta_{V}(X) = Q\})$ 
(where the notation is as in 
Definitions~\ref{genC1}, ~\ref{regular-branch-sets}, and ~\ref{genC1alpha}, and $\pi \, : \, \R^{n+1} \to \{0\} \times \R^{n}$ is the orthogonal projection); 
\item[(ii)] $$\sup_{B_{1/2}^n(0)} \, |u|+ \sup_{B_{1/2}^{n}(0) \setminus \CC_{u}} \, |Du| \leq C\left(\int_{\R\times B_1^n(0)}|x^1|^2\ \ext\|V\|(X)\right)^{1/2},$$
where $C = C(n, Q)$, and; 
\item[(iii)] $$V \res (\R \times B_{1/2}^{n}(0))  = ({\rm graph} \, u, \theta).$$
Here, writing $u(X) = \sum_{j=1}^{Q} \llbracket u_{j}(X) \rrbracket$ for $X \in B_{1/2}^{n}(0)$, we use the notation $\graph\,u = \{(u_{j}(X),X) \, : \, X \in B_{1/2}^{n}(0), \; j \in \{1, \ldots, Q\}\},$ and $({\rm graph} \, u, \theta)$ is the $n$-varifold on $\R \times B_{1/2}^{n}(0)$ induced by ${\rm graph} \, u,$ whose multiplicity function $\theta$ at each point $(u_{j}(X),X) \in {\rm graph} \, u$ is given by 
$\theta(u_{j}(X),X) = \# \{k \, : \, u_{k}(X) = u_{j}(X)\}$ for each $j=1, 2, \ldots, Q.$
\end{itemize}	
In particular, every singular point $Y$ of $V \res ({\mathbb R} \times B_{1/2}^{n}(0))$ is either a density $Q$ classical singularity or a density $Q$ branch point, with a unique tangent cone 
$\BC_{Y}$ at $Y$ in either case and with $\pi^{-1}(\pi(Y)) = \{Y\}$.

Moreover, we have that
$$\rho^{-n-2}\int_{{\mathbb R} \times B^{n}_{\rho}(\pi(Y))} {\rm dist}^{2} \, (X, {\rm spt} \, \|\BC_{Y}\|) \, \ext\|V\|(X) \leq C\rho^{2\alpha} \int_{{\mathbb R} \times B_{1}^{n}(0)} |x^{1}|^{2} \ext\|V\|(X) \;\;\;\; \forall \rho \in (0, 1/4]$$
and that
$${\rm dist}_{\mathcal H} \, ({\rm spt} \, \|\BC_{Y_{1}}\| \cap B_{1}^{n+1}(0),  {\rm spt} \, \|\BC_{Y_{2}}\| \cap B_{1}^{n+1}(0)) \leq C|Y_{1} - Y_{2}|^{\alpha} \left(\int_{\R\times B_1^n(0)}|x^1|^2\ \ext\|V\|(X)\right)^{1/2}$$
for any singular points $Y, Y_{1}, Y_{2}$ of $V \res ({\mathbb R} \times B_{1/2}^{n}(0)),$ where $C = C(n, Q) \in (0, \infty)$ and $\alpha = \alpha(n, Q) \in (0, 1)$.  
\end{thmx}

See Definitions~\ref{genC1} and \ref{genC1alpha} below for the definition of what we mean by a $Q$-valued \textit{generalised}-$C^{1,\alpha}$ function, which is a slightly weaker notion than that of a 
$Q$-valued $C^{1, \alpha}$ function. The difference stems from the fact that (for $Q \geq 3$), a density $Q$ classical stationary cone (which consists of half-hyperplanes) need not be the sum of \emph{full} hyperplanes, even when it lies close to a hyperplane (see Figure \ref{fig:0}).

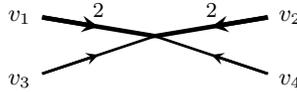
\begin{figure}[h]
\centering
\begin{tikzpicture}
    \draw [line width = 0.6mm] (-1.5,0.25) -- (0,0);
    \draw [-stealth, line width = 0.6mm] (-1.5,0.25) -- (-0.75,0.125);
    \draw [line width = 0.6mm] (1.5,0.25) -- (0,0);
    \draw [-stealth, line width = 0.6mm] (1.5,0.25) -- (0.75,0.125);
    \draw (-1.5,-0.5) -- (0,0);
    \draw [-stealth] (-1.5,-0.5) -- (-0.75,-0.25);
    \draw (1.5,-0.5) -- (0,0);
    \draw [-stealth] (1.5,-0.5) -- (0.75,-0.25);
    \node at (-0.75,0.35) {\SMALL $2$};
    \node at (0.75,0.35) {\SMALL $2$};
    \node at (-1.8,0.25) {\footnotesize $v_1$};
    \node at (1.8,0.25) {\footnotesize $v_2$};
    \node at (-1.8,-0.6) {\footnotesize $v_3$};
    \node at (1.8,-0.6) {\footnotesize $v_4$};
\end{tikzpicture}
\caption{\footnotesize Cross-section of a 1-parameter family of stationary $n$-dimensional classical cones in ${\mathbb R}^{n+1}$ consisting of 6 half-hyperplanes, none of which are the sum of three hyperplanes; the half-hyperplanes with multiplicity 2 are indicated. The unit vectors in the inward pointing directions are given by $v_1 = (\sqrt{1-\epsilon^2},-\epsilon)$, $v_2 = (-\sqrt{1-\epsilon^2},-\epsilon)$, $v_3 = (\sqrt{1-4\epsilon^2},2\epsilon)$, and $v_4 = (-\sqrt{1-4\epsilon^2},2\epsilon)$. In particular, we have as $\epsilon\downarrow 0$ that these cones limit onto a multiplicity 3 hyperplane. The fact that each example does not contain any hyperplane follows readily from the expressions of the position vectors.}
\label{fig:0}
\end{figure}
 
Note that  the assumption requiring absence of classical singularities of density $<Q$ is necessary for the conclusions of Theorem \ref{thm:A}: indeed, in the $Q=3$ case a union of a catenoid and a plane provides a simple example, and in the $Q=2$ case, a truncated catenoid with a disk provides an example; each of these can be made arbitrarily close to a plane by scaling it down. We stress that these examples are stable on their regular parts as defined in $(\mathcal{S}2)$. Figures \ref{fig:1} and \ref{fig:2} below provide illustrations of these.

\begin{figure}[h]
		\centering
		\begin{minipage}{0.45\textwidth}
		\centering
			\begin{tikzpicture}
				\draw[rotate = -90] (-0.3,3) parabola bend (0,0.5) (0.3,3);
				\draw[rotate = 90] (-0.3,3) parabola bend (0,0.5) (0.3,3);
				\draw (-3,0) -- (3,0);
			\end{tikzpicture}
			\caption{\footnotesize Cross-section of a catenoid with a plane which is close to a (multiplicity 3) plane, but not a $3$-valued graph over it.}
			\label{fig:1}
		\end{minipage}\hfill
		\begin{minipage}{0.55\textwidth}
		\centering
			\begin{tikzpicture}
		\draw (0,0) .. controls (-0.25,0.45) and (-1,0.4) .. (-2,0.45);
				\draw (0,0) .. controls (-0.25,-0.45) and (-1,-0.4) .. (-2,-0.45);
				\draw (1,0) .. controls (1.25,0.45) and (2,0.4) .. (3,0.45);
				\draw (1,0) .. controls (1.25,-0.45) and (2,-0.4) .. (3,-0.45);
				\draw (0,0) -- (1,0);
				\draw[dotted] (0,0) -- (-0.5,0.9);
				\draw[dotted] (-2,0) -- (3,0);
				\draw[dotted] (-0.15,0.5) to[bend left = 75] (0.7,0);
				\node at (0.25,0.25) {\SMALL 120$^\circ$};
	\end{tikzpicture}
	\caption{\footnotesize Cross-section of a modified catenoid which has stable regular part, yet is not expressible as a two-valued graph over the horizontal plane.}
	\label{fig:2}
		\end{minipage}
\end{figure}

{\bf Remark (Structure of the branch set):} Let $V \in \S_{Q}$. If $Q=2$, the countable $(n-2)$-rectifiability of the set $\B_{V}$ of density $Q$ branch points of $V$ follows from Theorem~\ref{thm:A} and the work  \cite{krummel2021fine}; see Theorem~\ref{thm:C} in Section~\ref{other} below. For general $Q$, 
Theorem~\ref{thm:A} does not immediately tell us anything about the size or the nature of $\B_{V}$, but it reduces the study of $\B_{V}$ to a PDE question.  
In view of the local graph description and the estimates provided by Theorem~\ref{thm:A}, including decay towards a unique tangent plane at branch points, perseverance with an approach similar to that seen in \cite{krummel2021fine} and \cite{krummel2017fine} is now likely to prove fruitful for the analysis of $\B_{V}$ for general $Q$. Indeed, the analysis in \cite{krummel2021fine} uses monotonicity of a frequency function (established in \cite{simon2016frequency}) for the height of the varifold (given by a two-valued $C^{1, \alpha}$ graph) relative to the average of the (two) sheets; similarly, the rectifiability and uniqueness of  blow-up results proved in \cite{krummel2017fine} for the branch set of Dirichlet energy minimising $Q$-valued functions, for arbitrary $Q,$ uses the frequency function for the height of the graph of the function relative to the average of the sheets. For the analysis of branch points of varifolds in $\S_{Q}$ for general $Q$, a natural way forward is to use (as a replacement for literally taking the average of the sheets) a centre manifold, as first introduced in Almgren's work \cite{almgrenalmgren} and for which a streamlined construction is given in \cite{DelSpa}, and then follow Almgren's argument to establish frequency monotonicity for the height relative to the centre manifold, and from there on proceed as in \cite{krummel2021fine}, \cite{krummel2017fine}. In the absence of an area minimising hypothesis as in Almgren's work, the key to such a centre manifold construction and a frequency monotonicity argument is provided by Theorem~\ref{thm:A} which gives a graph structure with uniform decay estimates. We also note that a key step in Almgren's proof of the dimension bound for the branch set of area minimising integral currents is establishing the corresponding dimension bound for the branch set of Dirichlet energy minimising multi-valued functions. Such functions provide the appropriate linear theory when one performs a blow-up in the area minimising setting. In the present setting, the blow-up class (discussed in Section \ref{initial-cb}) does not, of course, satisfy an energy minimising property. In Appendix \ref{app:A} we will show that the branch set for functions in our blow-up class has codimension at least 2, providing a first step in the 
program outlined above to further analyse ${\mathcal B}_{V}$ for general $Q$.

\subsection{A brief discussion of the proof of Theorem~A} Our proof of Theorem~\ref{thm:A} is based on a blow-up (or linearisation) argument, the idea of which in its most basic (namely, multiplicity 1) setting is contained in the works of De~Giorgi (\cite{DG}) and Allard (\cite{allard1972first}). The essence of the method in that setting is to approximate the varifold, when it is close to a multiplicity 1 plane, by the graph of a (single-valued) harmonic function, and use standard decay estimates for harmonic functions to obtain a decay estimate for the varifold. Iteration of this estimate ultimately leads to the conclusion that the varifold is an embedded graph in the interior. 

In \cite{wickramasekera2014general} a higher multiplicity version of this method is carried out, where it is shown that if the (codimension 1) varifold is stable and has no classical singularities, and is close to a multiplicity $Q$ hyperplane, then either the singular set must be lower dimensional (and hence by \cite{schoen1981regularity} the varifold must already be embedded as $Q$ ordered graphs), or the support of the varifold is well approximated by the graph of a single-valued harmonic function. Iterating this information leads to the embeddedness conclusion again.  

In the present setting, the embeddedness conclusion is false. A pair of intersecting hyperplanes illustrates this, and in fact there are more elaborate examples (\cite{simon2007stable}, \cite{krummel2019existence}) with both branch points and classical singularities that can be made (by scaling around a branch point) arbitrarily close to a hyperplane. 
The present setting requires approximation of the varifold by the graph of a function that belongs to an appropriate class of $Q$-valued functions. These $Q$-valued functions are obtained as ``vertical'' scaling limits (\emph{coarse blow-ups}) of sequences of varifolds in $\S_{Q}$ converging to a multiplicity $Q$ flat disk. Since the varifolds are now allowed to have classical singularities (of density $Q$),  unlike in \cite{wickramasekera2014general} the coarse blow-ups will not in general be separate single-valued harmonic functions. Nonetheless, they are, ultimately, shown to satisfy a certain uniform generalised-$C^{1, \alpha}$ decay estimate (Theorem~\ref{coarse_reg}). 

A key step in our proof of this estimate is establishing the monotonicity of the frequency function associated with coarse blow-ups.
Almgren introduced the frequency function in his monumental work (\cite{almgrenalmgren}) to study, among other things,  coarse blow-ups of area minimising currents (of codimension $>1$) converging to a plane. In that setting the convergence of energy in the blow-up process is shown to hold, with coarse blow-ups shown to be Dirichlet energy minimising and hence stationary for Dirichlet energy; consequently 
they satisfy two variational identities, termed in \cite{almgrenalmgren} as ``squash'' and ``squeeze'' identities, which directly lead to frequency monotonicity.  
In a different setting, namely in \cite{simon2016frequency}, it is shown that $C^{1, \alpha}$ two-valued harmonic functions, which need not be energy minimising, similarly satisfy frequency monotonicity, by an argument that depends on their $W^{2, 2}$ regularity. 

In contrast to either of these settings, 
in the present context the monotonicity of frequency has to be established having at our disposal neither the knowledge of a variational 
principle satisfied a priori by the coarse blow-ups, nor $W^{2, 2}$ regularity (unless $Q=2$), nor any control on the size of their branch set.
There does not seem to be any principle that would guarantee convergence of energy in the blow-up process; the standard first variation estimates only guarantee local weak $W^{1, 2}$ convergence, and the stability inequality does not provide, a priori, any further strengthening of this convergence, not least since stability is assumed only on the regular set of the varifolds which a priori could be very small in measure. This lack of convergence of energy means that in place of the squash identity, we get the weaker squash inequality (Lemma~\ref{squash}); this however, together with the squeeze identity (discussed next), suffices for the purposes of frequency monotonicity. The failure of $W^{2, 2}$ regularity of the coarse blow-ups is illustrated (in the case $Q=3$) by the possibility that (the graph of) a coarse blow-up may consist of six half-hyperplanes joining up along a common axis to form a picture where three of the half-hyperplanes lie on one half-space of $\R^{n+1}$ and three on the other, without them forming three full hyperplanes (such as in Figure \ref{fig:0}).  

As for the proof of the squeeze identity (Lemma~\ref{squeeze} below), our approach is to couple it with generalised-$C^{1}$ regularity of the coarse blow-ups, and establish both facts simultaneously. We first use a combination of variational and non-variational arguments to establish that generalised-$C^1$ regular coarse blow-ups $v$ satisfy the squeeze identity. 
To do this, we use: (i) (an elementary version of) a certain energy non-concentration estimate established in \cite{BKW2021} (Lemma~\ref{noncon}); this is used to show that, for the average-free part of $v$, the contribution to the integral on the left hand side of the squeeze identity  from a small neighbourhood of the (potentially large) ``branch set'' $\B_{v}$ of $v$ is negligible; (ii) a regularity result (Theorem~\ref{thm:fine_reg}) together with a first variational argument for the varifolds, to establish the squeeze identity locally near the set $\CC_{v}$ of classical singularities of $v$; and (iii) stability of the varifolds,  
which implies (via an application of another regularity result---Theorem~\ref{sheeting}) that the coarse blow-ups are classical harmonic functions locally away from $\B_{v} \cup \CC_{v}$. We  then combine these results with the help of a partition of unity to establish the squeeze identity everywhere (i.e.\ for arbitrary compactly supported test functions) for generalised-$C^{1}$ coarse blow-ups. This then implies frequency monotonicity for generalised-$C^{1}$ coarse blow-ups, which we ultimately use to prove that 
general (i.e.\ $W^{1, 2}$) coarse blow-ups are in fact of class generalised-$C^{1, \alpha}$ for some fixed $\alpha \in (0, 1)$, and satisfy a uniform decay estimate. 

The key purpose of the frequency function, as far as Theorem~\ref{thm:A} is concerned, is not to control the size of the singular set of the coarse blow-ups (which is its main purpose in Almgren's work \cite{almgrenalmgren}) but to classify ($Q$-valued, $W^{1, 2}_{\rm loc}$) homogeneous degree 1 coarse blow-ups as a step in our proof of their generalised-$C^{1, \alpha}$ regularity; this step can be thought of as establishing 
an integrability condition (see e.g.\ \cite{allard-almgren}, \cite{simon1993cylindrical} where this condition is used in a multiplicity 1 setting, 
and \cite{krummel2017fine} where it is shown to be satisfied at a.e.\ branch point for Almgren's multi-valued Dirichlet energy minimisers). Once the generalised-$C^{1, \alpha}$ regularity for the coarse blow-ups is established, it is possible to return to the frequency function to obtain further information about the branch set of the coarse blow-ups. For instance, by considering the average-free part and arguing in a way similar to \cite{almgrenalmgren}, we get that the Hausdorff dimension of the branch set of a blow-up is $\leq n-2$; the principal difference in our setting is in proving strong $W^{1,2}$ convergence for rescalings of the average-free part, which in \cite{almgrenalmgren} follows from the energy minimising property, but in our setting is achieved by means of continuity and energy non-concentration estimates that we establish for the average-free part (see Appendix \ref{app:A}). 
We note that this dimension bound however is not necessary for the proof of Theorem~\ref{thm:A}. 

\subsection{Notation and definitions}\label{notation}
The following notation will be used throughout the paper:
\begin{itemize}
\item $n$ is a fixed positive integer $\geq 2,$  ${\R}^{n+1}$ denotes the $(n+1)$-dimensional Euclidean space and 
$(x^{1}, x^{2}, y^{1}, y^{2}, \ldots, y^{n-1}),$ which we shall sometimes abbreviate  as 
$(x^{1}, x^{2}, y),$ denotes a general point in ${\R}^{n+1}$. We shall identify 
${\R}^{n}$ with the hyperplane $\{x^{1} = 0\}$ of ${\R}^{n+1}$ and ${\R}^{n-1}$ with the subspace 
$\{x^{1} = x^{2} = 0\}.$

\item For $Y \in {\R}^{n+1}$ and $\rho >0$, $B_{\rho}^{n+1}(Y) := \{X \in {\R}^{n+1} \, : \, |X - Y| < \rho\}.$

\item For $Y \in {\R}^{n}$ and $\rho >0$, $B_{\rho}(Y) := \{X \in {\R}^{n} \, : \, |X -Y| < \rho\}.$ We shall often abbreviate $B_{\rho}(0)$ as $B_{\rho}.$

\item For $Y \in {\R}^{n+1}$ and $\rho>0$, $\eta_{Y, \rho} \, : \, {\R}^{n+1} \to {\R}^{n+1}$ is the map defined by 
$\eta_{Y, \rho}(X) := \rho^{-1}(X - Y)$ and $\eta_{\rho}$ abbreviates $\eta_{0, \rho}$.

\item ${\mathcal H}^{k}$ denotes the $k$-dimensional Hausdorff measure in ${\R}^{n+1}$, and 
$\omega_{n} = {\mathcal H}^{n} \, (B_{1}(0)).$

\item For $A, B \subset {\R}^{n+1}$, ${\rm dist}_{\mathcal H} \, (A, B)$ denotes the Hausdorff distance between $A$ and $B$.  

\item For $X \in {\R}^{n+1}$ and $A \subset {\R}^{n+1}$, ${\rm dist} \, (X, A) := \inf_{Y \in A} \, |X - Y|.$

\item For $A \subset {\R}^{n+1}$, $\overline{A}$ denotes the closure of $A$.

\item $G_{n}$ denotes the space of $n$-dimensional subspaces of ${\R}^{n+1}.$ 

\item ${\mathcal A}_{Q}(\R) := \left\{\sum_{j=1}^{Q} 
\llbracket a_{j} \rrbracket \, : \, a_{j} \in \R \text{ for each $j=1, 2,\ldots, Q$} \right\}$,  the set of  unordered $Q$-tuples of points 
$a_{1}, \ldots, a_{Q} \in \R$ identified with Dirac masses $\llbracket a_{j} \rrbracket$ at $a_{j} \in \R$.

\item ${\mathcal G}$ denotes the metric on ${\mathcal A}_{Q}(\R)$ defined, for $a= \sum_{j=1}^{Q} \llbracket a_{j}\rrbracket , b= \sum_{j=1}^{Q} \llbracket b_{j} \rrbracket \in {\mathcal A}_{Q}(\R)$, by  
$${\mathcal G}\left(a, b\right) := \inf_{\sigma} \sqrt{\sum_{j=1}^{Q} |a_{j} - b_{\sigma(j)}|^{2}}$$ 
where the $\inf$ is taken over all permutations $\sigma$ of $\{1, 2, \ldots, Q\}.$ When $b = Q\llbracket 0 \rrbracket$, we simply write 
$|a| := {\mathcal G}(a, Q\llbracket 0 \rrbracket).$ 

\item For $A \subset \R^{n}$ and a function $f  \, : \, A \to  \A_{Q}(\R)$, we write
$f(x) = \sum_{j = 1}^{Q} \llbracket f^{j}(x) \rrbracket$ where we shall always 
choose indices so that $f^{1}(x) \leq \cdots \leq f^{Q}(x)$ for each $x \in A$.

\item  For $A \subset \R^{n}$ and a function $f  \, : \, A \to  \A_{Q}(\R)$, we let 
$${\rm graph} \, f = \{(f^{j}(x), x) \, \; \, x \in A, \; j \in \{1, \ldots, Q\}\}$$ and note that   
$${\rm graph} \, f  = \bigcup_{j=1}^{Q} {\rm graph} \, f^{j}.$$ 
Thus in particular ${\rm graph} \, f \subset \R \times A$. 
\end{itemize}

\noindent
{\bf Remark:} Note that the distance with respect to the metric ${\mathcal G}$ is equal to the ``ordered distance'', that is to say,  
$${\mathcal G}\left(\sum_{j=1}^{Q} \llbracket a_{j}\rrbracket , \sum_{j=1}^{Q} \llbracket b_{j} \rrbracket\right) = 
\sqrt{\sum_{j=1}^{Q} |a_{j} - b_{j}|^{2}}$$ 
if the indices are chosen such that $a_{1} \leq a_{2} \leq \cdots \leq a_{Q}$ and $b_{1} \leq b_{2} \leq \cdots \leq b_{Q}.$ 
Thus we can isometrically embed ${\mathcal A}_{Q}(\R) \hookrightarrow \R^{Q}$ via the map 
$ \sum_{j=1}^{Q} \llbracket a_{j} \rrbracket \mapsto (a_{1}, a_{2}, \ldots, a_{Q}),$ where the indices for points $a = \sum_{j=1}^{Q} \llbracket a_{j} \rrbracket \in {\mathcal A}_{Q}(\R)$ are so that $a_{1} \leq a_{2} \leq \cdots \leq a_{Q}$. We shall often use this embedding in subsequent parts of the paper without further comment. 

For an $n$-varifold $V$ on an open subset $\Omega$ of ${\R}^{n+1}$ (\cite{allard1972first}; see also \cite[Chapter 8]{simon1983lectures}), an open subset $\widetilde{\Omega}$ of $\Omega$, a Lipschitz mapping $f \, : \, \Omega \to {\R}^{n+1},$ and a countably $n$-rectifiable subset $M$ of $\Omega$ with locally finite ${\mathcal H}^{n}$-measure, we use the following notation:

\begin{itemize}
\item $V \, \res \,\widetilde{\Omega}$ abbreviates the restriction $V \, \res \, (\widetilde{\Omega} \times G_{n})$ of $V$ to $\widetilde{\Omega} \times G_{n}.$

\item $\|V\|$ denotes the weight measure on $\Omega$ associated with $V$. 

\item ${\rm spt} \, \|V\|$ denotes the support of $\|V\|$.

\item $\Theta_V(X)$ denotes the density of $V$ at $X$.

\item $f_{\#} \, V$ denotes the image varifold under the mapping $f.$ 

\item For $Z \in {\rm spt} \, \|V\| \cap \Omega$, ${\rm VarTan} \, (V, Z)$ denotes the set of tangent cones to $V$ at $Z$.
 
\item ${\rm reg} \, V$ denotes the (interior) regular part of 
${\rm spt} \, \|V\|$. Thus, $X \in {\rm reg} \,V$ if and only if $X \in {\rm spt} \, \|V\| \cap \Omega$ and there exists $\rho >0$ such that $\overline{B^{n+1}_{\rho}}(X) \cap {\rm spt} \, \|V\|$ is a smooth, compact, connected, embedded hypersurface-with-boundary, with its boundary contained in $\partial B^{n+1}_{\rho}(X).$

\item ${\rm sing} \, V$ denotes the interior singular set of ${\rm spt} \, \|V\|.$ Thus, ${\rm sing} \, V = ({\rm spt} \, \|V\| \setminus {\rm reg} \, V) \cap \Omega.$

\item  For $\theta \, : \, M \to \R$ a positive, locally $\H^{n}$-integrable function, $(M, \theta)$ denotes the varifold $V$ on $\Omega$ defined by $V(\varphi) := \int_{M} \varphi(X, T_{X} M) \, \theta(X) \ \ext{\H}^{n}(X)$ for all $\varphi \in C_{c}(\Omega \times G_{n});$ we call $\theta$ the \emph{multiplicity function} of  $V$.

\item $|M|$ denotes the multiplicity 1 varifold on $\Omega$ associated with $M$, i.e.\ $|M| = (M, 1)$.
\item For an open set $U \subset \R^{n}$ and $g  \, : \, U  \to  \A_{Q}(\R)$ such that $\graph \, g$ is a countably $n$-rectifiable subset (of $\R \times U$) with locally finite $\H^{n}$-measure, $${\bf v}(g) = (\graph \, g, \theta)$$ denotes the (integral) $n$-varifold  on $\R \times U$ whose multiplicity function 
$\theta \, : \, \graph \, g \to \N$ is given by $\theta(g^{\alpha}(x),x) = \# \{\beta \, : \, g^{\beta}(x) = g^{\alpha}(x)\}$ for $\alpha = 1, 2, \ldots, Q$ and $x\in U$. 
\end{itemize}

We shall use the following terminology to describe certain very specific types of cones and singularities associated with varifolds: 

\begin{defn}[Classical cones and classical singularities]\label{classical-sing}\hfill\\
\vspace{-.25in}
\begin{itemize}
\item A \emph{classical cone} in ${\mathbb R}^{n+1}$ is an integral varifold $\BC$ of the type $\BC = \sum_{j=1}^{N} q_{j}|H_{j}|$, where $N$ is an integer $\geq 3$, 
$H_{1}, \ldots, H_{N}$ are half-hyperplanes with a common boundary (the \textit{spine}) $L = \partial H_{j}$ for all $j=1, 2, \ldots, N$, and $q_{j}$ are positive integers; we let $S(\BC)$ denote the  spine $L$ of $\BC$. 

\item Let $T \subset \R^{n+1}$. A \emph{$C^{1}$ classical singularity} of $T$ is a point $y \in T$ such that for some $\rho >0$ and some integer $N \geq 3$,  
$$T \cap B_{\rho}^{n+1}(y) = \bigcup_{j=1}^{N} M_{j},$$ 
where $M_{j} \subset B_{\rho}^{n+1}(y)$ are embedded $C^{1}$ submanifolds-with-boundary having the same $(n-1)$-dimensional $C^{1}$ boundary $L = \partial  M_{j} \text{ for each } j=1, 2, \ldots, N$,  with $y \in L$, $M_{i} \cap M_{j} = L$ for $i \neq j$, and with $M_{i}$ and $M_{j}$ intersecting transversely at every point of $L$ for at least one pair of indices $i$, $j$. 

For $\alpha \in (0, 1)$, a point $y \in T$ is a \emph{$C^{1, \alpha}$ classical singularity} if $y$ satisfies the requirements of the definition (of $C^{1}$ classical singularity) above with $C^{1, \alpha}$ in place of $C^{1}$. 

We shall say a point $y \in T$ is a \emph{classical singularity} of $T$ if $y$ is a $C^{1, \alpha}$ classical singularity of $T$ for some $\alpha \in (0, 1)$.  

\item If $V$ is an $n$-varifold on $\R^{n+1}$, a point $y \in {\rm spt} \, \|V\|$ is a \emph{$C^{1}$ classical singularity} (respectively \emph{$C^{1, \alpha}$ classical singularity}, \emph{classical singularity}) of $V$ if it is a $C^{1}$ classical singularity (respectively $C^{1, \alpha}$ classical singularity, classical singularity) of ${\rm spt} \, \|V\|$. 
\end{itemize}
\end{defn} 
 
We next set up some notation needed to discuss singularities of varifolds satisfying the hypotheses of Theorem~\ref{thm:A}, and also singularities of ``coarse blow-ups'' obtained from sequences of such varifolds converging to $Q|\{0\} \times {\mathbb R}^{n}|$ in ${\mathbb R} \times B_{1}^{n}(0)$. 
\begin{itemize}
\item We write $\CC_Q$ for the set of classical cones of vertex density $Q$ of the following special type: each cone $\BC\in \CC_Q$ can be written as a rotation (in ${\mathbb R}^{n+1}$) of a cone of the form
$$\widetilde{\BC} = \sum^Q_{i=1}|\graph(h_i)| + \sum^Q_{i=1}|\graph(g_i)|,$$
where $h_i:\R^n_+  \to \R$, $g_i:\R^n_-  \to \R$ are of the form $h_i(x^2,\dotsc,x^{n+1}) = \lambda_i x^2$ and $g_i(x^2,\dotsc,x^{n+1}) = \mu_i x^2$ and such that $\spt\|\widetilde{\BC}\|$ is not a single hyperplane; here $\R^n_+ := \{(x^2,\dotsc,x^{n+1})\in \R^n: x^2>0\}$, $\R^n_-:= \{(x^2,\dotsc,x^{n+1})\in \R^n: x^2<0\}$. 

\item If $\psi \, : \, \R^{n} \to {\mathcal A}_{Q}(\R)$ is such that ${\bf v}(\psi) \in \CC_{Q}$, we let $S(\psi) = S({\bf v}(\psi))$, i.e.\ $S(\psi)$ is the spine of ${\bf v}(\psi);$ 
\item For $V\in \S_Q,$ we denote by $\CC_V$ the set of singular points $Y$ where one tangent cone belongs to $\CC_Q$ (which from Theorem \ref{thm:B} we know is the unique tangent cone, and moreover that this cone determines the local structure of $V$ near $Y$ as described by Theorem~\ref{thm:B}).

\item For $V\in \S_Q,$ we denote by $\B_V$ the set of singular points where at least one tangent cone is supported on a hyperplane (this will necessarily be a multiplicity $\geq Q$ hyperplane, by Theorem~\ref{sheeting}). 

\item For $A \subset \R^{n}$ and a function $f \, : \, A \to {\mathcal A}_{Q}(\R)$, we denote by $\CC_f$ the set of points $x \in A$ such that 
$f^{1}(x) = f^{2}(x) = \cdots = f^{Q}(x)$ and the point $(f^{1}(x),x)$ $(= (f^{2}(x),x) = \cdots = (f^{Q}(x),x))$ is a $C^{1}$ classical singularity of 
$\graph \, f$. 
\end{itemize}

\noindent
{\bf Remark:} Let $A \subset \R^{n}$ and $f \, : \, A \to {\mathcal A}_{Q}(\R)$. Then $\CC_{f}$ is the set of points $y \in A$ for which there is $\rho_{y} >0$ with $B_{\rho_{y}}(y) \subset A$, an $(n-1)$-dimensional subspace $L_{y} \subset {\mathbb R}^{n}$ and 
a $C^{1}$ function $\gamma_{y} \, : \, y + L_{y}  \to L_{y}^{\perp} \subset \R^{n}$  with $\gamma_{y}(y) = 0$
such that, letting $\Omega_{y}^{+}$, $\Omega_{y}^{-}$ denote the two connected components of $B_{\rho_{y}}^{n}(y) \setminus \graph \, \gamma_{y}$ (where ${\rm graph} \, \gamma_{y} = \{x + \gamma_{y}(x) \, : \, x \in y + L_{y}\}$),  the following holds: 
\begin{enumerate}
\item[(a)] $\left.f^{1}\right|_{\overline{\Omega_{y}^{+}}}, \ldots, \left.f^{Q}\right|_{\overline{\Omega_{y}^{+}}} \in C^{1}(\overline{\Omega_{y}^{+}})$ and $\left.f^{1}\right|_{\overline{\Omega_{y}^{-}}}, \ldots,\left.f^{Q}\right|_{\overline{\Omega_{y}^{-}}} \in C^{1}(\overline{\Omega_{y}^{-}});$
\item[(b)] $\left. f^{i}\right|_{B^{n}_{\rho_{y}}(y)\cap \partial \Omega_{y}^{+}} = \left. f^{j}\right|_{B^{n}_{\rho_{y}}(y)\cap\partial \Omega_{y}^{-}} \;\; \forall i, j \in \{1, \ldots, Q\}$;
\item[(c)] either $\left.Df^{1}\right|_{\overline{\Omega_{y}^{+}}}(z) \neq \left.Df^{Q}\right|_{\overline{\Omega_{y}^{+}}}(z) \;\; \forall z \in \partial \Omega_{y}^{+}$ or $\left.Df^{1}\right|_{\overline{\Omega_{y}^{-}}}(z) \neq \left.Df^{Q}\right|_{\overline{\Omega_{y}^{-}}}(z) \;\; \forall z \in \partial \Omega_{y}^{-}$ (or both hold).
\end{enumerate}

We next define the notions of \textit{generalised}-$C^1$ and \textit{generalised}-$C^{1,\alpha}$ regularity for an ${\mathcal A}_{Q}(\R)$-valued  function $f$.

\begin{defn} (Generalised-$C^{1}$). \label{genC1}
Let $U\subset {\mathbb R}^{n}$ be an open set. We say that a function $f\, :\, U\to \A_Q(\R)$ belongs to $GC^1(U)$, or equivalently, $f$ is  of class \textit{generalised}-$C^1$ in $U$, if: 
\begin{itemize}
\item[(i)] $f$ is differentiable (as a function on $U$) at every point $y \in U \setminus \CC_{f}$ in the classical sense, i.e.\ taking $f$ to be an $\R^{Q}$-valued function $x \mapsto (f^{1}(x), f^{2}(x), \ldots, f^{Q}(x))$ (with $f^{1} \leq f^{2} \leq \cdots \leq f^{Q}$, as always), and; 
\item[(ii)] the derivative $Df$ is continuous on $U \setminus \CC_{f}.$
\end{itemize}
\end{defn}

\begin{defn} (Regular and branch sets). \label{regular-branch-sets}
Let $U\subset {\mathbb R}^{n}$ be an open set, and let $f\, :\, U\to \A_Q(\R)$ be a generalised-$C^{1}$ function. 

We denote by ${\mathcal R}_{f}$ the set of points $y \in U$ for which there is $\rho_{y}  \in (0, {\rm dist} \, (y, \partial U))$ such that 
$\left.f^{1} \right|_{B_{\rho_{y}}(y)}, \ldots , \left.f^{Q}\right|_{B_{\rho_{y}}(y)}  \in C^{1}(B_{\rho_{y}}(y); \R).$ We call ${\mathcal R}_{f}$ the \emph{regular set} of $f$. 

We set ${\mathcal B}_{f} := U \setminus ({\mathcal R}_{f} \cup \CC_{f}),$ and call ${\mathcal B}_{f}$ the \emph{branch set} of $f$. 
\end{defn}

\noindent
{\bf Remarks:}  Let $U \subset \R^{n}$ be (non-empty and) open, and let $f \, : \, U \to {\mathcal A}_{Q}(\R)$ be of class $GC^{1}(U)$.
\begin{itemize}
\item[(1)] The sets ${\mathcal R}_{f}$, $\CC_{f}$ and $\B_{f}$ are pairwise disjoint and uniquely determined by $f$; moreover, 
${\mathcal R}_{f}$ is non-empty and open in $U$, $\B_{f}$ is relatively closed in $U$, and $\CC_{f}$ 
is an embedded $(n-1)$-dimensional submanifold of $U \setminus \B_{f}$ with $U\cap \partial \CC_{f} = \B_{f}$; in particular,  we have that 
$\CC_{f} = \emptyset \implies \B_{f} = \emptyset$.

\item[(2)] With the notation as in the remark preceding Definition~\ref{genC1}, if we additionally have for each $y \in \CC_{f}$ that $\left.Df^{j}\right|_{\overline{\Omega_{y}^{+}}}(y) = \left.Df^{Q - j +1}\right|_{\overline{\Omega_{y}^{-}}}(y)$ for all $j = 1, 2, \ldots, Q$, then $\graph\,  f$ is a $C^{1}$ immersion near each $y \in \CC_{f}$, and $f$ is a function in $C^{1}(U; {\mathcal A}_{Q}(\R))$ in the classical sense.

\item[(3)]  It follows from the definition of $\CC_{f}$ and condition (ii) of Definition~\ref{genC1} that the map taking a point $y \in \CC_{f} \, \cup \, \B_{f}$ to ${\rm spt} \, \|\BC_{y}\| \, \cap \, \overline{B_{1}^{n+1}(0)}$, where $\BC_{y}$ is the (unique) varifold tangent cone to ${\bf v}(f)$ at $(y, {\rm spt} \, f(y))$, is continuous (on $\CC_{f} \, \cup \, \B_{f}$) with respect to 
the Hausdorff metric (note that if $f(y) = Q\llbracket z \rrbracket$ for some $z = z(y)$, then $\spt\, f(y) = z$).

\item[(4)] It follows from condition (ii) of Definition~\ref{genC1} that for each compact set $K \subset U$, $$\lim_{\epsilon \to 0^{+}} \, \sup_{y \in K \cap {\mathcal R}_{f} \cap (\B_{f})_{\epsilon}} \, \inf_{z \in \B_{f}} \G(Df(y), Df(z)) \to 0.$$
\end{itemize}

\begin{defn}\label{genC1alpha}
Let $U\subset {\mathbb R}^{n}$ be open and $\alpha \in (0, 1)$. We say that a function $f \, : \, U\to \A_Q(\R)$ belongs to 
$GC^{1, \alpha}(U)$, or equivalently, $f$ is of class 
\textit{generalised}-$C^{1, \alpha}$ in $U,$ if $f \in GC^{1}(U)$ and, with the notation as in Definition~\ref{regular-branch-sets}: (a) for each compact set $K \subset U$, each $y \in {\mathcal R}_{f}$ and for the largest $\rho_{y}$ corresponding to $y$, the functions $f^{j},$ $j=1, 2, \ldots, Q$, are 
in $C^{1, \alpha}(\overline{B}_{\rho_{y}}(y) \cap K)$; (b)  each $y \in \CC_{f}$ is a $C^{1, \alpha}$ classical singularity of ${\rm graph} \, f$; 
(c) the map $Df$  is in 
$C^{0, \alpha}({\mathcal R}_{f} \cup \B_{f}; {\mathcal A}_{Q}({\mathcal M}_{1 \times n})),$ i.e.\ for each compact set 
$K \subset {\mathcal R}_{f} \cup \B_{f}$, 
$\sup_{x_{1}, x_{2} \in K; \, x_{1} \neq x_{2}} \, \frac{\G(Df(x_{1}),Df(x_{2}))}{|x_{1} - x_{2}|^{\alpha}} < \infty$.
\end{defn}

Finally, we will need to work with various different excesses throughout this work, so it is convenient to set up the notation now as follows:

\begin{itemize}
	
\item 	For $V\in \S_Q$ and $P$ a hyperplane of ${\mathbb R}^{n+1}$ (through the origin), we define the \textit{one-sided height excess of $V$ relative to $P$} by:
	$$\hat{E}_{V,P}^2:=\int_{\pi_P^{-1}(P\cap B^{n+1}_1(0))}\dist^2(X,P)\ \ext\|V\|(X),$$
	where $\pi_P:\R^{n+1}\to P$ is the orthogonal projection onto $P$.
	
	We abbreviate $\hat{E}_{V, \{0\} \times \R^{n}}$ as $\hat{E}_{V}$, so $\hat{E}_{V}^{2} = \int_{\R \times B^n_{1}} |x^{1}|^{2} \, 
	\ext\|V\|(X).$  
	
\item 	For $\BC\in\CC_Q$, we define the \textit{one-sided height excess of $V$ relative to $\BC$} by:
	$$E_{V,\BC}^2:= \int_{\R \times B^{n}_1(0)}\dist^2(X,\spt\|\BC\|)\ \ext\|V\|(X).$$

\item 	For $\BC\in \CC_Q$, we define the \textit{two-sided height excess of $V$ relative to $\BC$} by:
	$$Q_{V,\BC}^2= \int_{\R \times B^{n}_1(0)}\dist^2(X,\spt\|\BC\|)\ \ext\|V\|(X)\; +\; \int_{(\R \times B^{n}_{1/2}(0))\setminus \{r_\BC<1/16\}}\dist^2(X,\spt\|V\|)\ \ext\|\BC\|(X),$$
	where $r_\BC(X):= \dist(X,S(\BC))$, and $S(\BC)$ is the \textit{spine} of $\BC$, i.e. $S(\BC):= \{X\in \spt\|\BC\|: \Theta_{\BC}(0) = \Theta_{\BC}(X)\}$.
\end{itemize}

\noindent
{\bf Remark:} Throughout the proof of our main result, Theorem~\ref{thm:A} (starting in Part~\ref{main-thm-proof}), we shall only need to consider excesses $\hat{E}_{V}$ and $Q_{V, \BC}$ for varifolds $V \in \S_{Q}$  and cones $\BC \in \CC_{Q}$ for which $\hat{E}_{V}$ is small and ${\rm spt} \, \|\BC\| \cap (\R \times B_{1})$ is close to $\{0\} \times B^n_{1}$. 

\part{Previously known results for $\S_{Q}$, and applications of new and old results to area minimising hypersurfaces mod $p$}\label{other-results}
\setcounter{section}{2}{}
\setcounter{subsection}{0}
\setcounter{equation}{0}

\subsection{Varifolds in $\S_{Q}$ near a flat disk of multiplicity $< Q$: embeddedness} \label{wic-sheeting} For later purposes we will need the following result from \cite{wickramasekera2014general}, guaranteeing embeddedness of a varifold in $V \in \S_{Q}$ lying close to a flat disk of multiplicity $< Q$. Since by $(\S3)_Q$ and the upper semi-continuity of density there are no classical singularities in a varifold in $\S_{Q}$ when it is sufficiently close to a multiplicity $<Q$ flat disk, this result is immediate from \cite[Theorem 3.3]{wickramasekera2014general} and standard elliptic PDE theory.

\begin{thmx}[Sheeting Theorem (Theorem 3.3, \cite{wickramasekera2014general})]\label{sheeting}
 There exists a number $\epsilon_0 = \epsilon_0(n,Q)\in (0,1)$ such that if $V\in \S_Q$ satisfies:
	\begin{enumerate}
		\item [(i)] $(\w_n2^n)^{-1}\|V\|(B_{2}^{n+1}(0)) <Q-1/4$;
		\item [(ii)] $\int_{\R \times B^n_{1}} |x^{1}|^{2} \, \ext\|V\| <\epsilon_0$;
	\end{enumerate}
	then we have
	$$V\res (\R\times B_{1/2}) = \sum^q_{j=1}|\graph \, u_j|$$
	for some $q\in \{1,2,\dotsc,\lfloor Q-1/2\rfloor\}$, where $\lfloor x\rfloor$ denotes the integer part of $x$, $u_j\in C^{1, \beta}(B^n_{1/2})$ for each $j$, $u_1\leq u_2\leq\cdots\leq  u_q$, and:
	$$|u_j|_{C^{1,\beta}(B^n_{1/2})}\leq C\left(\int_{\R\times B^n_1}|x^1|^2\ \ext\|V\|(X)\right)^{1/2},$$
	where $C = C(n, Q) \in (0, \infty)$ and $\beta = \beta(n, Q) \in (0, 1)$. Furthermore, $u_j$ solves the minimal surface equation weakly on $B_{1/2}$, and hence in fact $u_j\in C^\infty(B^n_{1/2})$ for each $j=1,2,\dotsc,q$. 
	\end{thmx}

\subsection{Structure of varifolds in $\S_{Q}$ near classical cones: Theorem~C}\label{other}
Let $\BC_{0}$ be a stationary classical cone. It is shown in \cite[Section~16]{wickramasekera2014general} that a ``minimum distance theorem'' holds, which says that there is a constant $\epsilon = \epsilon({\mathbf C}_{0}) >0$ such that  a stable codimension 1 integral $n$-varifold $V$  with no classical singularities cannot be $\epsilon$-close to $\BC_{0}$ in the unit ball. The proof given in \cite{wickramasekera2014general} of the minimum distance theorem is by induction on the vertex density $\Theta_{\BC_{0}}(0)$ of $\BC_{0}$ (which must take values in $\{\frac{3}{2}, 2, \frac{5}{2}, 3, \ldots\}$); the argument of the inductive step contains a proof that if $V$ is sufficiently close to $\BC_{0}$ then in the interior it has the structure of a classical singularity (\cite{wickramasekera2014general} provides a direct contradiction to the no classical singularities hypothesis). This argument uses not that $V$ has no classical singularities, but only that it has no classical singularities of density $< \Theta_{\BC_{0}}(0)$ (assumed inductively). Thus for any $Q \in \{\frac{3}{2}, 2, \frac{5}{2}, 3\ldots\}$, the argument readily gives Theorem~\ref{thm:B} below concerning stable varifolds with no classical singularities of density $< Q$ (from which  assertion (a) in Section~\ref{branch-classical} above follows). 

In accordance with \cite[Section~16]{wickramasekera2014general}, in this theorem we shall use the following notation: $\BC_{0} = \sum^N_{i=1}q_i^{(0)} |H_i^{(0)}|$ is a stationary classical cone in $\R^{n+1}$ with density $\Theta_{\BC_{0}}(0) = Q$ and spine 
$L_{\BC_{0}} = \{(0, 0)\} \times \R^{n-1},$ where $q_{i}^{(0)}$ are integers $\geq 1$, $H_{i}^{(0)}$ are distinct half-hyperplanes with $\partial H_{i}^{(0)} = L_{\BC_{0}}$ for each $i= 1, 2, \ldots, N$\footnote{ By virtue of stationarity of $\BC_{0}$ we must have that $q_{i}^{(0)} \leq Q - 1/2$ for 
each $i \in \{1, \ldots, N\}.$}; thus $H_{i}^{(0)} = R_{i}^{(0)} \times \R^{n-1}$ where $R_{i}^{(0)} = \{t{\bf w}_{i}^{(0)} \, : \, t >0\}$ for 
distinct unit vectors  ${\bf w}_{1}^{(0)} , \ldots, {\bf w}_{N}^{(0)}$ in $\R^{2}$.

We let $\sigma_{0} = \max \, \{{\bf w}_{i}^{(0)} \cdot {\bf w}_{k}^{(0)} \, : \, i, k = 1, 2, \ldots, N, \; i\ \neq k\}$ and let $N(H_{i}^{(0)})$  denote the conical neighbourhood of $H_{i}^{(0)}$ defined by 
$$N(H_{i}^{(0)}) = \left\{ (x, y) \in \R^{2} \times \R^{n-1} \, : \, x \cdot {\bf w}_{i}^{(0)} > |x|\sqrt{\frac{1 + \sigma_{0}}{2}}\right\}.$$ 
Denote by $\widetilde{H}_{i}^{(0)}$ the hyperplane containing $H_{i}^{(0)},$ and by $\left(\widetilde{H}_{i}^{(0)}\right)^{\perp}$ the orthogonal complement of $\widetilde{H}_{i}^{(0)}$ in $\R^{n+1}$. 

\begin{thmx}\label{thm:B}
	Let $Q\in \{\frac{3}{2},2,\frac{5}{2},3,\dotsc\}$ and let $\BC_{0}$ be  as above. Let $\alpha \in (0, 1)$. There is a constant $\epsilon = \epsilon({\BC}_{0}, \alpha)$ such that the following holds: If $V \in \S_{Q}$, $\Theta_{V}(0) \geq \Theta_{\BC_{0}}(0)$, 
	$$\int_{B_{1}^{n+1}(0)} {\rm dist}^{2} \, (X, {\rm spt} \, \|\BC_{0}\|) \, \ext\|V\| < \epsilon, \; \mbox{and}$$ 
	$$\|V\|\left((B_{1/2}^{n+1}(0) \setminus \{r(X) < 1/8\}) \cap N(H_{i}^{(0)})\right) \geq \left(q_{i}^{(0)} - \frac{1}{4}\right) {\mathcal H}{^n} \left((B_{1/2}^{n+1}(0) \setminus \{r(X) < 1/8\}) \cap H_{i}^{(0)}\right)$$  
  for each $i \in \{1, \ldots, N\}$, then for each $i\in \{1,\dotsc,N\}$ there is a  function 
  $$\gamma_{i} \in C^{1, \alpha} \left(\overline{L_{\BC_{0}} \cap B_{1/2}^{n+1}(0)};\, \widetilde{H}_{i}^{(0)} \cap \{X \, : \, {\rm dist} \, (X, L_{\BC_{0}}) < 1/16\}\right),$$ 
  and functions $u_{i,j}:\Omega_i \to \left(\widetilde{H}_i^{(0)}\right)^\perp$ for $j=1,\dotsc,q_i^{(0)}$, where $\Omega_{i}$ is the connected component of $\widetilde{H}_{i}^{(0)} \cap B_{1/2}^{n+1}(0) \setminus \{x + \gamma_{i}(x) \, : \, x \in L_{\BC_{0}} \cap B_{1/2}^{n+1}(0)\}$ with  
  $\left(H_{i}^{(0)} \setminus \{X \, : \, {\rm dist} \, (X, L_{\BC_{0}}) < 1/16\}\right)  \, \cap \, B_{1/2}^{n+1}(0) \subset \Omega_{i}$ 
  such that: 
  \begin{itemize}
  \item[(i)]  $u_{i, 1} \cdot \nu_{i} \leq u_{i, 2} \cdot \nu_{i} \leq \cdots \leq u_{i, q_{i}^{(0)}} \cdot \nu_{i}$, where $\nu_{i}$  is a constant unit normal to $\widetilde{H}_{i}^{(0)};$ 
  \item[(ii)]  $u_{i, j} \in C^{1, \alpha}\left(\overline{\Omega}_i\,; \left(\widetilde{H}_i^{(0)}\right)^\perp\right)$ and $u_{i, j} \cdot \nu_{i}$ solve the minimal surface equation on $\Omega_{i};$
  \item[(iii)]  writing $\widetilde{u}_{i,j}(x):= x+u_{i,j}(x)$ for $x \in \overline{\Omega_{i}},$ we have 
  $$ \widetilde{u}_{i_{1}, j_{1}}(x + \gamma_{i_{1}}(x)) = \widetilde{u}_{i_{2}, j_{2}}(x + \gamma_{i_{2}}(x))
  $$
 for all $x \in \overline{L_{\BC_{0}} \cap B_{1/2}^{n+1}(0)}$ and all values of the indices $i_1,i_2 \in \{1, \ldots, N\}$, $j_1 \in \{1, \ldots, q_{i_1}^{(0)}\}$ and $j_2 \in \{1, \ldots, q_{i_2}^{(0)}\};$
\item[(iv)]  $$V\res B^{n+1}_{1/2}(0) = \sum^N_{i=1}\sum^{q_i}_{j=1}|\graph(u_{i,j}) \cap B_{1/2}^{n+1}(0)|; \;\; and$$
 \item[(v)]  $$\{Z \, : \, \Theta_{V}(Z) \geq Q\} \cap B_{1/2}^{n+1}(0) = \{Z \, : \, \Theta_{V}(Z) = Q\} \cap B_{1/2}^{n+1}(0) = \widetilde{u}_{i, k}(\partial \, \Omega_{i} \cap B_{1/2}^{n+1}(0))$$
	for each $i \in \{1, \ldots, N\}$ and $k \in \{1, \ldots, q_{i}^{(0)}\}$; 
	\end{itemize}
	moreover, for each each $i \in \{1,\dotsc,N\}$ and $j \in \{1,\dotsc,q_i^{(0)} \}$, we have
	$$|u_{i,j}|_{1, \alpha;\Omega_i} \leq C\left(\int_{B^{n+1}_{1}(0)}\dist^2(X,\spt\|\BC_{0}\|)\ \ext\|V\|(X)\right)^{1/2}.$$
	Here $C = C(\BC_{0}, \alpha) \in (0, \infty)$.
\end{thmx}

\noindent
{\bf Remark:} Once we have the conclusions of Theorem~\ref{thm:B}, it follows from the higher regularity theory of \cite{krummel2014regularity}  that indeed the boundary segments $\{x + \gamma_{i}(x) \, : \, x \in L _{\BC_{0}} \cap B_{1/4}^{n+1}(0)\}$ and the functions $\left.u_{i, j}\right|_{\overline{\Omega_{i}} \cap B_{1/4}^{n+1}(0)}$ are real analytic, with derivatives of $u_{i, j}$ of any order $k \in {\mathbb N}$  bounded uniformly in $\Omega_{i} \cap B_{1/4}^{n+1}(0)$ by a constant $C= C(k,\BC_0)$ times the height excess 
$\left(\int_{B^{n+1}_{1}(0)}\dist^2(X,\spt\|\BC_{0}\|)\ \ext\|V\|(X)\right)^{1/2}$. 
 
\begin{proof}[Proof of Theorem~\ref{thm:B}]
The theorem directly follows from the results in \cite{wickramasekera2014general}. There it is shown that if $W$ is a stable codimension 1 stationary integral varifold satisfying the condition that $W$ has no classical singularities (referred to as the $\alpha$\textit{-structural hypothesis} in \cite{wickramasekera2014general}), then: $(a)$ ${\rm sing} \, W$ has Hausdorff dimension $\leq n-7$ (\cite[Theorem~3.1]{wickramasekera2014general}); and (b) a ``sheeting property'' holds, i.e.\ if $W$ is close as a varifold to a hyperplane $P$ of some positive integer multiplicity, then in the interior $W$ is the sum of regular (embedded) graphs over $P$ with small curvature (\cite[Theorem~3.3]{wickramasekera2014general}). 

So if  $Q \in \{\frac{3}{2}, 2, \frac{5}{2}, 3, \ldots\},$ $W$ is a stable codimension 1 integral varifold with no classical singularities of density $< Q$, then by $(a)$ we have: (a$^{\prime}$) ${\rm dim}_{\mathcal H} \, ({\rm sing} \, W \res \Omega) \leq n-7$ for any open set $\Omega$ such that $\Theta_{W}(Y) < Q$ for all $Y \in \Omega$, and by (b) we have:  (b$^{\prime}$) if $W$ is close to a hyperplane with multiplicity $< Q$, then $W$ is regular in the interior.  (This is Theorem~\ref{sheeting}.)

To prove the present theorem, now proceed exactly as in \cite[Section~16]{wickramasekera2014general}. The key ingredient needed is the excess improvement result 
\cite[Lemma~16.9]{wickramasekera2014general}. The proof of this lemma carries over without any change to the present setting: simply drop the assumption that $V$ has no classical singularities, take (b$^{\prime}$) in place of the induction hypotheses (H1) therein, and  take the absence of density $< Q$ classical singularities (the present hypothesis) in place of the induction hypothesis (H2). In view of (a$^{\prime}$), the condition that there are ``no significant gaps'' in the set 
$\{Z \, : \, \Theta_{V}(Z) \geq \Theta_{{\BC}_{0}}(0)\} \cap B_{3/4}^{n+1}(0)$, needed for the proof, follows from the same reasoning given in \cite{wickramasekera2014general}. Once this excess decay lemma is established, the proof of \cite[Theorem~16.1]{wickramasekera2014general} carries over without any change to give the present theorem. Note that the asserted bound on $|u_{i, j}|_{1, \alpha; \Omega_{i}}$ follows from the estimates established in the proof of \cite[Theorem~16.1]{wickramasekera2014general} together with the standard global $C^{1, \alpha}$ Schauder estimates for weak solutions to second order uniformly elliptic equations. \end{proof}

{\bf Remarks on the proof of \cite[Lemma~16.9]{wickramasekera2014general}:} 

We digress a little to make a few comments.

\noindent
The proof of the excess decay result \cite[Lemma~16.9]{wickramasekera2014general}, in turn, follows closely the 
blow-up argument in the seminal work of L.~Simon in \cite{simon1993cylindrical},  although the setting in \cite{wickramasekera2014general} is technically one of higher multiplicity (unlike in \cite{simon1993cylindrical} where a multiplicity 1 class hypothesis is made). The two most crucial points to note here are: (i) in \cite[Lemma~16.9]{wickramasekera2014general} (in fact throughout the relevant part of \cite{wickramasekera2014general}, i.e.\ Section 16) a ``non-degenerate'' situation is assumed; that is to say, the varifold $V$ is assumed to be far from being flat at scale 1, by virtue of the assumption that it is close to a given, fixed classical cone $\BC_{0}$ at scale 1; and (ii) we have at our disposal a strong sheeting result (in \cite{wickramasekera2014general} this is the induction hypothesis (H1), and in the present setting this is the sheeting theorem \cite[Theorem~3.3]{wickramasekera2014general}, or more precisely, property (b$^{\prime}$) mentioned in the proof of Theorem~\ref{thm:B} above) which is applicable to $V$ away from the axis of $\BC_{0}$ by virtue of the fact that densities of the hyperplanes making up $\BC_{0}$ are constant integers $\leq \Theta_{\BC_{0}}(0)- 1/2$; this gives that $V$ in that region decomposes as the sum of regular, multiplicity 1 sheets satisfying good estimates, in particular \emph{ruling out branch points} in that region.  Thus many of the modifications necessary to the argument of \cite{simon1993cylindrical}---as documented in \cite[Section~16]{wickramasekera2014general}---are, though somewhat involved, of a technical nature. One place in the argument where a slightly different perspective is taken in \cite{wickramasekera2014general} is to use standard boundary regularity results for (single-valued) harmonic functions to deduce decay estimates for the blow-ups generated by sequences of stable varifolds converging to $\BC_{0}.$ Specifically, it is shown that there is a single $C^{2}$ (in fact smooth) function, defined on the axis of $\BC_{0},$ taking values orthogonal to the axis and satisfying an appropriate estimate, which determines a common boundary that joins together all of the ``sheets'' of the blow-up (defined as graphs over appropriate domains in the various half-hyperplanes making up $\BC_{0}$); the desired $C^{1, \alpha}$ decay estimate for the blow-up follows at once from this and a standard boundary regularity theorem for harmonic functions (see \cite[Theorem~16.7]{wickramasekera2014general}). 

More subtle adaptations of \cite{simon1993cylindrical} to other settings, including to other higher multiplicity settings, have recently been carried out in a number of situations. One such instance directly relevant to the present work is \cite[Sections~10--14]{wickramasekera2014general} where a degenerate situation is considered, in which the base cone $\BC_{0}$ is a higher-multiplicity hyperplane; the analysis done in \cite{wickramasekera2014general} in this case leads to Theorem~\ref{thm:fine_reg} below which plays an essential role in our proof of Theorem~\ref{thm:A}. See also: \cite{minter2021structure}  which studies a situation where branch points of the nearby varifolds do exist away from the 
axis of $\BC_{0}$; \cite{becker2017transverse}  where the ``no significant gaps'' condition  used in \cite{simon1993cylindrical} fails (see the proof of Theorem~\ref{thm:B} above); \cite{colombo2017singular} where the (multiplicity 1) base cone $\BC_{0}$ is singular away from its spine; and  \cite{krummel2013fine}, \cite{krummel2017fine}, \cite{krummel2021fine} where the ``cone'' $\BC_{0}$ is the graph (possibly with multiplicity $>1$) of a multi-valued 
homogeneous harmonic function $\varphi,$ with varying degrees of homogeneity, including degrees of homogeneity $< 1.$ 
 
Returning to Theorem~\ref{thm:B}, note that, just as for Theorem~\ref{thm:A}, the absence of classical singularities of density $<Q$ is a necessary hypothesis in Theorem~\ref{thm:B}. This is illustrated by the simple example in Figure \ref{fig:3}.

\begin{figure}[h]
		\centering
			\begin{tikzpicture}
				\draw (0,0.5) -- (-0.33,0.25) -- (-0.33,-0.25) -- (0,-0.5) -- (0.33,-0.25) -- (0.33,0.25) -- (0,0.5);
				\draw (0,0.5) -- (0,1.25);
				\draw (0,-0.5) -- (0,-1.25);
				\draw (-0.33,0.25) -- (-1,0.75);
				\draw (-0.33,-0.25) -- (-1,-0.75);
				\draw (0.33,-0.25) -- (1,-0.75);
				\draw (0.33,0.25) -- (1,0.75);
			\end{tikzpicture}
			\caption{\footnotesize An example of a multiplicity 1 stationary 1-dimensional varifold in ${\mathbb R}^{2}$ with stable regular part which (by scaling) can be made arbitrarily close to a classical cone of density $3$, but is not a $C^{1,\alpha}$ graph over it in the sense described in the conclusion of Theorem~\ref{thm:B}. All angles are $120^\circ$. Of course by taking the Cartesian product of this 
            with ${\mathbb R}^{n-1}$ we obtain higher dimensional examples.}
			\label{fig:3}
\end{figure}
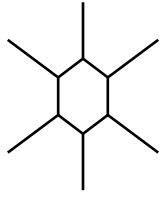

\subsection{The case $Q=2$: Theorem~D}

In the special case of $Q=2$, the conclusions of Theorems \ref{thm:A} and \ref{thm:B} can be strengthened; this is because the only classical singularities which are present are in fact (density 2) \emph{immersed} singular points. This is not true for $Q \geq 3$:

\begin{thmx}\label{thm:C}\footnote{
We do not need the full extent of the techniques employed in the present work to establish Theorem \ref{thm:C}; in particular, the frequency function argument used in the analysis of coarse blow-ups (carried out in Part~\ref{main-thm-proof} below for general $Q$) can be significantly simplified when $Q = 2$, capitalising on the fact that for stationary two-valued graphs we have that generalised-$C^{1} \implies C^{1}$ and consequently the (two-valued) coarse blow-ups are locally in $W^{2, 2}$.}
Suppose $V$ is a stable codimension 1 integral $n$-varifold in $B^{n+1}_2(0)$ such that $V$ has no triple junction singularities (i.e., classical singularities of density $\frac{3}{2}$), i.e.\ $V \in \S_{2}$. Suppose that 
$\BC_{0} = |{P}_{1}^{(0)}| + |P_{2}^{(0)}|$ is the multiplicity 1 cone supported on a distinct pair of hyperplanes $P_{1}^{(0)}$, $P_{2}^{(0)}$. Then:
\begin{itemize}
\item[(i)] there is $\epsilon_{1} = \epsilon_{1}(n) \in (0, 1)$ such that if $(\w_n 2^n)^{-1}\|V\|(B^{n+1}_2(0))<2+1/2$, $2-1/2 \leq \w_n^{-1}\|V\|(\R\times B_1)<2+1/2$, and $\int_{\R \times B_{1}} |x^{1}|^{2} \, \ext\|V\| < \epsilon_1$,
 then the conclusions of Theorem~\ref{thm:A} hold; moreover, every classical singularity of $V \res (\R \times B^n_{1/2})$ is an immersed point of ${\rm spt} \, \|V\|$, and we can take $u$ to be a $C^{1,1/2}$ two-valued function;
\item[(ii)] there is $\epsilon_{2} = \epsilon_{2}(\BC_{0}) \in (0, 1)$ such that if the hypotheses of Theorem~\ref{thm:B} taken with $\epsilon_{2}$ in place of $\epsilon$, $Q = 2,$ $N=2$ and $q_{1}^{(0)} = q_{2}^{(0)} = 1$ are satisfied, then  
$$V\res B^{n+1}_{1/2}(0) = |\graph(u_{1}) \cap B_{1/2}^{n+1}(0)| +  |\graph(u_{2}) \cap B_{1/2}^{n+1}(0)|,$$
where $u_{j} \, : \, P_{j}^{(0)} \cap B_{1/2}^{n+1}(0) \to \left(P_{j}^{(0)}\right)^{\perp}$, $j=1, 2,$ are smooth functions satisfying, for each $k=0, 1, \ldots$, the $C^{k}$ estimate  
$$|u_{j}|_{k, P^{(0)}_{j} \cap B_{1/2}^{n+1}(0)} \leq C \left(\int_{B^{n+1}_{1}(0)}\dist^2(X,\spt\|\BC_{0}\|)\ \ext\|V\|(X)\right)^{1/2},$$ where $C = C(n,k,\BC_0);$ 
\item[(iii)] the branch set of $V$ is countably $(n-2)$-rectifiable and, if non-empty on a ball $B$, then it has positive $(n-2)$-dimensional Hausdorff measure in $B$.
\end{itemize}
\end{thmx}

{\bf Remark:} In conclusion (i), the fact that $u$ is $C^{1,1/2}$ (and not just $C^{1,\alpha}$ for some $\alpha$) follows from \cite{simon2016frequency}, and it is the sharp regularity conclusion. 

\begin{proof}[Proof of Theorem~\ref{thm:C}] For conclusion~(i), note that we already know from Theorem \ref{thm:A} that we can locally express $V$ as a generalised-$C^{1,\alpha}$ two-valued graph. The difference here is that at any singular point $Y$ where there is a classical tangent cone $\BC$, the cone $\BC$ has density 2 at the origin and is stationary, hence it must be the sum of two distinct multiplicity 1 hyperplanes (not just 4 half-hyperplanes meeting along a common axis). In particular, since the function $u$ defining $V$ is generalised-$C^1$ and hence each half hyperplane of $\BC$ defines the one-sided derivative  of $u$ at the point in its domain corresponding to $Y$, the fact that the cone consists of the union of two full hyperplanes means that the two one-sided derivatives of $u$ on each half of a given hyperplane in the cone must agree. Thus every classical singularity of $V$ is an immersed point, and $V \res ({\mathbb R} \times B_{1/2})$ is expressible as a $C^{1,\alpha}$ stationary graph (i.e., generalised-$C^{1,\alpha}$ two-valued codimension one functions are necessarily classically $C^{1,\alpha}$ two-valued functions when their graphs are stationary). The fact that we can take $\alpha = 1/2$ follows from \cite{simon2016frequency}. Conclusion~(ii) follows similarly from Theorem~\ref{thm:B} and standard elliptic regularity results. 
Conclusion~(iii) on the rectifiability of the branch set follows from \cite{krummel2021fine}, with the $(n-2)$-dimension bound first being established in \cite{simon2016frequency}. 
\end{proof}

\subsection{Applications to area minimising hypersurfaces mod $p$: Theorem~E} \label{application} One class of stationary integral varifolds to which Theorem~\ref{thm:A}, Theorem~\ref{thm:B} and Theorem~\ref{sheeting} above can directly be applied is the class of integral varifolds associated with $n$-dimensional rectifiable currents $T$ in $B_{2}^{n+1}(0)$ that are area minimising mod $p$ and have $\partial \, T \res B_{2}^{n+1}(0) = 0 \;\;({\rm mod} \; p)$ (\cite[Section~4.2]{federer}). For any $p\in \Z_{\geq 2}$, such a current $T$ has an associated  integral varifold $V = |T|_{(p)}$ (where $|T|_{(p)} := (\spt\|T\|,\theta)$ for $\theta(x):= |\Theta_{T}(x)\ (\text{mod } p)|$) which is stationary in $U$ and has a stable regular part. Moreover, it is easy to check that if such a current is a classical cone, then the cross-section of it must be a 1-dimensional length minimising rectifiable current mod $p$ with zero mod $p$ boundary. Hence the cross-section must consist of (a finite number of) rays all oriented the same way, i.e.\ inwards towards the vertex or outwards from 
it (for if two rays have opposite orientation, then a standard ``wedge-replacement'' process---where a piece of the wedge formed by those two rays is replaced by a line segment---gives a comparison current still having zero boundary mod $p$ but smaller length), whence the number of rays must be an integer multiple of $p$ and so the density of the classical cone must be $\geq p/2$.  This readily implies that $T$ has no classical singularities of density $ < p/2$. 

Thus we have, for any integer $p \geq 2$: 
\begin{align*}
&\left\{|T|_{(p)} \, : \, T \; \text{is an }n\text{-dimensional area minimising rectifiable}\right.\\
 &\hspace{1in}\left.\; \text{current mod }p\text{ in }B^{n+1}_2(0)\text{ with } \partial T \res B_{2}^{n+1}(0) = 0 \;\;({\rm mod} \; p)\right\} \subset \S_{p/2}.
 \end{align*}

Theorem~\ref{thm:A}, Theorem~\ref{thm:B} and Theorem~\ref{sheeting} can therefore be applied with $Q=p/2$ to deduce regularity results for currents $T$ as above. 
Specifically, by applying Theorem~\ref{thm:B}, we obtain Theorem~\ref{thm:D}(ii) below. As noted in the introduction, this result has very recently been obtained in \cite{de2021area}. 
By applying Theorem~\ref{sheeting}, we obtain Theorem~\ref{thm:D}(iii) below, which says that if $T$ is close to a hyperplane of multiplicity $< p/2$ then it is regular in the interior. This result was originally proved by B.~White (\cite{white1984regularity}) with an argument that makes essential use of the minimising property of $T$ and valid in fact for codimension 1 rectifiable currents mod $p$ minimising general even parametric functionals. 

Since planar tangent cones to a mod $p$ area minimising (representative) current $T$ have constant integer multiplicity $\leq p/2$,  we see from Theorem~\ref{thm:D}(iii) that if $p$ is odd, no tangent cone to $T$ at a singular point can be supported on a plane. The same assertion holds if $p=2$ by the De~Giorgi--Allard regularity theory (\cite{allard1972first}, \cite{simon1983lectures}). If however $p$ is an even integer $\geq 4$, there may exist singular points $y$ where one tangent cone to $T$ is a multiplicity $p/2$ hyperplane; if $p=4$, it follows from a result of White (\cite{white1979structure}) that $T$ corresponds to a smoothly immersed hypersurface near such a point $y$. For general even $p$, Theorem~\ref{thm:A} readily provides regularity of the current in a neighbourhood of such a point as a $\frac{p}{2}$-valued generalised-$C^{1, \alpha}$ graph. This is Theorem~\ref{thm:D}(i).

\begin{thmx}\label{thm:D} Let $p\in \Z_{\geq 2}$  and suppose that $T$ is an $n$-dimensional area minimising rectifiable current mod $p$ in $B_{2}^{n+1}(0)$ 
with $\partial T \res B_{2}^{n+1}(0)= 0 \;\; \mbox{mod} \;\; p$.  
\begin{itemize}
\item[(i)] If $p$ is even, then the conclusions of Theorem~\ref{thm:A} hold with $Q = p/2$ and $V = |T|_{(p)}$ (the varifold associated with $T$) provided the small excess and mass hypotheses of Theorem~\ref{thm:A} are satisfied with $V = |T|_{(p)}$  and with $\epsilon$ equal to $\epsilon(n, p/2),$ the constant given by Theorem~\ref{thm:A}; 
\item[(ii)] \textnormal{(\cite{de2021area})} For arbitrary $p$ (even or odd), given a classical cone $\BC_{0}$ as in Theorem~\ref{thm:B} with $Q = p/2$ and a number $\alpha \in (0, 1)$, the conclusions of Theorem~\ref{thm:B} hold with $Q = p/2$ and $V = |T|_{(p)}$ provided the small excess and mass hypotheses of Theorem~\ref{thm:B} are satisfied with $V = |T|_{(p)}$ and  with $\epsilon$ equal to $\epsilon(\BC_{0}, \alpha),$ the constant given by Theorem~\ref{thm:B};
\item[(iii)] \textnormal{(\cite{white1984regularity})} For arbitrary $p$ (even or odd), the conclusions of Theorem~\ref{sheeting} hold with $V = |T|_{(p)}$  provided the small excess and mass hypotheses of Theorem~\ref{sheeting} are satisfied with $V = |T|_{(p)}$, $Q= \frac{p}{2}$ and 
$\epsilon_{0}$ equal to $\epsilon_{0}(n, \frac{p}{2})$, the constant given by Theorem~\ref{sheeting}.
\end{itemize}
\end{thmx}

\begin{proof} As discussed above, part (i) is an immediate consequence of Theorem~\ref{thm:A}, part (ii) is an immediate consequence of Theorem~\ref{thm:B}, and part (iii) is an immediate consequence of Theorem~\ref{sheeting}.
\end{proof}

{\bf Remark:} It is a well known fact that for sequences of area minimising integral currents, convergence in the integral flat norm ${\mathcal F}$ implies measure-theoretic convergence of the associated varifolds (\cite{simon1983lectures}); this fact extends in a straightforward manner to the mod $p$ minimising setting, with the mod $p$ flat (semi-)norm ${\mathcal F}^{p}$ taking the place of ${\mathcal F}.$  Thus in all parts of Theorem~\ref{thm:D}, we may replace the assumption that $|T|_{(p)}$ has height excess $< \epsilon$ relative to a plane or a classical cone (as a varifold) with the assumption that for appropriate $\epsilon^{\prime} = \epsilon^{\prime}(n, p)$, the current $T$ is $\epsilon^{\prime}$-close to a plane or a classical cone in ${\mathcal F}^{p}.$ 

{\bf Remark:} As mentioned above, if an $n$-dimensional classical cone $\BC$ in ${\mathbb R}^{n+1}$ is the varifold associated with an area minimising rectifiable current mod $p$, then by an elementary construction of a 
1-dimensional comparison current, one quickly gets that the density $\Theta_{\BC}(0) \geq p/2$.  By a similar construction it can in fact be shown that the density $\Theta_{\BC}(0) = p/2$ (see \cite[Proposition 3.5]{de2021area}). Thus (as pointed out in \cite{de2021area}) once we have Theorem~\ref{thm:D}(ii), it in fact provides an asymptotic description of a codimension 1 rectifiable current $T$ that minimises area mod $p$ near any point where one tangent cone to $T$ is a classical cone. We note that, although in general such singularities in $T$ obviously do arise, they can in certain special circumstances be ruled out; for instance, by a theorem of F.~Morgan (\cite{morgan}), if $T$ is such that 
$\partial \, T = B \;\; \mbox{mod} \; p$ for $B$ an extremal compact embedded $(n-1)$-dimensional $C^{1}$ submanifold with at most $p/2$ components, then $T$ has no classical singularities away from $\partial \, T$, and the interior singular set of $T$ is of Hausdorff dimension $\leq n-7$. 

{\bf Remark:} If $T$ is as in Theorem~\ref{thm:D}, $y \in \spt \, T \cap B_{2}^{n+1}(0)$ and $\BC_{y}$ is a tangent cone to $T$ at $y$, then we have the following: 
\begin{itemize}
\item[(a)] by Theorem~\ref{thm:D}(ii) and the preceding remark, if $\BC_{y}$ is a classical cone, then $y$ is a classical singularity; 
\item[(b)] by Theorem~\ref{thm:D}(iii), 
if $\BC_{y} = q\llbracket P \rrbracket$ for some oriented hyperplane $P$ and (integer) $q < p/2$, then $T$ is embedded in a neighbourhood of $y;$ 
\item[(c)] by Theorem~\ref{thm:D}(i), if $\BC_{y}  = \frac{p}{2} \llbracket P \rrbracket$ for some oriented hyperplane $P$ (which can hold only if $p$ is even),  then in a neighbourhood of $y$, the current $T$ is given by a $\frac{p}{2}$-valued function over $P$ of class $GC^{1, \alpha}$. 
\end{itemize}

{\bf Remark:} We have, of course,  that Theorem~\ref{thm:C} (the special case of Theorem~\ref{thm:A} for $Q=2$, giving additional information) is applicable to codimension 1 area minimising currents mod~4 in the same way that Theorem~\ref{thm:A} is applicable to mod $2Q$ minimising currents for general $Q$. However, while density 2 branch points do occur in stable codimension 1 integral varifolds (and so the conclusions of Theorem~\ref{thm:C} are sharp for those), by \cite{white1979structure}, it is known that branch points do not occur in codimension 1 area minimising currents mod 4 and 
in fact such a current is immersed away from a closed set of codimension at least 7.  

Whilst the above results summarise the state-of-the-art for the regularity of codimension one area minimising currents mod $p$, there are several key questions which remain open. Theorem~\ref{thm:D}(i) has since been used by De Lellis--Hirsch--Marchese--Spolaor--Stuvard (\cite{de2022fine}) to establish that the branch set of such an $n$-dimensional current $T$ has Hausdorff dimension at most $n-2$. As far as the authors are aware, there is no known example of such a $T$ which has a \emph{genuine} branch point, namely a branch point where locally $T$ does not decompose as a sum of smoothly embedded minimal hypersurfaces. In light of White's result mentioned in the previous remark, certainly there are no genuine branch points in codimension 1 mod $4$ minimisers. There are two related questions to this: (i) can $T$ have a genuine branch point when every classical singularity in $T$ is assumed to be (smoothly) immersed; and (ii) can $T$ have a (genuine) branch point which is a limit of non-immersed classical singularities? The latter question amounts to asking whether $GC^{1,\alpha}$ multi-valued functions are necessary to describe the local structure near branch points, or whether one can in fact use $C^{1,\alpha}$ multi-valued functions.

\subsection{Varifolds in $\S_{Q}$ with small coarse excess and significantly smaller fine excess}\label{sec:fine-to-coarse}

The following \textit{fine excess $\epsilon$-regularity theorem} will play a key role in our later analysis. 

\begin{theorem}[Fine excess $\epsilon$-regularity theorem]\label{thm:fine_reg}
	Let $Q\geq 2$ be an integer and $\alpha \in (0, 1)$. There exists $\epsilon_{0} = \epsilon_{0}(n,Q, \alpha)\in (0,1)$ and $\gamma_{0} = \gamma_{0}(n,Q, \alpha) \in (0, 1)$ such that the following is true: if for some $\epsilon \in (0, \epsilon_{0}]$ and some 
	$\gamma \in (0, \gamma_{0}],$ a varifold $V\in \S_Q$  and a cone $\BC \in \CC_{Q}$ with $S(\BC) = \{(0, 0)\} \times \R^{n-1}$ satisfy:
	\begin{enumerate}
		\item [(i)] $\Theta_V(0)\geq Q$, $(\w_n2^n)^{-1}\|V\|(B^{n+1}_2(0))<Q+1/2$, $\w_n^{-1}\|V\|(\R\times B_1)<Q+1/2$;
		\item [(ii)] $\hat{E}_{V}^{2} = \int_{\R \times B_{1}} |x^{1}|^{2} \, \ext\|V\|(X) < \epsilon$;
		\item [(iii)] $\hat{E}_{V}^2 < \frac{3}{2}  \inf_{P}\, \int_{\R \times B_{1}} {\rm dist}^{2} \, (X, P) \, \ext\|V\|(X)$, where the infimum is taken over all hyperplanes $P$ with $S(\BC) \subset P$; 
	\item [(iv)] $Q^2_{V,\BC}< \gamma \hat{E}^{2}_{V}$;
	\end{enumerate}
	then there is a cone ${\BC}_{0}\in \CC_Q$ with spine $S({\BC}_{0}) = \{(0, 0)\} \times \R^{n-1}$ and 
	$$\dist_\H(\spt\|{\BC}_{0}\|\cap(\R\times B_1),\spt\|\BC\|\cap (\R\times B_1))\leq CQ_{V,\BC},$$
	and an orthogonal rotation $\Gamma \, : \, \R^{n+1} \to \R^{n+1}$ with $|\Gamma(e_{1}) - e_{1}| \leq CQ_{V, \BC}$ and 
	$|\Gamma(e_{j})  - e_{j}| \leq C {\hat E}_{V}^{-1}Q_{V, \BC}$ for $j=2, 3, \ldots, n+1$, such that ${\BC}_{0}$ is the unique tangent cone to $\Gamma_{\#}^{-1} \, V$ at $0$ and 
	$$\sigma^{-n-2}\int_{\R\times B_\sigma}\dist^2(X,\spt\|{\BC}_{0}\|)\ \ext\|\Gamma_{\#}^{-1} \, V\|\leq C\sigma^{2\alpha}Q_{V,\BC}^2\ \ \ \ \text{for all }\sigma\in (0,1/2).$$
	Furthermore, there is a generalised-$C^{1,\alpha}$ function $u:B_{1/2}\to \A_Q(\R)$ such that:
	\begin{enumerate}
		\item $V\res (\R\times B_{1/2}) = {\bf v}(u);$
		\item  $\B_u\cap B_{1/2} = \emptyset;$ 
		\item ${\rm graph} \, (u) \cap (\R \times \CC_u) = \sing\,V\cap (\R\times B_{1/2});$ moreover, 
		\begin{itemize} 
		\item[(i)] $\sing\,V\cap (\R\times B_{1/2}) = {\rm graph} \, \varphi$ where  $\varphi = (\varphi_{1}, \varphi_{2}) \, : \, B_{1/2}^{n-1}(0) = S({\BC}) \cap (\R \times B_{1/2}) \to \R^{2} \times \{0\}$ is 
		of class $C^{1, \alpha}$ over $\overline{B_{1/2}^{n-1}(0)}$ with $|\varphi_{1}|_{1, \alpha; B_{1/2}^{n-1}(0)} \leq CQ_{V,\BC}$ and 
		$|\varphi_{2}|_{1, \alpha; B_{1/2}^{n-1}(0)} \leq C\hat{E}_{V}^{-1}Q_{V, \BC};$ 
		\item[(ii)] If $\Omega^\pm$ denote the two components of $B_{1/2}\setminus \CC_u = B_{1/2} \setminus {\rm graph} \, \varphi_{2},$ and if we write $u = \sum^Q_{j=1}\llbracket u^{j}\rrbracket$ with $u^{1} \leq u^{2} \leq \cdots \leq u^{Q}$, then $\left.u^{j}\right|_{\Omega^\pm}\in C^{1,\alpha}(\overline{\Omega^\pm})$ and $|u^{j}|_{1,\alpha;\Omega^\pm} \leq C\hat{E}_V$ for each $j \in \{1,\dotsc,Q\};$
		\end{itemize}
		\item If $\BC_Z\in \CC_Q$ denotes the (unique) tangent cone to $V$ at $Z \in \sing\,V\cap (\R\times B_{1/2})$, then $c\hat{E}_V \leq \dist_\H(\spt\|\BC_Z\|\cap (\R \times B_1), \{0\}\times B_1) \leq C\hat{E}_V$; moreover, if 
		$Z_{1}, Z_{2} \in \sing\,V\cap (\R\times B_{1/2})$ then
		$$\dist_\H(\spt\|\BC_{Z_1}\|\cap (\R \times B_1),\spt\|\BC_{Z_2}\|\cap (\R \times B_1)) \leq C|Z_1-Z_2|^\alpha Q_{V,\BC}.$$
	\end{enumerate}
	Here, $C = C(n,Q, \alpha)$ and $c = c(n,Q, \alpha)$; in particular, these constants are independent of $\BC$.
\end{theorem}

\begin{proof} {\bf Step 1:} First note that the argument in \cite[Lemma 14.1]{wickramasekera2014general} gives the following: for any given $M_{1} \in [1, \infty)$, there exist constants $\epsilon_1 = \epsilon_1(n,Q, M_{1}, \alpha)\in (0,1)$ and $\gamma_1 = \gamma_1(n,Q, M_{1}, \alpha)\in (0,1)$ such that if a varifold $V\in \S_Q$ and 
a cone $\BC \in \CC_{Q}$ with $S(\BC) = \{(0, 0)\} \times \R^{n-1}$ satisfy (i), (ii) and (iv) in the statement above for some $\epsilon \in (0, \epsilon_{1}]$ and $\gamma \in (0, \gamma_{1}]$, and also satisfies, in place of (iii), the more general condition
\begin{equation}\label{condM} 
\hat{E}_{V}^2 < \frac{3}{2}M_{1}  \inf_{P}\, \int_{\R \times B_{1}} {\rm dist}^{2} \, (X, P) \, \ext\|V\|(X),
\end{equation}
where the infimum is taken over all hyperplanes $P$ with $S(\BC) \subset P$ (i.e.\ over $P$ of the form $P = \{x^{1} = \lambda x^{2}\}$ for $\lambda \in \R$), 
then we can find: a sequence of numbers $\left(\sigma_{k}\right)_{k}$ with $\sigma_{0} = 1$ and 
$\frac{\sigma_{k+1}}{\sigma_{k}} \in [\overline{\theta}, \theta]$ where $\overline{\theta}, \theta \in (0, 1/2)$ are fixed constants depending only on $n$, $Q$, $M_{1}$ and $\alpha$ (in fact for each $k = 0, 1, 2, \ldots$, we have $\frac{\sigma_{k+1}}{\sigma_{k}} \in \{\theta_{1}, \theta_{2}, \ldots, \theta_{2Q-3}\}$, where each $\theta_{j} \in (0, 1/2)$ depends only on $n$, $Q$, $M_{1}$ and $\alpha$); 
orthogonal rotations $\Gamma_{k},  \Gamma: \, \R^{n+1} \to \R^{n+1}$ with $\Gamma_{0} =$ Identity and, for each $k=1, 2, \ldots,$ 
  \begin{equation}\label{E:rotate}
  |\Gamma(e_1) - \Gamma_{k}(e_1)| \leq C\sigma_{k}^{\alpha}Q_{V,\BC}\ \ \ \ \text{and}\ \ \ \ |\Gamma(e_j) - \Gamma_{k}(e_j)| \leq C\sigma_{k}^{\alpha}\hat{E}_V^{-1}Q_{V,\BC}
  \end{equation}
 for $j=2,\dotsc,n+1$; a cone ${\BC_{0}}$ (denoted ${\mathbf H}$ in \cite{wickramasekera2014general}) belonging to $\CC_Q$ with $S({\BC_{0}}) = \{(0, 0)\} \times \R^{n-1}$ such that:
	\begin{equation}\label{E:H}
	\dist_\H(\spt\|\BC_{0}\|\cap (\R\times B_1),\spt\|\BC\|\cap (\R\times B_1)) \leq C Q_{V,\BC}; 
	\end{equation}
	\begin{equation}\label{E:a}
	\sigma^{-n-2}\int_{\R\times B_\sigma}\dist^2(X,\spt\|\BC_{0}\|)\ \ext\|\Gamma^{-1}_{\#} V\|(X) \leq C\sigma^{2\alpha}Q^2_{V,\BC}\ \ \ \ \text{for all }\sigma\in (0,1/2);
	\end{equation}
and also, for $k=1, 2, \ldots$, 
\begin{equation}\label{E:aa}
\sigma_{k}^{-n-2}\int_{\R \times B_{\sigma_{k}}} {\rm dist}^{2} \, (X, {\rm spt} \, \|\BC_{0}\|) \, \ext \|\Gamma_{k \, \#}^{-1}V\|(X) \leq C\sigma_{k}^{2\alpha} Q_{V, \BC}^{2}\; ,\ \; \; \mbox{and},
\end{equation}
\begin{equation}\label{E:HH}
\sigma_{k}^{-n-2}\int_{\R \times \left(B_{\sigma_{k}/2} \setminus \{|x^{2}|<\sigma_{k}/16\}\right)} {\rm dist}^{2} \, (X, {\rm spt} \, \|\Gamma_{k \, \#}^{-1}V\|) \, \ext \|\BC_{0}\|(X) \leq C\sigma_{k}^{2\alpha} Q_{V, \BC}^{2},
\end{equation}
where $C = C(n, Q, M_{1}, \alpha).$  There are a few small modifications to the argument in \cite[Lemma 14.1]{wickramasekera2014general} which we need to make to get this exact statement: first, by \cite[Lemma~9.1]{wickramasekera2014general}, \cite[(10.2)]{wickramasekera2014general} and the argument of \cite[(9.4)]{wickramasekera2014general}, we see that for any given $\tau \in (0, 1/16]$, if $\epsilon_{1} = \epsilon_{1}(n,Q, M_{1}, \tau)$, $\gamma_{1} = \gamma_{1}(n,Q, M_{1}, \tau)$ are sufficiently small, the assumptions of the present theorem imply that  
\begin{equation}\label{good-density-pts} 
\{Z:\Theta_{V}(Z)\geq Q\}\cap (\R\times (B_{1/2}\setminus \{|x^2|<\tau\})) = \emptyset.
\end{equation} 
This with $\tau = 1/16$ is required in \cite[Lemma~14.1]{wickramasekera2014general}. Secondly, for any given $M_{1} \in [1, \infty)$ as above, note that 
\cite[Lemma~13.3]{wickramasekera2014general} holds (with the same proof) in a slightly more general form where: (a) instead of requiring that \cite[Hypothesis~($\star$)]{wickramasekera2014general} holds with $M = \frac{3}{2}M_{0},$ we require that   
\cite[Hypothesis~($\star$)]{wickramasekera2014general} holds with $M = \frac{3}{2}M_{1}M_{0}$, where $M_{0} = M_{0}(n, Q)$ is as in \cite{wickramasekera2014general} (given by the explicit expression in \cite[Section~10]{wickramasekera2014general}),  and 
(b) we allow all of the constants in \cite[Lemma~13.3]{wickramasekera2014general} (i.e.\ $\epsilon$, $\gamma,$ $\kappa$, $C_{0}$, $\nu_{1}, \ldots, \nu_{2Q-3}$, $\overline{C}_{1}$ and $C_{2}$) to depend also on $M_{1}$. Then, arguing as in the proof of \cite[Lemma 14.1]{wickramasekera2014general}, applying this slightly modified version of Lemma~13.3 iteratively in place of \cite[Lemma~13.3]{wickramasekera2014general}, we arrive at \cite[(14.2)--(14.8)]{wickramasekera2014general}, with the constants $C$, $C_{2}$, $\theta_{1}, \ldots, \theta_{2Q-3}$ all depending only on $n$, $Q$, $M_{1}$ and $\alpha$; in particular, we know then that $(\Gamma_k)_k$ is a Cauchy sequence of rotations of $\R^{n+1}$, and so converges to some rotation $\Gamma$ of $\R^{n+1}$. Also, the sequence ${\rm spt} \, \|\BC_{k}\| \cap (\R \times B_{1})$, where 
$\BC_{k} \in \CC_{Q}$ (produced by the modified Lemma~13.3) are as in the proof of  \cite[Lemma 14.1]{wickramasekera2014general}, is a Cauchy sequence in Hausdorff distance, and hence converges in Hausdorff distance to ${\rm spt} \, \|\BC_{0}\| \cap (\R \times B_{1})$ for some $\BC_{0} \in \CC_{Q}.$ It is then not difficult to see using \cite[(14.5)]{wickramasekera2014general} that in fact $\BC_{k} \to \BC_{0}$ as varifolds, after possibly redefining the (constant integer) multiplicities on the half-spaces making up ${\rm spt} \, \|\BC_{0}\|$. These facts and an interpolation between scales in the usual way establish (\ref{E:rotate})--(\ref{E:HH}).

We also have, by \cite[(14.10) and (14.11)]{wickramasekera2014general}, that  
\begin{equation}\label{E:hatE}
C^{-1}\hat{E}_{V} \leq \hat{E}_{\eta_{0, \sigma_{k} \, \#} \Gamma_{k \, \#}^{-1} \, V} \leq C\hat{E}_{V}
\end{equation}
for $k=1, 2, \ldots$, where $C = C(n, Q, M_{1}, \alpha) \in [1, \infty).$ 
	
From (\ref{E:H}) and (\ref{E:HH}), it follows that $\Gamma_{\#} \, \BC_{0}$ is the unique tangent cone to $V$ at $0$, so that $\eta_{0, \rho \, \#} \, V \rightharpoonup \Gamma_{\#} \, \BC_{0}$ as $\rho \to 0$. 
 
 \noindent
{\bf Step 2:}  Still assuming (\ref{condM}) for some $M_{1},$ we wish to repeat Step 1 after shifting the (density $\geq Q$) base point. To be able to ensure that after shifting the base point the required hypotheses (i.e.\ (i), (ii), (iv) for appropriate $\epsilon$, $\gamma$, and hypothesis (\ref{condM}) with a fixed choice of a constant (depending only on $n$, $Q$, $\alpha$ and $M_{1}$) in place of $M_{1}$) are satisfied, we need to introduce an auxiliary condition, namely \cite[Hypothesis $(\star\star)$]{wickramasekera2014general}; by switching to an appropriate different cone $\BC^{\prime}$ (in place of $\BC$) and choosing $\epsilon$, $\gamma$ sufficiently small, we can arrange for this condition to always be satisfied. We now provide details of this argument. 
	
 Write $\CC_Q(p)$ for the set of $\widetilde{\BC}\in \CC_Q$ with $S(\widetilde{\BC}) = \{(0, 0)\} \times \R^{n-1}$ for which the number of distinct half-hyperplanes in $\spt\|\widetilde{\BC}\|$ is $p$. If $V \in \S_{Q}$, $\BC \in \CC_{Q}$ are as in the theorem but with (\ref{condM}) in place of hypothesis (iii), then for $\epsilon = \epsilon(n,Q,M_{1}, \alpha)$, $\gamma = \gamma(n,Q, M_{1}, \alpha)$ sufficiently small, we know by \cite[Lemma~9.1]{wickramasekera2014general} that we necessarily have $\BC\in \cup_{p=4}^{2Q}\CC_Q(p)$. Set
$$Q^*_V(p) := \inf_{\widetilde{\BC}\in \cup^p_{k=4}\CC_Q(k)}Q_{V,\widetilde{\BC}}\; .$$
Following \cite{wickramasekera2014general}, consider, for $\beta \in (0, 1/2)$: 
	 
{\sc hypothesis $(\star\star)$}. Either 
\begin{enumerate}
    \item [(I)] $\BC\in \CC_Q(4)$, or
    \item [(II)] $Q\geq 3$, $\BC\in \CC_Q(p)$ for some $p\in \{5,\dotsc,2Q\}$, and $Q_{V,\BC}^2 < \beta \left(Q^*_V(p-1)\right)^2$.
\end{enumerate}
	
By the argument of \cite[Proposition 10.5]{wickramasekera2014general} we have the following: \emph{suppose that (\ref{condM}) holds for some $M_{1} \in [1, \infty),$ and let $\widetilde{\epsilon}, \widetilde{\gamma}  \in (0, 1/2)$ be given. There exists $\beta_{0} = \beta_{0}(n, Q, M_{1}, \alpha), \epsilon_{2} = \epsilon_{2}(n, Q, M_{1}, \alpha, \widetilde{\epsilon}, \widetilde{\gamma})$, and $\gamma_{2}(n, Q, M_{1}, \alpha, \widetilde{\epsilon}, \widetilde{\gamma})  \in (0, 1/2)$ such that if hypotheses \textnormal{(i)}, \textnormal{(ii)}, \textnormal{(iv)} and Hypothesis~$(\star\star)$ hold with $\epsilon = \epsilon_{2}$, $\gamma = \gamma_{2},$ and $\beta = \beta_{0}$, then for any $Z\in \spt\|V\|\cap (\R\times B_{9/16})$ with $\Theta_V(Z)\geq Q$, if we set $V_Z = (\eta_{Z,1/8})_\# V$,  we have (see \cite[p.~912]{wickramasekera2014general}) that (\ref{condM}) holds with $V_{Z}$ in place of $V$ and with $M_{1}M_{0}$ in place of $M_{1}$, i.e.\   
	\begin{equation}\label{E:hyp3Z} 
	\hat{E}_{V_Z}^2 < \frac{3}{2}M_{1}M_0\inf\int_{\R\times B_1}\dist^2(X,P)\ \ext\|V_Z\|(X),
	\end{equation}
	where $M_0 = M_0(n,Q) \in [1, \infty)$  is an explicit constant (as in \cite{wickramasekera2014general}) and where the infimum is  taken over all hyperplanes of the form $P = \{x^1=\lambda x^2\}$, $\lambda\in \R$, and we also have that hypotheses  
	\textnormal{(i)}, \textnormal{(ii)}, \textnormal{(iv)} hold with $V_Z$ in place of $V$, with the same cone $\BC$ and with $\epsilon = \widetilde{\epsilon}$, 
	$\gamma = \widetilde{\gamma}$. Furthermore, we have 
	by \cite[(10.32)]{wickramasekera2014general} that 
	\begin{equation}\label{E:hyp4Z}
	\hat{E}_{V_{Z}} \geq C\hat{E}_{V} \; \mbox{and,}
	\end{equation}
	by combining \cite[Corollary~10.2(a)]{wickramasekera2014general} and \cite[(10.33)]{wickramasekera2014general}, that 
\begin{equation}\label{E:hyp5Z}
Q_{V_{Z}, \BC} \leq CQ_{V, \BC}\; ,
\end{equation}
where $C = C(n, Q, M_{1}, \alpha).$} We emphasise that $\beta_{0}$ is independent of $\widetilde{\epsilon}$ and $\widetilde{\gamma}$ (and note that the purpose of Hypothesis~($\star\star$) is to enable us to use \cite[Theorem~10.1(a)]{wickramasekera2014general} in the proof of 
\cite[Proposition~10.5]{wickramasekera2014general}; we do not need to verify that Hypothesis~($\star\star$) holds with $V_{Z}$ in place of $V$). 
	
On the other hand, for any given $\beta \in (0, 1/2)$, $V \in \S_{Q}$ and $\BC \in \CC_{Q}$, if Hypothesis $(\star\star)$ fails, then we must have $Q \geq 3$ and $\BC \in \CC_{Q}(p)$ for some $p \in \{5, \ldots, 2Q\}$, and we can find a number $\ell \in \{4, \ldots, p-1\}$ and a cone $\BC^{\prime} \in \CC_{Q}(\ell)$  such that $Q_{V, \BC^{\prime}} \leq \left(\frac{3}{2\beta}\right)^{m} Q_{V, \BC}$  for some $m \in \{1, 2, \ldots, p-4\}$ and $\BC^{\prime}$ satisfies Hypotheses~($\star\star$), i.e.\ either we have $\BC^{\prime} \in \CC_{Q}(4)$ or we have $\ell \in \{5, \ldots, p-1\}$ and $Q_{V, \BC^{\prime}}^{2} < \beta \left(Q_{V}^{\star}(\ell - 1)\right)^{2}$. 

Thus for any given $\beta \in (0, 1/2)$, $V \in \S_{Q}$, $\BC \in \CC_{Q}$, we can always find a cone $\BC^{\prime} \in \CC_{Q}$ (possibly with $\BC^{\prime} = \BC$) such that  
\begin{enumerate} 
\item [(A)] $Q_{V, \BC^{\prime}} \leq \left(\frac{3}{2\beta}\right)^{2Q-4} Q_{V, \BC}$, and 
\item [(B)]  either 
\begin{enumerate}  
\item[(I)$^{\prime}$] $\BC^{\prime} \in \CC_{Q}(4)$, or 
\item[(II)$^{\prime}$]  $Q \geq 3$, $\BC^{\prime} \in \CC_{Q}(p)$ for some $p \in \{5, \ldots, 2Q\}$ and $Q_{V, \BC^{\prime}}^{2} < \beta \left(Q_{V}^{\star}(p - 1)\right)^{2}.$
\end{enumerate}
\end{enumerate}

So taking $\beta = \beta_{0}$ and applying the preceding discussion taking $\BC^{\prime}$ in place of $\BC$ and replacing $\gamma_{2}$ with 
$\left(\frac{3}{2\beta_{0}}\right)^{-(2Q-4)}\gamma_{2}$, we deduce the following:  

\emph{Claim: For each $M_{1} \in [1, \infty)$, there exists $\beta_{0} = \beta_{0}(n, Q, M_{1}, \alpha) \in (0, 1/2)$ such that 
if (\ref{condM}) holds, then for any given $\widetilde{\epsilon}, \widetilde{\gamma} \in (0, 1/2),$  there exist numbers $\epsilon_{2} = \epsilon_{2}(n, Q, M_{1}, \alpha, \widetilde{\epsilon}, \widetilde{\gamma})$, $\gamma_2 = \gamma_{2}(n, Q, M_{1}, \alpha, \widetilde{\epsilon}, \widetilde{\gamma})  \in (0, 1/2)$ such that whenever a varifold $V \in \S_{Q}$ and a cone $\BC \in \CC_{Q}$ satisfy hypotheses \textnormal{(i)}, \textnormal{(ii)}, \textnormal{(iv)} with $\epsilon = \epsilon_{2}$, $\gamma = \gamma_{2},$  there is a cone 
$\BC^{\prime} \in \CC_{Q}$  satisfying:    
\begin{enumerate} 
\item [(A)] 
\begin{equation}\label{new-cone}
Q_{V, \BC^{\prime}} \leq \left(\frac{3}{2\beta_{0}}\right)^{2Q-4} Q_{V, \BC}\, , \;\; \mbox{and}
\end{equation} 
\item [(B)]  either 
\begin{enumerate}  
\item[(I)$^{\prime}$] $\BC^{\prime} \in \CC_{Q}(4)$ or 
\item[(II)$^{\prime}$]  $Q \geq 3$, $\BC^{\prime} \in \CC_{Q}(p)$ for some $p \in \{5, \ldots, 2Q\}$ and $Q_{V, \BC^{\prime}}^{2} < \beta_{0} \left(Q_{V}^{\star}(p - 1)\right)^{2}$
\end{enumerate}
\end{enumerate}
such that for any $Z\in \spt\|V\|\cap (\R\times B_{9/16})$ with $\Theta_V(Z)\geq Q$, if we set $V_Z = (\eta_{Z,1/8})_\# V$, we have that hypotheses  
	\textnormal{(i)}, \textnormal{(ii)}, \textnormal{(iv)} hold with $V_Z$ in place of $V$, $\BC^{\prime}$ in place of $\BC$, 
	and with $\epsilon = \widetilde{\epsilon}$, 
	$\gamma = \widetilde{\gamma}$; we also have that (\ref{condM}) holds with $V_{Z}$ in place of $V$ and with $M_{1}M_{0}$ in place of $M_{1},$ where $M_{0} = M_{0}(n, Q) \in [1, \infty).$ Additionally, we have (by (\ref{E:hyp4Z}) and (\ref{E:hyp5Z})) that  
		\begin{equation}\label{E:hyp4Z-1}
	\hat{E}_{V_{Z}} \geq C\hat{E}_{V} \; \mbox{and}
	\end{equation}
	\begin{equation}\label{E:hyp5Z-2}
Q_{V_{Z}, \BC^{\prime}} \leq CQ_{V, \BC^{\prime}}\; ,
\end{equation}
where $C = C(n, Q, M_{1}, \alpha).$} 

Note also that in view of the fact that $\BC^{\prime}$ satisfies (B) above, and that (i), (ii), (iv) hold with 
$\eta_{0, 1/8 \, \#} \, V$ in place of $V,$ $\BC^{\prime}$ in place of $\BC$, and with $\epsilon = \widetilde{\epsilon}$, $\gamma = \widetilde{\gamma}$, and that (\ref{condM}) holds, 
it follows from \cite[Theorem~10.1(a)]{wickramasekera2014general} (taken with $\tau = 1/16$, say, and with $\BC^{\prime}$ in place of $\BC$) and (\ref{new-cone}) that, provided $\epsilon_{2}$, $\gamma_{2}$ are sufficiently small depending only on $n$, $Q$, $\alpha,$
\begin{equation}\label{cone-dist-bound}
{\rm dist}_{\mathcal H} \, ({\rm spt} \, \|\BC^{\prime}\| \cap (\R \times B_{1}), {\rm spt} \, \|\BC\| \cap (\R \times B_{1})) \leq CQ_{V, \BC}\; ,
\end{equation} 	
where $C = C(n, Q, M_{1}, \alpha).$ 
	
\noindent
{\bf Step 3:} Suppose that (\ref{condM}) holds for some $M_{1} \in [1, \infty)$, and let $\overline{\epsilon}_{1} = \epsilon_1 = \epsilon_{1}(n, Q, M_{1}M_{0}, \alpha)$, $\overline{\gamma}_{1} = \gamma_1 = \gamma_{1}(n, Q, M_{1}M_{0}, \alpha),$ where $\epsilon_{1}$, $\gamma_{1}$ are as in Step~1. Suppose also that (i), (ii), (iv) hold with $\epsilon \in (0, \epsilon_2]$, $\gamma \in (0, \gamma_2]$ where $\epsilon_{2} = \epsilon_{2}(n, Q, M_{1}M_{0}, \alpha,\overline{\epsilon}_{1}, \overline{\gamma}_{1})$, $\gamma_{2} = \gamma_{2}(n, Q, M_{1}M_{0}, \alpha, \overline{\epsilon}_{1}, \overline{\gamma}_{1})$, are as in the Claim in Step~2. 
	
In view of the Claim in Step 2, there is a cone $\BC^{\prime} \in \CC_{Q}$ such that (by applying Step 1) we have the conclusions as in Step~1 with $\BC^{\prime}$ in place of $\BC$ and uniformly at each ``base point'' $Z\in \spt\|V\|\cap (\R\times B_{9/16})$ with $\Theta_V(Z)\geq Q$, i.e.\ for each such $Z$ we can find: a sequence of numbers $\left(\sigma_{k}^{Z}\right)_{k}$ with $\sigma_{0} = 1$ and 
	\begin{equation}\label{scale-bounds}
	\frac{\sigma_{k+1}^{Z}}{\sigma_{k}^{Z}} \in [\overline{\theta}, \theta];
	\end{equation}
	 orthogonal rotations $\Gamma_{k}^{Z},  \Gamma_{Z}: \, \R^{n+1} \to \R^{n+1}$ with $\Gamma_{0}^{Z} =$ Identity and, for each $k=1, 2, \ldots,$ 
  \begin{equation}\label{E:rotateZ}
  |\Gamma_{Z}(e_1) - \Gamma_{k}^{Z}(e_1)| \leq C\left(\sigma_{k}^{Z}\right)^{\alpha}Q_{V_{Z},\BC^{\prime}}\ \ \ \ \text{and}\ \ \ \ |\Gamma_{Z}(e_j) - \Gamma_{k}^{Z}(e_j)| \leq C\left(\sigma_{k}^{Z}\right)^{\alpha}\hat{E}_{V_{Z}}^{-1}Q_{V_{Z},\BC^{\prime}}
\end{equation}  
for $j=2,\dotsc,n+1$;	
	a cone ${\BC_{Z}}\in \CC_Q$ with $S({\BC_{Z}}) = \{(0, 0)\} \times \R^{n-1}$, such that:
	\begin{equation}\label{E:HZ}
	\dist_\H(\spt\|\BC_{Z}\|\cap (\R\times B_1),\spt\|\BC^{\prime}\|\cap (\R\times B_1)) \leq C Q_{V_{Z},\BC^{\prime}} \;\; \mbox{and} 
	\end{equation}
	\begin{equation}\label{E:aZ}
	\sigma^{-n-2}\int_{\R\times B_\sigma}\dist^2(X,\spt\|\BC_{Z}\|)\ \ext\|\Gamma_{Z \, \#}^{-1}V_{Z}\|(X) \leq C\sigma^{2\alpha}Q^2_{V_{Z},\BC^{\prime}}\ \ \ \ \text{for all }\sigma\in (0,1/8),
	\end{equation}
and also, for $k=1, 2, 3, \ldots$,
\begin{equation}\label{E:aaZ}
\left(\sigma_{k}^{Z}\right)^{-n-2}\int_{\R \times B_{\sigma^{Z}_{k}}} {\rm dist}^{2} \, (X, {\rm spt} \,\|\BC_{Z}\|) \, \ext \|\left(\Gamma_{k}^{Z}\right)^{-1}_{ \, \#} V_{Z}\|(X)\leq C\left(\sigma_{k}^{Z}\right)^{2\alpha} Q_{V_{Z}, \BC^{\prime}}^{2}; 
\end{equation}
\begin{equation}\label{E:HHZ}
\left(\sigma_{k}^{Z}\right)^{-n-2}\int_{\R \times \left(B_{\sigma_{k}^{Z}/2} \setminus \{|x^{2}|<\sigma_{k}^{Z}/16\}\right)} {\rm dist}^{2} \, (X, {\rm spt} \, \|\left(\Gamma_{k}^{Z}\right)^{-1}_{ \, \#} V_{Z}\|) \, \ext \|\BC_{Z}\|(X) \leq C\left(\sigma_{k}^{Z}\right)^{2\alpha} Q_{V_{Z}, \BC^{\prime}}^{2} \;\; \mbox{and}
\end{equation}
\begin{equation}\label{E:hatEZ}
C^{-1}\hat{E}_{V_{Z}} \leq \hat{E}_{\eta_{0, \sigma_{k}^{Z} \, \#} \left(\Gamma_{k}^{Z}\right)^{-1}_{\#} \, V_{Z}} \leq C\hat{E}_{V_{Z}}, 
\end{equation}
where $C = C(n, Q, M_{1}, \alpha).$ Furthermore, by  (\ref{new-cone}), (\ref{E:hyp4Z-1}) and (\ref{E:hyp5Z-2}), 
we also have that 
		\begin{equation}\label{E:EZ}
	\hat{E}_{V_{Z}} \geq C\hat{E}_{V}\; , \; \mbox{and}
	\end{equation}
\begin{equation}\label{E:QZ}
Q_{V_{Z}, \BC^{\prime}} \leq \overline{C}Q_{V, \BC} 
\end{equation}
where $\overline{C} = \overline{C}(n, Q, M_{1}, \alpha).$

From these, we can draw the following conclusions:  
\begin{itemize}
\item for every $Z\in \spt\|V\|\cap (\R\times B_{9/16})$ with $\Theta_{V}(Z) \ge Q,$ $\Gamma_{Z \, \#} \, \BC_{Z} \in \CC_{Q}$ is the unique tangent cone to $V$ at $Z$; 
\item  $\{\Theta_V\geq Q\}\cap (\R\times B_{9/16}(0)) = \{\Theta_V = Q\}\cap (\R\times B_{9/16}(0));$  
\item $\{\Theta_V\geq Q\}\cap (\R\times B_{9/16}(0)) = \CC_V\cap (\R\times B_{9/16}(0));$ this is so by hypothesis $(\S3)_{Q}$ and an application of Theorem~\ref{thm:B}; 
\item $\B_V \cap (\R \times B_{9/16}) = \emptyset;$ this is so by the preceding fact since by Theorem~\ref{sheeting} we have that $\Theta_{V}(Z) \geq Q$ for every $Z \in \B_{V}$. 
\end{itemize}

\noindent
{\bf Step 4:} Continue to assume that (\ref{condM}) holds for some $M_{1} \in [1, \infty)$, and also, for $\widetilde{\epsilon}$, $\widetilde{\gamma}$ to be chosen depending only on $n$, $Q$, $M_{1},$ that (i), (ii), (iv) 
hold with 
\begin{align}\label{parameter-cond}
\begin{split}
\epsilon & \in (0,\epsilon_{2}(n, Q, M_{1}M_{0}, \widetilde{\epsilon}, \widetilde{\gamma}, \alpha)],\\
\gamma & \in (0,  \gamma_{2}(n, Q, M_{1}M_{0}, \widetilde{\epsilon}, \widetilde{\gamma}, \alpha)],
\end{split}
\end{align}
where $\epsilon_{2}$, $\gamma_{2}$ 
are as in the Claim in Step 2.

Set $S = \pi(\{\Theta_V \geq Q\})$ where $\pi:\R^{n+1}\to \{0\}\times \R^n$ is the orthogonal projection. If $y\in B_{1/2}(0)$ has $\dist(y,S)>1/16$, then
	$$(1/16)^{-n-2}\int_{\R\times B_{1/16}(y)}|x^1|^2\ \ext\|V\|(X) \leq  (1/16)^{-n-2}\hat{E}_V^2$$
	and so if $\epsilon < 16^{-n-2}\epsilon_{0}$  where $\epsilon_{0} = \epsilon_{0}(n, Q)$ is as in Theorem~\ref{sheeting}, as $V$ has no classical singularities of density $< Q$ in the region $\R\times B_{1/16}(y)$, we may apply Theorem~\ref{sheeting} to see that  $V$ is regular in $\R\times B_{1/32}(y)$, with 
	\begin{equation}\label{E:embedded}
	V \res \, (\R \times B_{1/32}(y)) = \sum_{j=1}^{Q} |{\rm graph} \, u_{j}|,
	\end{equation}
	 for $u_{j} \in C^{2}(B_{1/32}(y))$ satisfying $u_{1} \leq \cdots \leq u_{Q}$ and $|u|_{2, B_{1/32}(y)} \leq C\hat{E}_{V}$, where $C = C(n, Q)$;  if on the other hand $y \in B_{1/2}$ is such that $0<\dist(y,S)\leq 1/16,$ we choose $z\in S\cap B_{9/16}(0)$ for which $\dist(y,S) = \dist(y,z)$, set $\sigma_y = \dist(y,z)/4,$ choose $Z = (\zeta^1,\zeta^2,\eta)$ with $\Theta_{V}(Z) \ge Q$ and $\pi (Z) = z$ and choose the integer $k$ such that $\sigma_{k+1}^{Z} < 40\sigma_{y} \leq \sigma_{k}^{Z},$ and note that
	\begin{align*}
	&\sigma_y^{-n-2}\int_{\R\times B_{\sigma_y}(y)}|x^1 - \zeta^{1}|^2\ \ext \|V\|(X) \leq 
	5^{n+2}\cdot (5\sigma_y)^{-n-2}\int_{\R\times B_{5\sigma_y}(z)}|x^{1} - \zeta^{1}|^{2} \ext\|V\|(X) \nonumber\\
	&\hspace{1in}\leq C\overline{\theta}^{-n-2}(\sigma_{k}^{Z}/4)^{-n-2}\int_{\R \times B_{\sigma_{k}^{Z}/4}(0)} |x^{1}|^{2} \ext\|V_{Z}\|(X)\nonumber\\
	&\hspace{1in} \leq C_{1}\left(\hat{E}_{V_{Z}}^{2} + \hat{E}_{V_{Z}}^{-2}Q_{V_{Z}, \BC^{\prime}}^{2}\right) \leq C_{1}(\widetilde{\epsilon} + \widetilde{\gamma}),
	\end{align*}
where $C_{1} = C_{1}(n, Q, M_{1}, \alpha),$ and we have used (\ref{E:rotateZ}) and (\ref{E:hatEZ}) in the third inequality and Step~2 in the last inequality. 

Now choose 
\begin{equation}\label{parameter-cond-1}
\widetilde{\epsilon} = \min\{\overline{\epsilon}_{1}, (2C_{1})^{-1}\epsilon_{0}\} \; \mbox{and} \; \widetilde{\gamma} = \min\{\overline{\gamma}_{1}, (2C_{1})^{-1}\epsilon_{0}\}, 
\end{equation}
where $\overline{\epsilon}_{1} = \epsilon_{1}(n, Q, M_{1}M_{0}, \alpha)$, $\overline{\gamma}_{1} = \gamma_{1}(n, Q, M_{1}M_{0}, \alpha)$ are as in Step~3  and $\epsilon_{0} = \epsilon_{0}(n, Q)$ is as in Theorem~\ref{sheeting}.
Then, if (\ref{condM}) holds for some $M_{1} \in (0, \infty)$, and if hypotheses (i), (ii), (iv) 
hold with $\epsilon$, $\gamma$,  
satisfying (\ref{parameter-cond}),  we have the conclusions at the end of Step~3, and moreover, by the preceding estimate, that 
$\sigma_y^{-n-2}\int_{\R\times B_{\sigma_y}(y)}|x^1 - \zeta^{1}|^2\ \ext \|V\|(X) < \epsilon_{0}$. Whence, 
applying Theorem~\ref{sheeting} again, we get that $V \res (\R \times B_{\sigma_{y}/2}(y))$ is given by $Q$ embedded, ordered minimal graphs over $B_{\sigma_{y}/2}(y)$ with small gradient. Since this holds for every $y \in B_{1/2} \cap \{x \, : \, 0 < {\rm dist} \, (x, S) < 1/16\}$, and (as we have already seen) $\{\Theta_{V} \geq  Q\} \cap \left(\R \times B_{9/16}\right)$ consists of $C^{1, \alpha}$ classical singularities, it follows that there is a function $u \, : \, B_{1/2} \to {\mathcal A}_{Q}(\R)$ of class $GC^{1, \alpha}$ with $\CC_{u} = S \cap B_{1/2}$ and $\B_{u} = \emptyset$ such that $V \res (\R \times B_{1/2}) = {\mathbf v}(u).$ This establishes conclusions (1), (2) and the first assertion of conclusion (3).

\noindent
{\bf Step 5:} Now we suppose that hypotheses (i)--(iv) as in the statement of the theorem hold with 
$\epsilon   \in (0, \epsilon_{2}(n, Q, M_{0}, \alpha, \epsilon^{\prime}, \gamma^{\prime})]$ and $\gamma  \in (0, \gamma_{2}(n, Q, M_{0}, \alpha, \epsilon^{\prime}, \gamma^{\prime})]$, where $\epsilon_{2}$, $\gamma_{2}$  are as in Step~2 and $\epsilon^{\prime}$, $\gamma^{\prime}$ are to be chosen depending only on $n$, $Q$, and $\alpha$.  In particular, we require that 
$\epsilon^{\prime} < \widetilde{\epsilon}$ and $\gamma^{\prime} < \widetilde{\gamma}$, where $\widetilde{\epsilon}$, $\widetilde{\gamma}$ are as in (\ref{parameter-cond-1}) with $M_{1}=1$. Since $M_{0}$ depends only on $n$, $Q$, our eventual choice of $\epsilon^{\prime}$, $\gamma^{\prime}$ will imply that $\epsilon$, $\gamma$ and $\beta$ depend only on $n$, $Q$, and $\alpha$.

We wish to establish conclusion (4) and conclusions (3)(i) and (3)(ii). Since (by Step 4 taken with $M_{1} = 1$) we have conclusions (1), (2) and the first assertion of conclusion (3), this will complete the proof of the theorem. 

First note that the inequalities $c\hat{E}_V\leq \dist_\H(\spt\|\BC_Z\|\cap B_1,\{0\}\times B_1)\leq C\hat{E}_V,$ for some constants $c= c(n, Q, \alpha) \in (0, \infty)$ and $C = C(n, Q, \alpha) \in (0, \infty),$ follow from (\ref{E:HZ}) and hypothesis (iv). 

For the H\"older continuity estimate in conclusion (4), we proceed as follows: pick any two points $Z_1,Z_2 \in \CC_{V} \cap (\R \times B_{1/2})$, and set $\sigma = |Z_1-Z_2|$. If $\sigma \ge 1/32$ the estimate holds trivially, so assume $\sigma<1/32$, and choose $k$ such that $\sigma_{k+1}^{Z_{2}} < 16\sigma \leq \sigma_{k}^{Z_{2}}$. Set $\widetilde{V}= (\eta_{0,\sigma_{k}^{Z_{2}}}\circ 
	\left(\Gamma^{Z_2}_{k}\right)^{-1})_{\#} \,  V_{Z_2}$ and  
	$\widetilde{Z} = (\sigma_{k}^{Z_{2}}/8)^{-1}\left(\Gamma^{Z_2}_{k}\right)^{-1}(Z_1-Z_2)$. Then 
	$\widetilde{V}_{\widetilde{Z}} = (\eta_{0,\sigma_{k}^{Z_{2}}}\circ 
	\left(\Gamma^{Z_2}_{k}\right)^{-1})_{\#} \,  V_{Z_1},$ and $\widetilde{Z}$ is a density $Q$ point of $\widetilde{V}$,  and moreover by (\ref{E:hatEZ}) and (\ref{E:hyp3Z}) (with $M_{1} = 1$) we have that 
	$$\hat{E}_{\widetilde V} \leq C\hat{E}_{V_{Z_{2}}} \leq \frac{3}{2}CM_0\inf _{P = \{x^{1} = \lambda \, x^{2} \, : \, \lambda \in \R\}} \int_{\R\times B_1}\dist^2(X,P)\ \ext\|V_{Z_{2}}\|(X)\; ,$$ where $C = C(n, Q, \alpha).$ Also, by (\ref{E:aaZ}), (\ref{E:HHZ}), (\ref{E:hatEZ}), (\ref{E:EZ}) and (\ref{E:QZ}) we have that 
	\begin{equation}\label{fine-coarse-Z} 
	Q_{\widetilde{V},\BC_{Z_2}} \leq C\sigma_{k}^\alpha Q_{V_{Z_{2}},\BC^{\prime}} \leq C\sigma_{k}^{\alpha} Q_{V, \BC} \leq C\gamma\hat{E}_{V} \leq C\gamma\hat{E}_{V_{Z_{2}}} \leq C^{\prime}\gamma \hat{E}_{\eta_{0, \sigma_{k}^{Z_{2}} \, \#} \left(\Gamma_{k}^{Z_{2}}\right)^{-1}_{\#} \, V_{Z_{2}}} = C^{\prime}\gamma\hat{E}_{\widetilde{V}}\; ,
	\end{equation} 
where $C^{\prime} = C^{\prime}(n, Q, \alpha)$. We also clearly have that 
$\hat{E}_{V_{Z_{2}}} < C\epsilon$, whence $\hat{E}_{\widetilde{V}} < C\epsilon$. 

Thus  (\ref{condM}) is satisfied with $\widetilde{V}$ in place of $V$ and $M_{1} = CM_{0}$, and hypotheses (i), (ii), (iv) are satisfied with $\widetilde{V}$ in place of $V$, $\BC_{Z_{2}}$ in place of $\BC$ and with 
$C\epsilon$ in place of $\epsilon$ and $C^{\prime}\gamma$ in place of $\gamma$. So if we choose $\epsilon$, $\gamma$ sufficiently small depending only on $n$, $Q$, and $\alpha$, we can apply the Claim in Step~2 to find a cone $\BC^{\prime\prime} \in \CC_{Q}$, which, by 
(\ref{E:QZ}) and (\ref{cone-dist-bound}) satisfies 
\begin{equation}\label{fine-excess-est} 
Q_{\widetilde{V}_{\widetilde{Z}}, \BC^{\prime\prime}} \leq \overline{C} Q_{\widetilde{V}, \BC_{Z_{2}}}\; ,
\end{equation}
where $\overline{C}$ is the constant as in (\ref{E:QZ}) with $M_{1} = CM_{0}$, and  
\begin{equation}\label{cone-dist-bound-final}
{\rm dist}_{\mathcal H} \, ({\rm spt} \, \|\BC^{\prime\prime}\| \cap (\R \times B_{1}), {\rm spt} \, \|\BC_{Z_{2}}\| \cap (\R \times B_{1})) \leq 
CQ_{\widetilde{V}, \BC_{Z_{2}}}\; ,
\end{equation} 	
where $C = C(n, Q,\alpha),$ 
and such that 
(i), (ii), (iv) hold with $\widetilde{V}_{\widetilde{Z}}$ in place of $V$, $\BC^{\prime\prime}$ in place of $\BC$ and with 
$\epsilon_{1} (n, Q, CM_{0}^{2}, \alpha)$ in place of $\epsilon$ and $\overline{C}^{-1}\gamma_{1}(n, Q, CM_{0}^{2}, \alpha)$ in place of $\gamma$, where $\epsilon_{1}$, $\gamma_{1}$ are as in Step~1. 

Thus we can apply Step~1 to obtain a cone $\widetilde{\BC}_{\widetilde{Z}}\in \CC_Q$ with $S(\widetilde{\BC}_{\widetilde{Z}}) = \{(0, 0)\} \times \R^{n-1}$ and a rotation $\widetilde{\Gamma}_{\widetilde{Z}}$ such that 
\begin{equation}\label{E:H3-a}
	\dist_\H(\spt\|\widetilde{\BC}_{\widetilde{Z}}\|\cap (\R\times B_1),\spt\|\BC^{\prime
	\prime}\|\cap (\R\times B_1)) \leq C Q_{\widetilde{V}_{\widetilde{Z}},\BC^{\prime\prime}} \; \mbox{and}
	\end{equation}
	\begin{equation}\label{E:rotate-est}
  |\widetilde{\Gamma}_{\widetilde{Z}}(e_1) - e_1| \leq CQ_{\widetilde{V}_{\widetilde{Z}},\BC^{\prime\prime}}\ \ \ \ \text{and}\ \ \ \ |\widetilde{\Gamma}_{\widetilde{Z}}(e_j) - e_j| \leq C\hat{E}_{\widetilde{V}_{\widetilde{Z}}}^{-1}Q_{\widetilde{V}_{\widetilde{Z}},\BC^{\prime\prime}}
\end{equation}  
for $j=2,\dotsc,n+1$.	These together with (\ref{fine-excess-est}) and (\ref{cone-dist-bound-final}) imply 
\begin{equation}\label{E:H3}
	\dist_\H(\spt\|\widetilde{\Gamma}_{\widetilde{Z} \, \#} \, \widetilde{\BC}_{\widetilde{Z}}\|\cap (\R\times B_1),\spt\|\BC_{Z_2}\|\cap (\R\times B_1)) \leq C Q_{\widetilde{V},\BC_{Z_{2}}}. 
	\end{equation}
	Now as $\left(\widetilde{\Gamma}_{\widetilde{Z}}\right)_\#\widetilde{\BC}_{\widetilde{Z}}$ is also the unique tangent cone of $\widetilde{V}$ at $\widetilde{Z}$, we can readily check that
	\begin{equation}\label{rotation-equality}
	\left(\widetilde{\Gamma}_{\widetilde{Z}}\right)_{\#} \widetilde{\BC}_{\widetilde{Z}} = 
	\left[\left(\Gamma_{k}^{Z_2}\right)^{-1}\circ{\Gamma_{Z_1}}\right]_\# \BC_{Z_1}
	\end{equation}
	which combined with (\ref{E:H3}) and (\ref{E:rotateZ})  gives 
\begin{equation}
	\dist_\H(\spt\|\Gamma_{Z_{1} \, \#} \, \BC_{Z_{1}}\|\cap (\R\times B_1),\spt\|\Gamma_{Z_{2} \, \#} \, \BC_{Z_2}\|\cap (\R\times B_1)) \leq C Q_{\widetilde{V},\BC_{Z_{2}}} + C \left(\sigma_{k}^{Z_{2}}\right)^{\alpha}Q_{V_{Z_{2}},\BC^{\prime}}. 
\end{equation} 	
By (\ref{fine-coarse-Z}) and (\ref{E:QZ}), this readily gives 
\begin{equation}
	\dist_\H(\spt\|\Gamma_{Z_{1} \, \#} \, \BC_{Z_{1}}\|\cap (\R\times B_1),\spt\|\Gamma_{Z_{2} \, \#} \, \BC_{Z_2}\|\cap (\R\times B_1)) 
	\leq C |Z_{1} - Z_{2}|^{\alpha}Q_{V,\BC} 
\end{equation} 	
which is conclusion (4). 

Finally, to see (3)(i) and (3)(ii), note first that by Theorem~\ref{thm:B} it follows that for each point $(0, 0, y) \in \{(0, 0)\} \times \R^{n-1} \cap B_{1/2} \equiv B_{1/2}^{n-1}(0)$, 
we have that $\left(\R^{2} \times \{y\}\right) \cap \{\Theta_{V} = Q\} \neq \emptyset;$ moreover, this set consists precisely of one point $Z_{y}$. (To see this last assertion, note that if for some $(0, 0, y) \in B_{1/2}^{n-1}(0)$ we have two distinct points 
$Z_{1}, Z_{2}  \in \left(\R^{2} \times \{y\}\right) \cap \{\Theta_{V} = Q\}$, then we can choose $k$ such that $\sigma_{k+2}^{Z_{1}} < |Z_{1} - Z_{2}| \leq \sigma_{k+1}^{Z_{1}}$ and use (\ref{E:aaZ})--(\ref{E:QZ}) with $Z = Z_{1}$ to see that (\ref{good-density-pts}) must hold with 
$V^{\star} = \eta_{Z^{1}, \sigma_{k}^{Z_{1}} \, \#} \left(\Gamma_{k}^{Z_{1}}\right)^{-1}_{\#} \, V$ in place of $V$ and  $\tau = (\overline{\theta})^{2}/2$, where $\overline{\theta}$ is as in (\ref{scale-bounds}) taken with $M_{1} = CM_{0},$ provided $\epsilon_{2}, \gamma_{2}$ are sufficiently small depending only on $n$, $Q$, $\alpha$; but since $\theta \geq \frac{|Z_{1} - Z_{2}|}{\sigma_{k}^{Z_{1}}} \geq (\overline{\theta})^{2},$  
$Z^{\star}  = \left(\Gamma_{k}^{Z_{1}}\right)^{-1}(Z_{2} - Z_{1})/\sigma_{k}^{Z_{1}}$ satisfies $\Theta_{V^{\star}}(Z^{\star}) \geq Q$ and $\pi_{(0, 0) \times \R^{n-1}}(Z_{1} - Z_{2}) = 0$, this is a contradiction.) So define $\phi \, : \, B_{1/2}^{n-1}(0) \to \R^{2}$ by setting  $\phi(y) = (\phi_{1}(y), \phi_{2}(y), y) = Z_{y}$.  Then $\phi$ is of class $C^{1, \alpha}$ since 
${\rm graph} \, \phi = \CC_{V}$ and, by Theorem~\ref{thm:B}, $\CC_{V}$ is an $(n-1)$-dimensional $C^{1, \alpha}$  submanifold. Moreover, 
the tangent plane to $\CC_{V}$ at a point $Z \in \CC_{V}$ is the spine of the cone $\Gamma_{Z \, \#} \BC_{Z}$, which is $\Gamma_{Z}(\{(0, 0)\} \times \R^{n-1}).$ Thus for any $y \in \R^{n-1}$ with $|y| < 1/2$, 
$D\phi_{1}(y)(\{(0, 0)\} \times \R^{n-1}) = \pi_{e_{1}} \, \Gamma_{\phi(y)}(\{(0, 0)\} \times \R^{n-1})$ and $D\phi_{2}(y)(\{(0, 0)\} \times \R^{n-1}) = \pi_{e_{2}} \, \Gamma_{\phi(y)}(\{(0, 0)\} \times \R^{n-1}),$ where $\pi_{e_{i}}$ denotes the orthogonal projection 
$X \mapsto (X \cdot e_{i})e_{i},$ $X \in \R^{n+1}$. 

On the other hand, by (\ref{E:rotate-est}) and (\ref{rotation-equality}), we have that for any $j=3, 4, \ldots, n+1$, if we let $a_j = \sum_{\ell=3}^{n+1}a_{j\ell}e_\ell\in \{(0,0)\}\times\R^{n-1}$ be such that
$$\widetilde{\Gamma}_{\widetilde{Z}}(a_j) = \left[\left(\Gamma_k^{Z_2}\right)^{-1}\circ\Gamma_{Z_1}\right](e_j)$$
then we have, by the triangle inequality, $\dist(\pi_{e_1}\Gamma_{Z_1}(e_j), \pi_{e_1}\Gamma_{Z_2}(\{(0,0)\}\times\R^{n-1})) \leq |e_1\cdot\Gamma_{Z_1}(e_j) - e_1\cdot \Gamma_{Z_2}(a_j)| = |\Gamma^{-1}_{Z_2}(e_1)\cdot ((\Gamma_{Z_2}^{-1}\circ\Gamma_{Z_1})(e_j) - a_j)| \leq |(\Gamma_{Z_2}^{-1}(e_1)-e_1)\cdot(\Gamma^{-1}_{Z_2}\circ\Gamma_{Z_1}(e_j)-a_j)| + |e_1\cdot(\Gamma_{Z_2}^{-1}\circ\Gamma_{Z_1}(e_j)-a_j)| \leq CQ_{V,\BC}|\Gamma_{Z_2}^{-1}\circ\Gamma_{Z_1}(e_j) - a_j| + |e_1\cdot(\Gamma_{Z_2}^{-1}\circ\Gamma_{Z_1}(e_j)-a_j)|$, where in the last inequality we have used (\ref{E:rotateZ}). To bound the first term note that $|\Gamma^{-1}_{Z_2}\circ\Gamma_{Z_1}(e_j)-a_j| \leq |(\Gamma_k^{Z_2})^{-1}\circ\Gamma_{Z_1}(e_j)-a_j| + |(\Gamma_{Z_2}^{-1}-(\Gamma_k^{Z_2})^{-1})(\Gamma_{Z_1}(e_1))| \leq |\widetilde{\Gamma}_{\widetilde{Z}}(a_j) - a_j| + |\Gamma^{-1}_{Z_2}-(\Gamma^{Z_2}_k)^{-1}||\Gamma_Z(e_1)| \leq C\sigma_k^\alpha$, using (\ref{E:rotate-est}), (\ref{fine-excess-est}), (\ref{fine-coarse-Z}), and (\ref{E:rotateZ}). Thus now we have
$$\dist(\pi_{e_1}\Gamma_{Z_1}(e_j),\pi_{e_1}\Gamma_{Z_2}(\{(0,0)\}\times\R^{n-1})) \leq C|Z_1-Z_2|^\alpha Q_{V,\BC} + |e_1\cdot (\Gamma_{Z_2}^{-1}\circ\Gamma_{Z_1}(e_j)-a_j)|$$
for some $C = C(n,Q,\alpha)$. To deal with the remaining term note that, by the triangle inequality, $|e_1\cdot(\Gamma_{Z_2}^{-1}\circ\Gamma_{Z_1}(e_j)-a_j)| \leq |e_1\cdot((\Gamma_k^{Z_2})^{-1}\circ\Gamma_{Z_1}(e_j)-a_j)| + |e_1\cdot(\Gamma_{Z_2}^{-1}\circ\Gamma_{Z_1}(e_j)-(\Gamma_k^{Z_2})^{-1}\circ\Gamma_{Z_1}(e_j))| \leq |e_1\cdot(\widetilde{\Gamma}_{\widetilde{Z}}(a_j)-a_j)| + |(\Gamma_{Z_1}^{-1}\circ\Gamma_{Z_2}(e_1) - \Gamma_{Z_1}^{-1}\circ\Gamma_{k}^{Z_2}(e_1))\cdot e_j| \leq |e_1\cdot\widetilde{\Gamma}_{\widetilde{Z}}(a_j)| + |(\Gamma^{-1}_{Z_1}\circ(\Gamma_{Z_2}-\Gamma^{Z_2}_k)(e_1)| \leq |(e_1-\widetilde{\Gamma}_{\widetilde{Z}}(e_1))\cdot\widetilde{\Gamma}_{\widetilde{Z}}(a_j)| + |\Gamma_{Z_1}^{-1}||\Gamma_{Z_2}(e_1)-\Gamma_k^{Z_2}(e_1)| \leq |e_1-\widetilde{\Gamma}_{\widetilde{Z}}(e_1)| + C|\Gamma_{Z_2}(e_1)-\Gamma_k^{Z_2}(e_1)|$, where we have used that $a_j\cdot e_1 = 0$, $|a_j| =1$, and (\ref{E:rotateZ}) to bound $|\Gamma_{Z_1}^{-1}|$. Thus using (\ref{E:rotate-est}), (\ref{fine-excess-est}), (\ref{fine-coarse-Z}), and (\ref{E:rotateZ}) we see $|e_1\cdot(\Gamma_{Z_2}^{-1}\circ\Gamma_{Z_1}(e_j)-a_j)| \leq C\sigma_k^\alpha Q_{V,\BC}$, and thus
$$\dist(\pi_{e_1}\Gamma_{Z_1}(e_j),\pi_{e_1}\Gamma_{Z_2}(\{(0,0)\}\times\R^{n-1}))\leq C|Z_1-Z_2|^\alpha Q_{V,\BC}$$
for each $j=3,4,\dotsc,n+1$, for some $C = C(n,Q,\alpha)$. Similar reasoning, using the estimate $|\widetilde{\Gamma}_{\widetilde{Z}}(e_2)-e_2|\leq C\hat{E}_{\widetilde{V}_{\widetilde{Z}}}^{-1}Q_{\widetilde{V}_{\widetilde{Z}},\BC^{\prime\prime}}$, gives
$$\dist(\pi_{e_2}\Gamma_{Z_1}(e_j),\pi_{e_2}\Gamma_{Z_2}(\{(0,0)\}\times\R^{n-1})) \leq C|Z_1-Z_2|^\alpha\hat{E}_{V}^{-1}Q_{V,\BC}$$
for $j=3,4,\dotsc,n+1$. These bounds readily give the desired H\"older continuity estimates for $D\phi_1$ and $D\phi_2$. To see the supremum bound on just $D\phi_1$, $D\phi_2$, note similarly to the above that for each such $\phi(y) = Z$ and each $j=3,\dotsc,n+1$,
$$|e_1\cdot \Gamma_{Z}(e_j)| = |(e_1-\Gamma_Z(e_1))\cdot \Gamma_Z(e_j)| \leq |e_1-\Gamma_Z(e_1)| \leq CQ_{V,\BC}$$
and the bound on $D\phi_2$ follows in the same way. The supremum bounds on $\phi_1,\phi_2$ follow immediately from \cite[Corollary 10.2]{wickramasekera2014general}. This completes the proof of the theorem.
\end{proof}

\subsection{Coarse blow-ups of varifolds in $\S_{Q}$ and their initial properties}\label{initial-cb} 

We now recall the definition of the coarse blow-up class, $\FB_Q$, as defined in \cite[Section 5]{wickramasekera2014general}. 

Let $(V_k)_k$ be a sequence of $n$-dimensional stationary integral varifolds on $B^{n+1}_2(0)$ such that for each $k=1,2,3,\dotsc$:
\begin{equation}\tag{$\star$}
(\w_n 2^n)^{-1}\|V_k\|(B^{n+1}_2(0))<Q+1/2;\ \ \ \ Q-1/2 \leq \w_n^{-1}\|V_k\|(\R\times B^n_1(0))<Q+1/2.
\end{equation}
Assume also that $\hat{E}_k\to 0$, where $\hat{E}_k$ is the one-sided height excess of $V_k$ relative to $\{0\}\times\R^n$, i.e.
$$\hat{E}_k^2 \equiv \hat{E}_{V_k}^2 = \int_{\R\times B^n_1(0)}|x^1|^2\ \ext\|V_k\|(X)\;,$$
where $X =(x^1,x^2,\dotsc,x^{n+1})$. Let $\sigma\in (0,1)$. By applying \cite[Corollary 3.11]{almgrenalmgren}, for all sufficiently large $k$, there exist Lipschitz functions $u^j_k:B_\sigma^n(0)\to \R$, $j=1,\dotsc,Q$, with $u^1_k\leq u^2_k\leq \cdots\leq u^Q_k$ and $\Lip(u^j_k)\leq 1/2$ for each $j\in \{1,2,\dotsc,Q\}$ and such that
\begin{equation}\label{good-set}
\spt\|V_k\|\cap (\R\times(B_\sigma\setminus\Sigma_k)) = \bigcup_{j=1}^{Q}\graph(u^j_k)\cap (\R\times(B_\sigma\setminus\Sigma_k)\;,
\end{equation}
where for each $k$, $\Sigma_k\subset B_\sigma$ is a measurable subset with 
\begin{equation}\label{badset-est}
\H^n(\Sigma_k) + \|V_k\|(\R\times\Sigma_k) \leq C\hat{E}^2_k
\end{equation}
for some $C = C(n,Q,\sigma)$; we set $\Omega_k:= B^n_1(0)\setminus \Sigma_k$. Now set $v^j_k(x):= \hat{E}^{-1}_k u^j_k(x)$ for $x\in B_\sigma$, and write $v_k = (v^1_k,v_k^2,\dotsc,v_k^{Q})$. Then $v_k$ is Lipschitz on $B_\sigma$, and moreover it can be readily checked (see \cite[inequalities~(5.8) \& (5.9)]{wickramasekera2014general}) that $\|v_k\|_{W^{1,2}(B_\sigma)} \leq C$ for some $C = C(n,Q,\sigma)$. Thus as $\sigma\in (0,1)$ is arbitrary, we can apply Rellich's compactness theorem and a diagonal argument to obtain a function $v\in W^{1,2}_{\text{loc}}(B_1;\R^Q)\cap L^2(B_1;\R^Q)$ and a subsequence $(k_j)$ of $(k)$ such that $v_{k_j}\to v$ as $j\to\infty$, strongly in $L^2(B_\sigma;\R^Q)$ and weakly in $W^{1,2}(B_\sigma;\R^Q)$, for every $\sigma\in (0,1)$.

\begin{defn}
	Let $v\in W^{1,2}_{\text{loc}}(B_1;\R^Q)\cap L^2(B_1;\R^Q)$ correspond, in the manner described above, to (a subsequence of) a sequence $(V_k)_k$ of stationary integral $n$-varifolds of $B_2^{n+1}(0)$ satisfying $(\star)$ and with $\hat{E}_{V_k}\to 0$. We call such a $v$ a \textit{coarse blow-up} of the sequence $(V_k)_k$.
\end{defn}

\begin{defn}
	We write $\FB_Q$ for the collection of all coarse blow-ups of sequences of varifolds $(V_k)_k\subset \S_Q$ satisfying $(\star)$ and for which $\hat{E}_{V_k}\to 0$.
\end{defn}

In the same way as \cite[Section 8]{wickramasekera2014general}, we can show that $\FB_{Q}$ satisfies the following properties, where (in $(\FB4\textnormal{I})$) the constant $C = C(n, Q) \in (0, \infty)$ depends only on $n$ and $Q$:
\begin{enumerate}
	\item [$(\FB1)$] $\FB_Q\subset W^{1,2}_{\text{loc}}(B_1;\R^Q)\cap L^2(B_1;\R^Q)$;
	\item [$(\FB2)$] If $v\in \FB_Q$, then $v^1\leq v^2\leq \cdots \leq v^Q$ a.e.\ in $B_{1}$;
	\item [$(\FB3)$] If $v\in \FB_Q$, then $\Delta v_a = 0$ in $B_1$, where $v_a = Q^{-1}\sum^Q_{j=1}v^j$ a.e.\ in $B_{1}$;
	\item [$(\FB4)$] For each $v\in \FB_Q$ and $z\in B_1$, either $(\FB4\textnormal{I})$ or $(\FB4\textnormal{II})$ below is true:
	\begin{enumerate}
		\item [$(\FB4\textnormal{I})$] The \textit{Hardt-Simon inequality}
		$$\sum^Q_{j=1}\int_{B_{\rho/2}(z)} R_z^{2-n}\left(\frac{\del((v^j-v_a(z))/R_{z})}{\del R_z}\right)^2 \leq C\rho^{-n-2}\int_{B_\rho(z)}|v-\ell_{v,z}|^2$$
		holds for each $\rho\in (0,\frac{3}{8}(1-|z|)]$, where $R_z(x) = |x-z|$, $\ell_{v,z}(x) = v_a(z) + Dv_a(z)\cdot (x-z)$ and $v-\ell_{v,z} = (v^1-\ell_{v,z},v^2-\ell_{v,z},\dotsc,v^Q-\ell_{v,z})$;
		\item [$(\FB4\textnormal{II})$] There exists $\sigma = \sigma(z)\in (0,1-|z|]$ such that (after redefining $\left.v \right|_{ B_{\sigma}(z)}$ on a set of measure zero) $\Delta v = 0$ in $B_\sigma(z)$;
	\end{enumerate}
	\item [$(\FB5)$] If $v\in \B_{Q}$, then
	\begin{enumerate}
		\item [$(\FB5\textnormal{I})$] $v_{z,\sigma}(\cdot) \equiv \|v(z+\sigma(\cdot))\|^{-1}_{L^2(B_1(0))}v(z+\sigma(\cdot))\in \B_{Q}$ for each $z\in B_1$ and $\sigma\in (0,\frac{3}{8}(1-|z|)]$ such that $v\not\equiv 0$ in $B_\sigma(z)$;
		\item [$(\FB5\textnormal{II})$] $v\circ\gamma\in\B_{Q}$ for each orthogonal rotation $\gamma$ of $\R^n$;
		\item [$(\FB5\textnormal{III})$] $\|v-\ell_v\|^{-1}_{L^2(B_1(0))}(v-\ell_v)\in \B_{Q}$ whenever $v-\ell_v\not\equiv0$ in $B_1$, where $\ell_v(x) = v_a(0) + Dv_a(0)\cdot x$ for $x\in \R^n$ and $v-\ell_v = (v^1-\ell_v,v^2-\ell_v,\dotsc,v^Q-\ell_v)$;
	\end{enumerate}
	\item [$(\FB6)$] If $(v_k)_{k=1}^\infty\subset \B_{Q}$ then there exists a subsequence $(k')$ of $(k)$ and a function $v\in \B_{Q}$ such that $v_{k'}\to v$ locally in $L^2(B_1)$ and locally weakly in $W^{1,2}(B_1)$.
\end{enumerate}

The only slight difference to note in proving $(\FB1)-(\FB6)$ compared to that seen in \cite{wickramasekera2014general} is in the proof of $(\FB4)$. If for some $v\in \FB_Q$ and $z\in B_1$ we have that $(\FB4\text{I})$ does not hold, then in the same way as \cite[Equation (8.8)]{wickramasekera2014general} we can show that there is $\sigma_1>0$ such that, if $(V_k)_k\subset\S_Q$ is a sequence of stationary integral varifolds generating $v$, then for all sufficiently large $k$:
$$Z\in \spt\|V_k\|\cap (\R\times B_{\sigma_1}(z))\ \ \Longrightarrow\ \ \Theta_{V_k}(Z)<Q.$$
In particular, there are no classical singularities of density $\geq Q$ in $\spt\|V_k\|\cap (\R\times B_{\sigma_1}(z))$. But by assumption $(S3)_Q$ there are no classical singularities in $V_k$ of density $<Q$ in $\spt\|V_k\|\cap (\R\times B_{\sigma_1}(z))$. Thus there are no classical singularities in $\spt\|V_k\|\cap (\R\times B_{\sigma_1}(z))$ for all $k$ sufficiently large, and so we can apply the sheeting theorem \cite[Theorem 3.3]{wickramasekera2014general} and standard elliptic PDE theory to conclude that for all $k$ sufficiently large
$$V_k\res (\R\times B_{\sigma_1/2}(z)) = \sum^Q_{j=1}|\graph(u^j_k)|\;,$$
where $u^j_k:B_{\sigma_1/2}(z)\to \R$ are $C^2$ functions satisfying
$$\sup_{B_{\sigma_1/2}(z)}\left(\sum^Q_{j=1}|Du^j_k| + |D^2 u^j_k|\right) \leq C\hat{E}_k$$ 
and solving the minimal surface equation on $B_{\sigma_1/2}(z)$, where $C = C(n,Q,\sigma_1)\in (0,\infty)$. This readily shows that $(\FB4\text{II})$ holds, with $\sigma = \sigma_1/2$.

{\bf Notation:} For $v\in \FB_Q$ we write $$\Gamma_{v}^{\textnormal{HS}} = \{y \in B_{1} \, : \, (\FB4\textnormal{I}) \text{ holds with $z = y$}\}\text{,} \;\;\;\; \text{and}\;\;\;\; \Gamma_v = \Gamma_v^{\textnormal{HS}}\setminus\Omega_v\;,$$
where
$$\Omega_v = \{x\in B_1:\text{there exists }\rho\in (0,1-|x|]\text{ such that } v^1=v^2=\cdots=v^Q\text{ a.e. in }B_\rho(x)\}.$$
Note that it follows from property $(\FB4)$ that $\Gamma_v^{\textnormal{HS}}$, $\Gamma_{v}$, are relatively closed subsets of $B_{1}^{n}(0).$

\noindent
{\bf Remark:}  We note the following: let $v \in \FB_{Q}$ be the coarse blow-up of a sequence $(V_{k})_{k} \subset \S_{Q}$ 
with $\Theta_{V_{k}}(0) \geq Q$ for each $k$. Then  $v_{a}(0) = 0$ and 
$(\FB4\textnormal{I})$ holds with $z = 0$. 

To see this, note that by exactly the same argument leading to \cite[inequality~(8.9)]{ wickramasekera2014general}, we obtain that for any $\rho \in (0, 3/8]$, any $\sigma \in (0, \rho/4)$ and sufficiently large $k$, 
\begin{align*}
\left(\frac{\sigma^{2}}{\delta_{k}^{2} + \sigma^{2}}\right)^{\frac{n+2}{2}} \sum^Q_{j=1}\int_{B_{\rho/2}(0) \setminus (B_{\sigma}(0) \cup \Sigma_{k})} 
R^{2-n}\left(\frac{\del(u_{k}^{j}/R)}{\del R}\right)^2\\ 
&\hspace{-1in}\leq C_{2} \rho^{-n-2}\int_{\R \times B_\rho(0)}|x^{1}|^{2} \, \ext\|V_{k}\|,
\end{align*}
where $u_{k}$, $\Sigma_{k}$ are as in (\ref{good-set}) and (\ref{badset-est}), 
$R(x) = |x|,$ $\delta_{k} \to 0$ and $C_{2} = C_{2}(n, Q)$. 
Dividing this inequality  by $\hat{E}_{k}^{2}$ and letting $k \to \infty$, 
and then letting $\sigma \to 0$, we see that 
$$\sum^Q_{j=1}\int_{B_{\rho/2}(0)} R^{2-n}\left(\frac{\del(v^j/R)}{\del R}\right)^2 \leq C\rho^{-n-2}\int_{B_\rho(0)}|v|^{2}.$$ By the triangle inequality, this in particular implies that 
$\int_{B_{\rho/2}(0)} R^{2-n}\left(\frac{\del(v_{a}/R)}{\del R}\right)^2 < \infty$ which in turn readily implies, since $v_{a}$ is $C^{1}$, 
that $v_{a}(0) = 0.$  To see that $(\FB4\textnormal{I})$ with $z=0$ holds, note first that if $v^{j}(x) = Dv_{a}(0)\cdot x$ for all $j=1, 2, \ldots, Q$ and all $x \in B_{1}$, then $(\FB4\textnormal{I})$ with $z=0$ holds with both sides equal to zero. Otherwise, let $L(x) = Dv_{a}(0) \cdot x$ and note that for each fixed sufficiently large $\sigma \in (0, 1)$, choosing for each $k$ an appropriate orthogonal rotation $\Gamma_{k}$ of $\R^{n+1}$ that takes the hyperplane ${\rm graph} \, \hat{E}_{k} L$ to $\{0\} \times \R^{n}$,  and passing to a subsequence of $(k)$ without relabelling, we have that the coarse blow-up of the sequence $(W_{k})_k$ where $W_{k} = \Gamma_{k \, \#} \, \eta_{0, \sigma \, \#} \, V_{k}$ is $w(x) \equiv \|v(\sigma(\cdot)) - \sigma L\|_{L^{2}(B_{1})}^{-1}(v(\sigma x) -\sigma L(x)),$ and moreover, we have $\Theta_{W_{k}}(0) \geq Q$. So we can apply the above inequality with $w$ in place of $v,$ and then let $\sigma \to 1$ to see that $(\FB4\textnormal{I})$ holds for $v$ with $z= 0$. 

In the present setting, property $(\FB7)$ of \cite[Section 4]{wickramasekera2014general} is no longer true;
instead, $(\FB7)$ is replaced by an $\epsilon$\textit{-regularity property} (Theorem~\ref{thm:B7} below) for the coarse blow-ups. This relaxation of property $(\FB7)$ allows coarse blow-ups to contain branch points, making their analysis  considerably more involved than that seen in \cite[Section 4]{wickramasekera2014general}. 
In particular, for this purpose we shall use Almgren's \textit{frequency function} which we show is monotone subject to a regularity assumption on the coarse blow-ups (and ultimately unconditionally, once we establish regularity). 

We will also need the following preliminary result concerning the structure of coarse blow-ups in $\FB_Q$ which are close to elements in $\CC_Q$; this result is a direct consequence of \cite[Lemma~9.1]{wickramasekera2014general}. We shall improve its conclusions in Theorem~\ref{thm:B7} below. 

\begin{prop}\label{prop:sheets}
	 Let $v\in \FB_Q$ and suppose that $0\in \Gamma_v$ and $\|v\|_{L^2(B_1(0))} = 1$. Let $\psi:\R^n\to \A_Q$ be such that $\psi_a\equiv 0$, $\mathbf{v}(\psi)\in \CC_Q,$ the spine 
	 $S(\psi) = \{(0, 0)\} \times \R^{n-1},$ and 
	$$\int_{B_1(0)}\G(v,\psi)^2<\epsilon,$$
	where $\epsilon \in (0, 1)$. Then:
	\begin{enumerate}
		\item[(i)] there is $\epsilon_{0} = \epsilon_{0}(n, Q)$ such that if the above holds with $\epsilon \leq \epsilon_{0}$ and if we write $\psi|_{\R^n_+} = \sum^Q_{\alpha=1}\llbracket h^\alpha \rrbracket$ and $\psi|_{\R^n_-} = \sum^Q_{\alpha=1}\llbracket g^\alpha \rrbracket$ where $h^\alpha:\R^n_+\to \R$, $g^\alpha:\R^n_-\to \R$ are of the form $h^\alpha(x^2,\dotsc,x^{n+1}) = \lambda^\alpha x^2$ and $g^\alpha(x^2,\dotsc,x^{n+1}) = \mu^\alpha x^2$, where $\lambda^1\leq \cdots\leq \lambda^Q$ and $\mu^1\leq \cdots\leq \mu^Q$, then $|\lambda^1-\lambda^Q|\geq C$ and $|\mu^1-\mu^Q|\geq C$ for some $C = C(n,Q)\in (0,\infty)$;
		\item[(ii)] for each $\tau\in (0,1/2)$, $\sigma\in (1/2,1)$, there is $\epsilon_{0} = \epsilon_{0}(n, Q, \tau, \sigma) \in (0, 1)$ such that if the above holds with $\epsilon \leq \epsilon_{0}$ then for any sequence of varifolds $(V_j)_j\subset\S_Q$ that generates $v$, we have that for all $j$ sufficiently large,
		$$(\R\times B_{\sigma})\cap\{|x^2|>\tau\}\subset \{Z:\Theta_{V_j}(Z)<Q\};$$
		\item[(iii)] for each $\tau\in (0,1/2)$, $\sigma\in (1/2,1)$, there is $\epsilon_{0} = \epsilon_{0}(n, Q, \tau, \sigma) \in (0, 1)$ such that if the above holds with 
		$\epsilon \leq \epsilon_{0}$ then $v$ is harmonic on $B_{(1+\sigma)/2}\cap \{|x^2|>\tau/2\}$ and, with the notation as in \textnormal{(i)},  
		$$\sum_{\alpha=1}^{Q} \left(\sup_{x \in \{x^{2} > \tau\} \cap B_{\sigma}} \, |v^{\alpha}(x) - h^{\alpha}(x)|^{2} + \sup_{x \in \{ x^{2} < -\tau\} \cap B_{\sigma}} \, 
		|v^{\alpha}(x) - g^{\alpha}(x)|^{2} \right)\leq C\int_{B_1(0)}\G(v,\psi)^2,$$ 
		where $C = C(n, Q, \tau, \sigma)$.
	\end{enumerate}
\end{prop}

\begin{proof}
	To see (i) we argue by contradiction. If (i) were not true then for each $k= 1, 2, 3, \ldots$ there is $v_k\in \FB_Q$ with $0\in \Gamma_{v_k}$, $\|v_k\|_{L^2(B_1(0))} = 1,$ and 
	$\psi_k:\R^n\to \A_Q$ with $(\psi_k)_a\equiv 0$ and $\mathbf{v}(\psi_{k})\in \CC_Q$ such that, if $h^\alpha_k, g^\alpha_k, \lambda^\alpha_k, \mu^\alpha_k$ are the functions and coefficients as in (i) for $\psi_k$:
	$$\int_{B_1(0)}\G(v_k,\psi_k)^2 \to 0\ \ \ \ \text{and}\ \ \ \ \min\{|\lambda^1_k-\lambda^Q_k|,\ |\mu^1_k-\mu^Q_k|\}\to 0$$
	as $k\to\infty$. We may then pass to a subsequence to ensure $\psi_k\to \psi$ for some $\psi\in \CC_Q$; we know that $\mathbf{v}(\psi)$ is not a supported on a single hyperplane as $\|\psi\|_{L^2(B_1(0))} = 1$ and $\psi_a\equiv 0$. Without loss of generality we can assume that $|\lambda^1_k-\lambda^Q_k|\to 0$, and so if $\lambda^1,\dotsc,\lambda^Q$ are the corresponding quantities for $\psi$, we see $\lambda^1 =\cdots= \lambda^Q = 0$, i.e. on $\R^n_+$, $\mathbf{v}(\psi)$ coincides with a multiplicity $Q$ half-hyperplane.
	
	Since $v_k\in \FB_Q$ we may find a sequence $(V_{k,\ell})_\ell \subset \S_Q$ with $V_{k,\ell}\weakly Q|\{0\}\times\R^n|$ in $\R\times B_1(0)$ as $\ell\to \infty$. Since $0\in \Gamma_{v_k}$ we have  that for each $k$ and all $\ell$ sufficiently large, 
	$\Theta_{V_{k,\ell}}(Z_{k,\ell})\geq Q$ for some sequence of points $(Z_{k,\ell})_{\ell}$ with $Z_{k,\ell}\to 0$ as $\ell\to \infty$ (see the argument establishing property $(\FB4),$ given immediately following the list of properties $(\FB1)-(\FB6)$). By translating and rescaling $V_{k,\ell}$ appropriately, and relabelling the indices, we may assume without loss of generality that $\Theta_{V_{k,\ell}}(0)\geq Q$ for each $k$ and $\ell$; this ensures by the monotonicity formula that the mass bounds $(\star)$ in Section \ref{initial-cb} hold for the $V_{k,\ell}$. Write $u_{k,\ell}$ for the $Q$-valued Lipschitz functions approximating $V_{k,\ell}$ (on $B_{\sigma_\ell}(0)$ with $\sigma_\ell\uparrow 1$) which generate $v_k$ after scaling by the height excess $\hat{E}_{V_{k,\ell}}$. Passing to a subsequence we can ensure that for all $k$
	$$\int_{B_1(0)}\G(v_k,\psi)^2< 1/k$$
	and additionally, for any given $\sigma\in (0,1)$ and each $k$, we can find some positive integer $\ell_k$ such that
	$$\int_{B_\sigma(0)}\G(\hat{E}^{-1}_{V_{k,\ell_{k}}}u_{k,\ell_{k}},v_k)^2 < 1/k,$$
	whence
	$$\int_{B_\sigma(0)}\G(\hat{E}^{-1}_{V_{k,\ell_k}}u_{k,\ell_k},\psi)^2<2/k.$$
	In particular we have that $\psi$ is the coarse blow-up of $V_{k,\ell_k}$, and so $\psi\in \FB_Q$. But then this is a direct contradiction to 
	\cite[Lemma~9.1]{wickramasekera2014general} taken with $q = Q$; note that even though the statement of \cite[Lemma~9.1]{wickramasekera2014general} assumes $\psi$ (denoted $v^{\star}$ in the notation of   \cite[Lemma~9.1]{wickramasekera2014general}) is a coarse blow-up of a sequence $V_{j}$ of stable codimension 1 integral varifolds with \emph{no} classical singularities, its proof requires only that $V_{j}$ have no classical singularities of density $< Q$, i.e.\ that $V_{j} \in \S_{Q}$.
	
	To prove (ii) we follow the same contradiction argument, now using the result from (i). In particular if (ii) were not true, by the same argument as for (i) we can find sequences $(v_k)_k\subset \FB_Q$ and $(\psi_k)_k$ with $0\in \Gamma_{v_k}$, $\|v_k\|_{L^2(B_1(0))}=1$, $(\psi_k)_a\equiv 0$ and $\mathbf{v}(\psi_{k})\in \CC_Q$ which obey $\int_{B_1(0)}\G(v_k,\psi_k)^2\to 0$. We may pass to a subsequence to ensure that $\psi_k\to \psi$ for some $\psi\in \CC_Q$, where from (i) we now know that if $\lambda^\alpha,\mu^\alpha$ are as in (i), then $|\lambda^1-\lambda^Q|\geq C$ and $|\mu^1-\mu^Q|\geq C$ for some $C = C(n,Q)\in (0,\infty)$. Following the construction in (i), under the assumption that (ii) fails, we find that $\psi$ is the coarse blow-up of a sequence $(W_j)_j\subset \S_Q$ such that for each $j$ we have a point $Z_j\in (\R\times B_{\sigma})\cap \{|x^2|>\tau\}$ with $\Theta_{W_j}(Z_j)\geq Q$. But then by the argument establishing property $(\FB4)$ we see that we must have that $(\FB4\text{I})$ holds, with $\psi$ in place of $v$, for some $z \in\overline{B}_{\sigma} \cap \{|x^2|\geq\tau\}$. Hence $\psi^1(z) = \cdots = \psi^Q(z) = 0$, whence $\min\{|\lambda^1-\lambda^Q|, |\mu^1-\mu^Q|\} = 0$, a contradiction. Thus (ii) must hold as claimed. 

To see (iii), note that if $\epsilon = \epsilon(n, Q, \tau, \sigma)$ is sufficiently small, then from (ii) it follows that $V_j$ has no classical singularities in the region $(\R\times B_{(3+\sigma)/4})\cap\{|x^2|>\tau/4\}$, and so we may apply Theorem~\ref{sheeting} to get that
$$V_j\res\left((\R\times B_{(1+\sigma)/2})\cap \{x^2<-\tau/2\}\right) = \sum^Q_{\alpha=1}|\graph \, u^{(\alpha, \; -)}_j|$$
and
$$V_j\res\left((\R\times B_{(1+\sigma)/2}) \cap \{x^2>\tau/2\}\right) = \sum^Q_{\alpha=1}|\graph \, u^{(\alpha, +)}_j|,$$
where $u^{(1, \; -)}_j\leq \cdots \leq u^{(Q, \; -)}_j$ and $u^{(1, \; +)}_j \leq \cdots \leq u^{(Q, \; +)}_j$ are $C^{2}$ functions on $B_{(1+\sigma)/2}\cap \{x^2<-\tau/2\}$ and $B_{(1+\sigma)/2}\cap \{x^2>\tau/2\}$ respectively, solving the minimal surface equation and satisfying 
$$\|u^{(\alpha, \; -)}_{j}\|_{C^{1}(B_{(1+\sigma)/2} \cap \{x^{2} < -\tau/2\})} +  \|u^{(\alpha, \; +)}_{j}\|_{C^{1}(B_{(1+\sigma)/2} \cap \{x^{2} > \tau/2\})} \leq C\hat{E}_j\;,$$
where $C = C(n, Q, \tau, \sigma)$. These and consequent higher derivative estimates imply that $\hat{E}_{j}^{-1}u^{(\alpha, \; -)}_{j} \to v^{\alpha}$ on $B_{(1+\sigma)/2} \cap \{x^{2} < -\tau/2\}$ and 
$\hat{E}_{j}^{-1}u^{(\alpha, \; +)}_{j} \to v^{\alpha}$ on  $B_{(1+ \sigma)/2} \cap \{x^{2} > \tau/2\},$ where the convergence is in the $C^{2}$ norm on the respective domains. It follows that 
$v^{\alpha}$ is harmonic in $B_{(1+\sigma)/2} \cap \{|x^{2}| > \tau/2\}$; the estimate asserted in (iii) is the result of a standard estimate for harmonic functions and the fact that 
the multi-valued distance ${\mathcal G}(a, b)$ is equal to the ``ordered'' distance $|a - b| = \sqrt{\sum_{\alpha=1}^{Q} |a^{\alpha} - b^{\alpha}|^{2}}$ for $a = \sum_{\alpha=1}^{Q} \llbracket 
a^{\alpha} \rrbracket \in {\mathcal A}_{Q}(\R)$, $b = \sum_{\alpha=1}^{Q} \llbracket b^{\alpha} \rrbracket \in {\mathcal A}_{Q}(\R)$ whenever the labeling is so chosen that $a^{1} \leq a^{2} \leq \cdots \leq a^{Q}$ and $b^{1} \leq b^{2} \leq \cdots \leq b^{Q}$. 
\end{proof}

\part{Proof of Theorem~A}\label{main-thm-proof}
\setcounter{section}{3}
\setcounter{subsection}{0}
\setcounter{equation}{0}
  We now begin the proof of Theorem~\ref{thm:A}. A central step in the proof is Theorem~\ref{coarse_reg} below, which gives a uniform decay estimate for the coarse blow-ups constructed in Section~\ref{initial-cb}. Our first goal is to establish this estimate. To do this, we start by deriving, in part from the properties 
  $(\FB1)$-$(\FB6)$ recorded in Section~\ref{initial-cb}, the following key additional properties of the coarse blow-ups: 
(i) H\"older continuity (Lemma~\ref{continuity}), and a homogeneity continuation property (Lemma~\ref{unique-cont}); (ii) an $\epsilon$-regularity property (Theorem~\ref{thm:B7}); (iii) the squash inequality (Lemma~\ref{squash}); (iv) an energy non-concentration estimate (Lemma~\ref{noncon2}), and (v) the squeeze identity and frequency monotonicity under the assumption of generalised-$C^{1}$ regularity (Lemma~\ref{squeeze} and Theorem~\ref{frequency}).

\subsection{A continuity estimate for coarse blow-ups}

We have the following  as a direct consequence of property $(\FB4)$ (in subsection~\ref{initial-cb}): 

\begin{lemma}[Coarse blow-up continuity estimate]\label{continuity} 
If $v \in \FB_{Q}$ then (after re-defining $v$ on a set of measure zero) we have that for any $\beta \in (0, 1)$, $\left.v\right|_{B_{\gamma}} \in C^{0, \beta}(\overline{B}_{\gamma}; {\mathcal A}_{Q}(\R))$ for each $\gamma \in (0, 1)$, and $v$ satisfies the estimate 
$$\sup_{B_{\gamma}} \, |v|^{2} + \sup_{x_{1}, x_{2} \in B_{\gamma}, \, x_{1} \neq x_{2}} \, \frac{|v(x_{1}) - v(x_{2})|^{2}}{|x_{1} - x_{2}|^{2\beta}} \leq C \int_{B_{1}} |v|^{2}\;,$$
where $C = C(n, Q, \beta, \gamma) \in (0, \infty);$ moreover, for each $z \in \Gamma_{v}^{\textnormal{HS}}$, we have that $v^{\alpha}(z) = v_{a}(z)$ for each $\alpha \in \{1, 2, \ldots, Q\}$.
\end{lemma} 

\begin{proof}
For $\delta > 0$, let $\eta_{\delta} \in C^1([0,\infty))$ be a non-decreasing function such that $\eta_{\delta}(t) = 0$ for $t \in [0,\delta/2]$, $\eta_{\delta}(t) = 1$ for $t \in [\delta,\infty)$ and 
$|D\eta_{\delta}| \leq 3/\delta$ for all $t \in [0,\infty)$. Let $\varphi \in C^{1}([0, \infty))$ be such that $\varphi(t) = 1$ for $t \in [0, 1/4]$, $\varphi(t) = 0$ for $t \in [3/8, \infty)$ and 
$|\varphi^{\prime}(t)| \leq 10$ for all $t \in [0, \infty)$. Note that since $v \in W^{1,2}_{\rm loc}(B_1(0);\mathcal{A}_Q(\R))$ we have that $|v - v_{a}(z)|^2 \in W^{1,1}_{\text{loc}}(B_1(0))$  for any fixed $z \in B_{1}$.   For any fixed $\beta >0,$ fixed $z \in B_{1/2}$ such that $(\FB4\textnormal{I})$ holds, and any fixed $\rho \in (0, 1/4)$, set $X^i(x) = \varphi(R_{z}/\rho)^2 \eta_{\delta}(R_{z}) R_{z}^{-n+\beta-2} |v(x) - v_{a}(z)|^2 (x^i - z^{i})$  where $R_{z} = |x - z|$. Then by the divergence theorem we have $\int_{\mathbb{R}^n} D_i X^i = 0$,  whence 
\begin{align}\label{cont-1} 
	&\beta \int \varphi(R_{z}/\rho)^2 \eta_{\delta}(R_{z}) R_{z}^{-n+\beta-2} |v - v_{a}(z)|^2 
	= -\int \varphi(R_{z}/\rho)^2 \eta_{\delta}(R_{z}) R_{z}^{1-n+\beta} \frac{\partial  (R_{z}^{-2} |v - v_{a}(z)|^2)}{\partial \, R_{z}} \nonumber\\
	&\hspace{2.4in} - 2 \int \rho^{-1}\varphi(R_{z}/\rho) \varphi'(R_{z}/\rho) \eta_{\delta}(R_{z}) R_{z}^{-n+\beta -1} |v - v_{a}(z)|^2 \nonumber\\
	&\hspace{3in}- \int \varphi(R_{z}/\rho)^2 \eta'_{\delta}(R_{z}) R_{z}^{-n+\beta-1} |v - v_{a}(z)|^2. 
\end{align}
Noting that $\frac{\partial \, (R_{z}^{-2} |v - v_{a}(z)|^2)}{\partial R_{z}} = \sum_{\alpha =1}^{Q}2 R_{z}^{-1}(v^{\alpha} - v_{a}(z)) \frac{\partial \, ((v^{\alpha} - v_{a}(z))/R_{z})}{\partial R_{z}}$ a.e. in $B_1(0)$, we see, using the Cauchy-Schwartz inequality (in the form $2ab \leq \epsilon a^{2} + \epsilon^{-1}b^{2}$) and property $(\FB4\textnormal{I})$, that this implies 
\begin{align*} \label{keyest_cor_eqn2}
	&\int \varphi(R_{z}/\rho)^2 \eta_{\delta}(R_{z}) R_{z}^{-n+\beta-2} |v - v_{a}(z)|^2\\ 
	&\leq C\int \left( \varphi(R_{z}/\rho)^2 R_{z}^{2-n+\beta} \left(\frac{\partial \, (v - v_{a}(z))/R_{z})}{\partial \, R_{z}}\right)^2 
		+ \rho^{-2}\varphi^{\prime}(R_{z}/\rho)^2 R_{z}^{-n+\beta} |v - v_{a}(z)|^2 \right) \nonumber \\
		&\leq C \rho^{-n-2+\beta}\int_{B_{\rho}(z)}  |v - v_{a}(z)|^2 \;,
\end{align*}
where $C = C(n,Q, \beta) \in (0,\infty)$ and we have discarded the last integral on the right hand side of (\ref{cont-1}).  
Letting $\delta \downarrow 0$ and using the monotone convergence theorem we deduce from this that 
$$\int_{B_{\rho/4}(z)} \frac{|v - v_{a}(z)|^{2}}{R_{z}^{n+2 -\beta}} \leq C \rho^{-n-2+\beta}\int_{B_{\rho}(z)} |v - v_{a}(z)|^{2}$$
for every $\rho \in (0, 1/4)$ and every $z \in B_{1/2} \cap \Gamma^{\textnormal{HS}}_{v}$, whence 
$$\sigma^{-n}\int_{B_{\sigma}(z)} |v - v_{a}(z)|^{2} \leq C \left(\frac{\sigma}{\rho}\right)^{2- \beta} \rho^{-n}\int_{B_{\rho}(z)} |v - v_{a}(z)|^{2}$$
for every $\sigma$, $\rho$ with $0 < \sigma \leq \rho/4 \leq 1/16$ and every $z \in B_{1/2} \cap \Gamma_{v}^{\textnormal{HS}}$. Since by property $(\FB4)$ we have that $v$ is harmonic in $B_{1} \setminus \Gamma_{v}^{\textnormal{HS}}$, the desired conclusion for 
$\gamma = 1/2$, with values of $v$ for each $z \in B_{1/2} \cap \Gamma_{v}^{\textnormal{HS}}$ defined by $v^{\alpha}(z) = v_{a}(z)$ for $\alpha = 1, 2, \ldots, Q$, follows from this (taken with $2 - 2\beta$ in place of $\beta$) and an appropriate version of the Campanato  lemma (e.g.\ as in \cite[Lemma 4.3]{wickramasekera2014general} or \cite{minter2021campanato}). The conclusion for $\gamma \in (1/2, 1)$ follows by applying the conclusion for $\gamma = 1/2$ to appropriately translated and rescaled  $v$ and using a covering argument.
\end{proof}

\subsection{A homogeneity continuation property} 
In our argument for the classification of homogeneous degree 1 blow-ups (Theorem~\ref{classification} below) and our proof of $GC^{1, \alpha}$ regularity of coarse blow-ups (Theorem~\ref{coarse_reg} below) we will use the following fact, which says that if a coarse blow-up is homogeneous of degree 1 in an annulus then it is homogeneous of degree 1 everywhere. This is an elementary consequence of properties $(\FB3)$, $(\FB4)$ (in subsection~\ref{initial-cb}) of coarse blow-ups and the unique continuation property of harmonic functions: 

\begin{lemma}\label{unique-cont}
Let $v \in \FB_{Q}$ and suppose that $v$ is homogeneous of degree 1 on an annulus $B_1\setminus \overline{B}_r$ for some $r \in (0, 1)$, i.e.\ $\frac{\del(v/R)}{\del R} = 0$ a.e. in $B_1\setminus\overline{B}_r$ where $R(x) = |x|$. Then $v$ is homogeneous of degree 1 in $B_1$.
\end{lemma}

\begin{proof}
First note the following general fact: if $\Omega \subset \R^{n}$ is a connected, open set, and if $u \, : \, \Omega \to \R$ is a harmonic function which is homogeneous of degree $1$ on some (non-empty) open subset of $\Omega$, then $u$ is homogeneous of degree $1$ on $\Omega$. Indeed, for $x \in \Omega \setminus \{0\}$, 
set $w:= \frac{\del \, (u/R)}{\del R} = \frac{x\cdot Du - u}{R^{2}}$ where $R(x) = |x|,$ and note that 
$\Delta(R^{2}w) = \Delta(x\cdot Du) - \Delta u = 0$ in $\Omega \setminus \{0\}$. Since $w\equiv 0$ on some open subset of $\Omega$, it follows from unique continuation for harmonic functions that $R^{2}w\equiv 0 $ on $\Omega \setminus \{0\}$ and hence $w\equiv 0$ on $\Omega \setminus\{0\}$. Thus $u$ is homogeneous of degree $1$ on $\Omega$.
	
	Now to prove the lemma, note that since $v_{a}$ is harmonic in $B_{1}$ 
	(by property $(\FB3)$) and homogeneous of degree 1 in $B_{1} \setminus B_{r}$, we have that $v_{a}$ is homogeneous of degree 1 in $B_{1}$ (and hence in fact is linear). Set $U = \{x \in B_{r} \, : \, v^{\alpha}(x) \neq v_{a}(x) \;\; \mbox{for some} \;\; \alpha \in \{1, \ldots, Q\}\}.$ Then $U$ is open by continuity of $v$ (Lemma~\ref{continuity}), and by property $(\FB4)$ and Lemma~\ref{continuity}, 
	for each $x \in U$, there is $\rho_{x} >0$ such that $\left. v^{\alpha}\right|_{B_{\rho_{x}}(x)}$ is harmonic for each $\alpha \in \{1, \ldots, Q\}$. If $U = \emptyset$, then 
	$v^{\alpha} = v_{a}$ on $B_{r}$ for each $\alpha$ and the lemma follows. So suppose $U \neq \emptyset$, and let $\Omega$ be any connected component of $U$. If $v^{\alpha}(x) = v_{a}(x)$ for each $\alpha$ and each 
	$x \in \partial \Omega$ (boundary taken in $B_{1}$), then since  $v^{\alpha}$ is harmonic on $\Omega$ for each $\alpha$, it follows from the weak maximum principle that $v^{\alpha} = v_{a}$ on $\overline{\Omega}$ for each $\alpha$, which is impossible since $\Omega \subset U$. So there is a point $x_{0} \in \partial \, \Omega$ such that $v^{\alpha}(x_{0}) \neq v_{a}(x_{0})$ for some $\alpha \in \{1, \ldots, Q\}$, whence (by $(\FB4)$) there is $\rho_{0}>0$ such that $v^{\alpha}$ is harmonic in $B_{\rho_{0}}(x_{0})$ for each $\alpha$. Now, it is not possible that $x_{0} \in B_{r}$ (for if $x_{0} \in B_{r}$ then $x_{0} \in U$ and hence (since $\Omega$ is a component) $x_{0} \not\in \partial \Omega$). Thus $x_{0} \in \partial B_{r} \cap \partial \Omega$. Then, since 
	$\Omega_{0} \equiv \Omega \cup B_{\rho_{0}}(x_{0})$   is connected, $v^{\alpha}$ is harmonic in $\Omega_{0}$ for each $\alpha$, and (by assumption) $v^{\alpha}$ is homogeneous of degree 1 in $B_{r_{0}}(x_{0}) \setminus B_{r}$, it follows  from the above general fact that $v^{\alpha}$ is homogeneous of degree 1 in $\Omega$ for each $\alpha$. This is true for every component of $U$, so $v$ is homogeneous of degree 1 in $U$. We also have, since 
	$v^{\alpha} = v_{a}$ on $B_{r} \setminus U$ for each $\alpha$, that $\frac{\partial \, (v^{\alpha}/R)}{\partial R} =0$ a.e.\ on $B_{r} \setminus U$ for each $\alpha$. Thus $\frac{\partial \, (v/R)}{\partial R} =0$ a.e.\ on $B_{r}.$ 
But for any function $h \in C^{1}(\overline{B_{r}}; \R^{Q})$, any $s \in (0, r)$ and any $\rho, \sigma \in [s, r],$  we have that 
	$\left|\frac{h(\rho\omega)}{\rho} - \frac{h(\sigma\omega)}{\sigma}\right| \leq s^{1-n}\int_{s}^{r} t^{n-1}\left|\frac{d(h(t\omega)/t)}{d t}\right| \ext t$ for any $\omega \in {\mathbb S}^{n-1}$, whence, integrating this with respect to $\omega,$
$\int_{{\mathbb S}^{n-1}} \left|\frac{h(\rho\omega)}{\rho} - \frac{h(\sigma\omega)}{\sigma}\right| \ext\omega \leq s^{1-n}\int_{B_{r} \setminus B_{s}} \left|\frac{\del (h/R)}{\del R}\right|;$ by an approximation argument, this extends to any $h \in C^{0}(\overline{B_{r}}; \R^{Q}) \cap W^{1, 2}(B_{r}; \R^{Q})$, so taking $h = v$, we conclude that $\frac{v(\rho \omega)}{\rho} = \frac{v(\sigma \omega)}{\sigma}$
for every $\omega \in {\mathbb S}^{n-1}$ and every $\rho, \sigma \in (0, r]$, i.e.\ $v$ is homogeneous of degree 1 on $B_{r}$ and hence on $B_{1}$. 
\end{proof}

\subsection{An $\epsilon$-regularity property for coarse blow-ups}\label{sec:epsilon_reg}

We have the following $\epsilon$-regularity property of coarse blow-ups in $\FB_Q$, which can be viewed as the analogue of Theorem \ref{thm:B} for $\FB_Q$:

\begin{theorem}[Coarse blow-up $\epsilon$-regularity property]\label{thm:B7}
	Let $\alpha \in (0, 1)$ and $\gamma \in (0, 1)$. There exists $\epsilon = \epsilon(n,Q, \alpha, \gamma)\in (0,1)$ such that whenever $v\in \FB_Q$ satisfies $0\in \Gamma_v,$ $\|v\|_{L^2(B_{1})} = 1,$ and 
$$\int_{B_1}\G(v,\psi)^2<\epsilon$$
for some $\psi:\R^n\to \A_Q(\R)$ with $\psi_a\equiv 0$ and ${\mathbf v}(\psi) \in {\mathcal C}_{Q},$
then we have the following: 
	\begin{itemize}
		\item [(i)] $v$ is of class generalised-$C^{1,\alpha}$ in $B_{1/2};$ moreover, in the notation of Definition \ref{genC1alpha},
		$\B_v \,\cap \, B_{1/2} = \emptyset$, $\CC_v \, \cap \, B_{1/2} = {\rm graph} \, \varphi \, \cap \, B_{1/2}$ for some function $\varphi \, : \, S(\psi) \cap B_{1/2}  \to S(\psi)^{\perp} \subset \R^{n}$ of class $C^{1, \alpha}$ with $|\varphi|_{1, \alpha; S(\psi) \cap B_{1/2}} \leq C\gamma$, and,  if $\Omega^\pm$ denote the two components of $B_{1/2}\setminus \CC_v$, then $\left.v\right|_{\Omega^\pm}\in C^{1,\alpha}(\overline{\Omega^\pm})$ and $|v|_{1,\alpha;\Omega^\pm}\leq C$, where $C = C(n, Q) \in (0, \infty)$.
	\item [(ii)] There is a function $\widetilde{\psi}:\R^n\to \A_Q(\R)$ with ${\mathbf v}(\widetilde{\psi})\in \CC_Q$ such that 
	$$\dist^2_\H(\graph(\widetilde{\psi})\cap (\R\times B_1), \graph(\psi)\cap (\R\times B_1))\leq C\gamma \;\; \mbox{and}$$
	$$\sigma^{-n-2}\int_{B_\sigma}\G(v(x),\widetilde{\psi}(x))^{2}\ \ext x \leq C\gamma\sigma^{2\alpha}\ \ \ \ \text{for all }\sigma\in (0,1/2),$$
	where $C = C(n, Q) \in (0, \infty).$
	\end{itemize}
\end{theorem}

\begin{proof} Fix $v\in \FB_Q$ with $0\in \Gamma_v$,  $\|v\|_{L^2(B_1(0))} = 1$ and $\psi:\R^n\to \A_Q(\R)$ with $\psi_a\equiv 0$, $\mathbf{v}(\psi)\in \CC_Q$, and $\int_{B_1}\G(v,\psi)^2<\epsilon$
where $\epsilon$ is to be chosen depending only on $n,Q$.  By property $(\FB5\text{II})$ we can without loss of generality rotate to assume that the $(n-1)$-dimensional spine $S(\psi) = \{(0,0)\}\times\R^{n-1}$. Let $(\widetilde{V}_{j})_j\subset\S_Q$ be any sequence of varifolds with $\widetilde{V}_j\weakly Q|\{0\}\times\R^n|$ in $\R\times B_1$ such that $v$ is the coarse blow-up of  $(\widetilde{V}_{j})_{j}$.  Since $0\in \Gamma_v$, we have by the argument establishing property $(\FB4)$ (see the paragraph after the list of properties $(\FB1)$-$(\FB6)$) that after passing to a subsequence of $(j)$ without relabelling, $\Theta_{\widetilde{V}_j}(Z_j)\geq Q$ for some sequence of points $(Z_j)_j$ with $Z_j\to 0$ as $j\to\infty$. Set $V_{j} =\eta_{Z_{j}, 1 - |Z_{j}| \, \#} \, \widetilde{V}_{j}$.  Then $\Theta_{V_j}(0)\geq Q$ for each $j$, $\hat{E}_{V_{j}, \{0\} \times \R^{n}} \to 0$ and the coarse blow-up of $(V_{j})$ is still $v$. Write $u_{j}$ for the $Q$-valued Lipschitz function approximating $V_j$ (on $B_{\sigma_j}$ with $\sigma_j\uparrow 1$) which generates $v$ after scaling by the height excess $\hat{E}_{V_j}\equiv \hat{E}_j = \sqrt{\int_{\R\times B^n_1(0)}|x^1|^2\ \ext \|V_{j}\|}$. We claim that if $\epsilon \in (0, 1)$ is sufficiently small, then for all sufficiently large $j$  we have
\begin{equation}\label{best-plane} 
\hat{E}^2_j < \frac{3}{2}\inf_P\int_{\R\times B_1}\dist^2(X,P)\ \ext\|V_j\|(X)\;,
\end{equation}
where the infimum is taken over all hyperplanes $P$ of the form $P = \{x^1=\lambda x^2\}$, $\lambda\in\R$. Indeed, if this were false then we could pass to a subsequence (without relabelling) and find a sequence of hyperplanes $P_j = \{x^1=\lambda_j x^2\}$ such that
$\int_{\R\times B_1}\dist^2(X,P_j)\ \ext\|V_j\| \leq \frac{5}{6}\hat{E}^2_{j}.$
Thus for each $\sigma\in (1/2,1)$, for all $j$ sufficiently large we have
\begin{equation}\label{E:3.1}
(1+\lambda_j^2)^{-1}\sum^Q_{\alpha=1}\int_{B_\sigma\setminus\Sigma_{j}}|u_j^\alpha(x^2,y) - \lambda_j x^2|^2\ \ext x \leq \frac{5}{6}\hat{E}_j^2\;,
\end{equation}
from which it follows that  $(1 + \lambda_{j}^{2})^{-1} \lambda_{j}^{2} \int_{B_{1/2} \setminus \Sigma_{j}} |x^{2}|^{2} \, \ext x \leq \frac{11}{3}\hat{E}_j^{2}$ and hence $|\lambda_j|\leq C\hat{E}_{j}$ for all $j$ and some constant $C = C(n) \in (0,\infty)$, and so we may pass to a further subsequence (without relabelling) to assume that $\hat{E}_j^{-1}\lambda_j\to \lambda\in \R$.  Moreover, from (\ref{E:3.1}) and (\ref{badset-est}) it readily follows that
$$\sum^Q_{\alpha=1}\int_{B_\sigma}|u^\alpha_j-\lambda_j x^2|^2\ \ext x^2\ext y \leq \frac{5}{6}(1+\lambda_j^2)\hat{E}_j^2 + 2C_\sigma\sup_{B_\sigma}(|u_j|^2+\lambda_j^2|x^2|^2)\hat{E}_j^2\;,$$
where $C = C_\sigma\in (0,\infty)$ is a constant. Dividing this by $\hat{E}_j^2$, letting $j\to\infty$, and then $\sigma\uparrow 1$, we see, writing $\ell(x^2,\dotsc,x^{n+1}) := \lambda x^2$, that 
$\int_{B_1(0)}|v-\ell|^2\leq \frac{5}{6}$
whence by the triangle inequality $\int_{B_1(0)}|\psi-\ell|^2 \leq \left(\sqrt{\frac{5}{6}}+\sqrt{\epsilon}\right)^{2}$; but since $\psi$ is average-free we have $|\psi-\ell |^2 = |\psi|^2 - 2\ell\sum^Q_{\alpha=1}\psi^\alpha + Q|\ell|^2 = |\psi|^2 + Q|\ell|^2$, and so we get
$\int_{B_1(0)}|\psi|^2 + Q\int_{B_1(0)}|\ell|^2 \leq \left(\sqrt{\frac{5}{6}}+\sqrt{\epsilon}\right)^{2}$
which, as $\int_{B_1(0)}|\psi|^2 > (1-\sqrt{\epsilon})^{2}$, is a contradiction for $\epsilon<\frac{1}{4}\left(1 - \sqrt\frac{5}{6}\right)^{2}$. So (\ref{best-plane}) must hold for all sufficiently large $j$ as claimed if we choose 
$\epsilon \in \left(0, \frac{1}{4}\left(1 - \sqrt\frac{5}{6}\right)^{2}\right)$.

Our next aim is to show that for any given $\gamma \in (0, 1)$, if we choose $\epsilon = \epsilon(n, Q, \gamma) \in (0, 1)$ sufficiently small, then we have for all sufficiently large  $j$ that 
\begin{equation} \label{fine-excess-hyp}
Q_{V_j,\BC_j}^2<\gamma \hat{E}_{V_j}^2 
\end{equation}
for a suitable sequence of cones $\BC_j\in \CC_Q$ with $S(\BC_{j}) = \{(0, 0)\} \times \R^{n-1}$. Indeed, set $\BC_j := \mathbf{v}(\hat{E}_j\psi)$. Then we have
$$\int_{\R\times B_\sigma}\dist^2(X,\spt\|\BC_j\|)\ \ext\|V_j\| \leq 2\int_{B_\sigma}|u_j - \hat{E}_j \psi|^2 + C_\sigma\sup_{X\in \spt\|V_j\|\cap (\R\times B_\sigma)}\dist^2(X,\spt\|\BC_j\|)\hat{E}_j^2.$$
Since $\int_{B_{1}} \G(v, \psi)^{2} <\epsilon < 1$ and $\|v\|_{L^{1}(B_{1})} = 1$, we have by the triangle inequality that 
$\int_{B_1}|\psi|^2> \left(1 -  \sqrt{\epsilon}\right)^{2}$, 
and so as $\psi$ is homogeneous of degree 1 we see that for each $\sigma\in (0,1)$, 
$\int_{B_\sigma}|\psi|^2 > \left(1-\sqrt{\epsilon}\right)^{2}\sigma^{n+2}.$
So for fixed $\sigma\in (1/2,1)$, we have for all $j$ sufficiently large 
$\int_{B_\sigma}|u_j|^2\geq \left(\left(1-\sqrt{\epsilon}\right)\sigma^{\frac{n+2}{2}} - \sqrt{\epsilon}\right)^{2}\hat{E}_j^2$. 
Now a simple calculation shows
$$\int_{\R\times B_\sigma}|x^1|^2\ \ext\|V_j\|(X) \geq \int_{B_\sigma}|u_j|^2 - 2C_\sigma \hat{E}_j^2\sup_{B_\sigma}|u_j|^2$$
from which it follows that
$$\int_{\R\times (B_1\setminus B_\sigma)}|x^1|^2\ \ext\|V_j\|(X)\leq \left(1-\left(\left(1-\sqrt{\epsilon}\right)\sigma^{\frac{n+2}{2}} - \sqrt{\epsilon}\right)^{2} + 2C_\sigma\sup_{B_\sigma}|u_j|^2\right)\hat{E}_j^2.$$
Another simple calculation shows that for all sufficiently large $j$,
$$\int_{\R\times (B_1\setminus B_\sigma)}\dist^2(X,\spt\|\BC_j\|)\ \ext\|V_j\|(X) \leq 2\int_{\R\times (B_1\setminus B_\sigma)}|x^1|^2\ \ext\|V_j\|(X) + C\H^n(B_1\setminus B_\sigma)\hat{E}^2_j\;,$$
where $C = C(n,Q)\in (0,\infty)$. Combining these inequalities we see that
\begin{eqnarray*}\label{fine-excess-hyp1} 
&&\int_{\R\times B_1}\dist^2(X,\spt\|\BC_j\|)\ \ext\|V_j\| \nonumber\\ 
&& \hspace{2em}\leq 2\left(C_\sigma\sup_{X\in \spt\|V_j\|\cap (\R\times B_\sigma)}\dist^2(X,\spt\|\BC_j\|) + \int_{B_\sigma}|\hat{E}_j^{-1}u_j-\psi|^2\right. \nonumber\\
&&\hspace{3em}\left.+ 1-\left(\left(1-\sqrt{\epsilon}\right)\sigma^{\frac{n+2}{2}} - \sqrt{\epsilon}\right)^{2}  + \; 2C_\sigma\sup_{B_\sigma}|u_j|^2 + C\H^n(B_1\setminus B_\sigma)\right)\hat{E}_j^2\;,
\end{eqnarray*} 
where $C = C(n, Q).$ Hence, for any given $\gamma \in (0, 1)$,  we may fix $\sigma = \sigma(n,Q, \gamma)$ sufficiently close to $1$, and choose 
$\epsilon_0 = \epsilon_0(n,Q, \gamma)\in (0,1)$ such that if $\epsilon\leq \epsilon_0$, then
$$\int_{\R\times B_1}\dist^2(X,\spt\|\BC_j\|)\ \ext\|V_j\| \leq \frac{\gamma}{4}\hat{E}_j^{2}$$
for all sufficiently large $j$. Now from Proposition \ref{prop:sheets}(ii), if $\epsilon = \epsilon(n, Q) \in (0,1)$ is sufficiently small, then for all $j$ sufficiently large we must have
$$\Theta_{V_j}(Z)<Q \text{ for all }Z\in\spt\|V_j\|\cap(\R\times B_{5/8})\cap\{|x^2|>1/32\}.$$
Thus we see that $V_j$ has no classical singularities in the region $(\R\times B_{5/8})\cap\{|x^2|>1/32\}$, and so we may apply Theorem~\ref{sheeting} to get that
$$V_j\res\left((\R\times B_{9/16})\cap \{x^2<-3/64\}\right) = \sum^Q_{\alpha=1}|\graph \, u^{(\alpha, \; -)}_j|\res \left((\R\times B_{9/16})\cap \{x^2<-3/64\}\right)$$
and
$$V_j\res\left((\R\times B_{9/16})\cap \{x^2>3/64\}\right) = \sum^Q_{\alpha=1}|\graph \, u^{(\alpha, \; +)}_j|\res \left((\R\times B_{9/16})\cap \{x^2> 3/64\}\right),$$
where $u^{(1, \; -)}_j\leq \cdots \leq u^{(Q, \; -)}_j$ and $u^{(1, \; +)}_j \leq \cdots \leq u^{(Q, \; +)}_j$ are $C^2$ functions on $B_{5/8}\cap \{x^2<-3/64\}$ and $B_{5/8}\cap \{x^2> 3/64\}$ solving the minimal surface equation and satisfying 
$$\|u^{(\alpha, \; -)}_{j}\|_{C^{1}(B_{9/16}\cap \{x^2<-3/64\})} + \|u^{(\alpha, \; +)}_{j}\|_{C^{1}(B_{9/16}\cap\{x^2>3/64\})}\leq C\hat{E}_j$$
for some $C = C(n)$. Using the notation $\psi|_{\R^n_+} = \sum^Q_{\alpha=1}\llbracket h^\alpha \rrbracket$ and $\psi|_{\R^n_-} = \sum^Q_{\alpha=1}\llbracket g^\alpha \rrbracket$ where $h^\alpha:\R^n_+\to \R$, $g^\alpha:\R^n_-\to \R$ are of the form $h^\alpha(x^2,\dotsc,x^{n+1}) = \lambda^\alpha x^2$ and $g^\alpha(x^2,\dotsc,x^{n+1}) = \mu^\alpha x^2$, with $\lambda^1\leq \cdots\leq \lambda^Q$ and $\mu^1\leq \cdots\leq \mu^Q,$  we have that  
\begin{align*}
{\rm dist} \, ((x, \hat{E}_{j}h^{\alpha}(x)), {\rm spt} \, \|V_{j}\|) & \leq |\hat{E}_{j}h^{\alpha}(x) - u^{(\alpha, \; +)}_{j}(x)|\nonumber\\ 
& \leq \hat{E}_{j}(|h^{\alpha}(x) - v^{\alpha}(x)| + |v^{\alpha}(x) - \hat{E}_{j}^{-1}u_{j}^{(\alpha, \; +)}(x)|)
\end{align*}
for each $x \in B_{1/2} \cap \{x^{2} > 1/16\}$ and similarly, 
\begin{align*}
{\rm dist} \, ((x, \hat{E}_{j}g^{\alpha}(x)), {\rm spt} \, \|V_{j}\|) & \leq |\hat{E}_{j}g^{\alpha}(x) - u^{(\alpha, \; -)}_{j}(x)| \nonumber\\ 
& \leq \hat{E}_{j}(|g^{\alpha}(x) - v^{\alpha}(x)| + |v^{\alpha}(x) - \hat{E}_{j}^{-1}u_{j}^{(\alpha, \; -)}(x)|)
\end{align*} 
for each $x \in B_{1/2} \cap \{x^{2} < -1/16\};$ thus, 
\begin{align*}
&\int_{\R\times(B_{1/2}\setminus\{|x^2|>1/16\})}\dist^2(X,\spt\|V_j\|)\ \ext\|\BC_j\|(X) \nonumber\\
&\hspace{1in}\leq \left( 2\int_{B_{1}}\G(v, \psi)^{2} + \sum_{\alpha=1}^{Q}\int_{B_{1/2}}|v^{\alpha} - \hat{E}_{j}^{-1}u_{j}^{\alpha}|^{2}\right)\hat{E}_j^2
\end{align*}
and so, if $\epsilon = \epsilon(n,Q, \gamma)$ is sufficiently small, we have (\ref{fine-excess-hyp}) for all $j$ sufficiently large.
In particular, for any $\gamma \in (0, \gamma_{0}]$ where $\gamma_{0} = \gamma_{0}(n, Q, \alpha)$ is as in Theorem~\ref{thm:fine_reg}, and for 
$\epsilon = \epsilon(n, Q, \alpha) \in (0, 1)$ sufficiently small, we have that the hypotheses of Theorem~\ref{thm:fine_reg} are satisfied with $V_{j}$ in place of $V$ and $\BC_{j}$ in place of $\BC$ for all sufficiently large $j$. So applying Theorem~\ref{thm:fine_reg} we obtain for each sufficiently large $j,$  functions $\varphi_{1}^{(j)}, \varphi_{2}^{(j)} \in C^{1, \alpha} \, (\overline{B}_{1/2} \cap \left(\{(0, 0)\} \times \R^{n-1}\right); \R^{2})$ with $\varphi_{1}^{(j)}(0) = \varphi_{2}^{(j)}(0) = 0$, $|\varphi_{1}^{(j)}|_{1, \alpha; B_{1/2} \cap \left(\{(0, 0)\} \times \R^{n-1}\right)} \leq C\hat{E}_{V_{j}}$ and $|\varphi_{2}^{(j)}|_{1, \alpha;  B_{1/2} \cap \left(\{(0, 0)\} \times \R^{n-1}\right)} \leq C\hat{E}_{V_{j}}^{-1} Q_{V_{j}, \BC_{j}},$ 
and a function $\widetilde{u}_{j} = \sum_{i=1}^{Q} \llbracket \widetilde{u}_{j}^{i}\rrbracket \in 
GC^{1, \alpha}(B_{1/2}; {\mathcal A}_{Q})$ where $\widetilde{u}_{j}^{1} \leq \widetilde{u}_{j}^{2} \leq \cdots \leq \widetilde{u}_{j}^{Q}$,
such that: $V_j \res (\R\times B_{1/2}) = \mathbf{v}(\widetilde{u}_j)\res (\R\times B_{1/2});$
$\B_{\widetilde{u}_{j}} = \emptyset;$ $\CC_{\widetilde{u}_{j}} = {\rm graph} \, \varphi_{2}^{(j)} \cap B_{1/2};$    $\widetilde{u}_{j}^{i}(z, \varphi_{2}^{(j)}(z))= \varphi_{1}^{(j)}(z)$ for each $z \in \pi_{\{(0, 0)\} \times \R^{n-1}}(\CC_{\widetilde{u}_{j}})$ and $i \in \{1, 2, \ldots, Q\}$; 
and if $\Omega^\pm_j$ denote the two components of $B_{1/2}\setminus\CC_{\widetilde{u}_j}$, then $\left.\widetilde{u}_{j}\right|_{\Omega_{j}^\pm} \in C^{1,\alpha}(\overline{\Omega_{j}^\pm})$ and $|\widetilde{u}_{j}|_{1,\alpha;\Omega_{j}^\pm}\leq C\hat{E}_{V_j}.$ Here $C = C(n, Q)$. 
Moreover, again by Theorem~\ref{thm:fine_reg}, for each $j$ there is a cone ${\BC}_{0}^{(j)} \in \CC_Q$ with $S({\BC}_{0}^{(j)}) = \{(0, 0)\} \times \R^{n-1}$ and 
	\begin{equation}\label{fine-excess-reg2}
	\dist_\H(\spt\|{\BC}_{0}^{(j)}\|\cap(\R\times B_1),\spt\|\BC_{j}\|\cap (\R\times B_1))\leq CQ_{V_{j},\BC_{j}},
	\end{equation}
	and a rotation $\Gamma_{0}^{(j)} \, : \, \R^{n+1} \to \R^{n+1}$ with $\|\Gamma_{0}^{(j)} - I\| \leq C {\hat E}_{V_{j}}^{-1}Q_{V_{j}, \BC_{j}},$ where $I$ is the identity map on $\R^{n+1},$ such that  $\BC_{0}^{(j)}$ is the unique tangent cone to $\left(\Gamma_{0}^{(j)}\right)^{-1}_{\#} \, V_{j}$ at $0$ and 
	\begin{equation}\label{fine-excess-reg3}
	\sigma^{-n-2}\int_{\R\times B_\sigma}\dist^2(X,\spt\|{\BC}_{0}^{(j)}\|)\ \ext\|\left(\Gamma^{(j)}_{0}\right)^{-1}_{\#} \, V_{j}\|\leq C\sigma^{2\alpha}Q_{V_{j},\BC_{j}}^2\ \ \ \ \text{for all }\sigma\in (0,1/2),
	\end{equation}
where $C = C(n, Q)$.

From these estimates, it follows that $\left.v\right|_{B_{1/2}} \in GC^{1, \alpha}(B_{1/2}; {\mathcal A}_{Q}),$  
$\B_{v} \, \cap \, B_{1/2} = \emptyset$ and that there are functions $\varphi_{1}, \varphi_{2} \in C^{1, \alpha}(\overline{B}_{1/2} \cap \{(0, 0)\} \times \R^{n-1}; \R)$ with $|\varphi_{1}|_{1, \alpha; B_{1/2} \cap \{(0, 0)\} \times \R^{n-1}} \leq C$ and $|\varphi_{2}|_{1, \alpha;  B_{1/2} \cap \left(\{(0, 0)\} \times \R^{n-1}\right)} \leq C\gamma$ such that:  $\CC_{v} \, \cap \, B_{1/2} = {\rm graph} \, \varphi_{2} \, \cap \, B_{1/2}$; if $\Omega^\pm$ denote the two components of $B_{1/2}\setminus\CC_{v}$, then $\left.v\right|_{\Omega^\pm} \in C^{1,\alpha}(\overline{\Omega^\pm})$ and $|v|_{1,\alpha;\Omega^\pm}\leq C$; and 
$v^{i}(z, \varphi_{2}(z))= \varphi_{1}(z)$ for each $z \in \pi_{\{(0, 0)\} \times \R^{n-1}}(\CC_{v})$ and 
$i \in \{1, 2, \ldots, Q\}$. Here again $C = C(n, Q)$. Thus conclusion (i) of the theorem holds. Noting that $v(0) = 0$ and letting $\widetilde{\psi} \, : \, \R^{n} \to {\mathcal A}_{Q}(\R)$ be the unique function such that ${\bf v}(\widetilde{\psi}) = \sum_{j=1}^{Q}|H_{j}^{+}| + |H_{j}^{-}|$ where $H_{j}^{\pm}$ are the tangent half-planes to ${\rm graph} \, \left.v^{j}\right|_{\Omega^{\pm}}$ at $0$ for each $j \in \{1, \ldots, Q\}$, we  see from (\ref{fine-excess-reg2}) and (\ref{fine-excess-reg3}) that conclusion (ii) of the theorem follows. 
\end{proof}

\noindent
{\bf Remark:} Suppose that for some $\alpha \in (0, 1)$ and $\epsilon = \epsilon(n, Q, \alpha)$ as given by Theorem~\ref{thm:B7}, the hypotheses of Theorem~\ref{thm:B7} are satisfied by some $v \in \FB_{Q}$ and some $\psi \, : \, \R^{n} \to {\mathcal A}_{Q}(\R)$ with ${\bf v}(\psi) \in \CC_{Q}.$  Let $(\widetilde{V}_{k})_{k}$ be any sequence of varifolds in  $\S_{Q}$  whose coarse blow-up is $v$, and let 
$\Gamma \, : \, \R^{n+1} \to \R^{n+1}$ be the rotation $\Gamma(x^{1}, x^{\prime}) = \gamma(x^{\prime})$ where 
$\gamma \, : \, \R^{n} \to \R^{n}$ is a rotation such that $\gamma(S(\psi)) = \{(0, 0)\} \times \R^{n-1}$. The proof of Theorem~\ref{thm:B7} shows the following: after possibly passing to a subsequence of $(\widetilde{V}_{k})_{k}$, there is a sequence of points $(Z_{k})_{k}$  with $Z_{k} \to 0$ and $\Theta_{\widetilde{V}_{k}}(Z_{k}) \geq Q$ for all $k$, and moreover, for \emph{any} such sequence of points $(Z_{k})_k$, if we set $V_{k} = \Gamma_{\#} \circ \eta_{Z_{k}, 1 - |Z_{k}| \, \#} \, \widetilde{V}_{k},$ then $(V_{k})_{k} \subset \S_{Q},$ $\hat{E}_{V_{k}, \{0\} \times \R^{n}} \to 0$, $v$ is the coarse blow-up of $(V_{k})_k,$ and for all sufficiently large $k$, the hypotheses of Theorem~\ref{thm:fine_reg} are satisfied with $V_{k}$ in place of $V$ and ${\BC}_{k} = {\mathbf v}(\hat{E}_{V_{k}, \{0\} \times \R^{n}} \psi \circ \gamma)$ in place of $\BC$. 

\subsection{Squash inequality}
Here and subsequently, for $v \in \FB_{Q}$,  we shall use the notation
$$v_f= v-v_a =  (v^1-v_a,v^2-v_a,\dotsc,v^Q-v_a).$$ 
\begin{lemma}[Squash inequalities for coarse blow-ups]\label{squash}
	For each $v\in \FB_Q$ we have:
	$$\int_{B_1}|Dv|^2\zeta \leq - \int_{B_1}\sum^Q_{\alpha=1}v^\alpha Dv^\alpha\cdot D\zeta$$
	and 
	$$\int_{B_1}|Dv_{f}|^2\zeta \leq - \int_{B_1}\sum^Q_{\alpha=1}v_{f}^\alpha Dv_{f}^\alpha\cdot D\zeta$$
for every $\zeta\in C^1_c(B_{3/4};\R)$. In particular, we have 
	$$\int_{B_{1/2}}|Dv|^2 \leq 16\int_{B_1}|v|^2$$
	and
	$$\int_{B_{1/2}}|Dv_{f}|^2 \leq 16\int_{B_1}|v_{f}|^2.\ \ \ \ \ $$
\end{lemma}
\begin{proof}
	Let $\zeta\in C^1_c(B_{3/4};\R)$, and let $\widetilde{\zeta}$ denote the vertical extension of $\zeta$ given by $\widetilde{\zeta}(y,x):= \zeta(x)$ for $(y,x)\in \R\times B_{1}$. Choosing $x^1\widetilde{\zeta}e^{1}$ as the test vector field in the first variation formula for $V_j$ we have
	$$\int_{\R\times B_1}\widetilde{\zeta}|\nabla^{V_j} x^1|^2\ \ext\|V_j\|(x) = -\int_{\R\times B_1}x^1\nabla^{V_j}x^1\cdot \nabla^{V_j}\widetilde{\zeta}\ \ext\|V_j\|(x).$$
	 
 Omitting the non-graphical part of the left hand side, using the bound $J^\alpha_j\geq 1$ for the Jacobian of $u^\alpha_j$ and the bound $|Du_{j}^{\alpha}| \leq 1/2$, we have from the above:
	\begin{align*}
	\int_{B_1\cap\Omega_j}\zeta\sum^Q_{\alpha=1}|Du^\alpha_j(x)|^2\ \ext x \leq -\int_{B_1\cap\Omega_j}\sum^{Q}_{\alpha=1}&u^\alpha_j Du_j^\alpha\cdot D\zeta (1 + |Du_{j}^{\alpha}|^{2})^{-1} \, \ext x\\
	&-\int_{B_1\setminus\Omega_j}x^1\nabla^{V_j}x^1\cdot\nabla^{V_j}\tilde{\zeta}\ \ext\|V_j\|(x)\\
	&\hspace{-1in}\leq -\int_{B_1\cap\Omega_j}\sum^{Q}_{\alpha=1}u^\alpha_j Du_j^\alpha\cdot D\zeta \,\ext x + 
	C\left(\sup_{B_{3/4}} \, |u_{j}^{\alpha}||D\zeta|\right) \hat{E}_{j}^{2}\\
	& -\int_{B_1\setminus\Omega_j}x^1\nabla^{V_j}x^1\cdot\nabla^{V_j}\widetilde{\zeta}\ \ext\|V_j\|(x)\\
	&\hspace{-2.3in}\leq -\int_{B_1\cap\Omega_j}\sum^{Q}_{\alpha=1}u^\alpha_j Du_j^\alpha\cdot D\zeta \,\ext x + 
	C\left(\sup_{X \in {\rm spt} \, \|V_{j}\| \cap ( \R \times B_{3/4})} \, |x^{1}||D\zeta(\pi \, X)|\right) \hat{E}_{j}^{2}\\
	\end{align*}
	where $C = C(n, Q)$, and we used the bounds $\|V_j\|(B_{3/4}\setminus\Omega_j)\leq C\hat{E}^2_j$ 
	and $$\int_{\R \times B_{3/4}} |\nabla^{V_{j}} \, x^{1}|^{2} \, \ext\|V_{j}\|(x)\leq C \hat{E}_{j}^{2}$$ 
	with $C = C(n, Q)$, and the 
	Cauchy-Schwarz inequality in the last step. Note also that by a similar calculation, we also have that 
	$\sum_{\alpha=1}^{Q}\int_{B_{3/4}} |Du_{j}^{\alpha}|^{2} \leq C\hat{E}_{j}^{2}$. 
	 So dividing these inequalities by $\hat{E}_j^2$, we have that $\sum_{\alpha = 1}^{Q}\int_{B_{1} \cap \Omega_{j}} \zeta|Dv_{j}^{\alpha}|^{2} \leq -\int_{B_1\cap\Omega_j}\sum^{Q}_{\alpha=1}v^\alpha_j Dv_j^\alpha\cdot D\zeta \,\ext x + 
	C\left(\sup_{X \in {\rm spt} \, \|V_{j}\| \cap ( \R \times B_{3/4})} \, |x^{1}||D\zeta(\pi \, X)|\right)$ and 
	$\sum_{\alpha = 1}^{Q}\int_{B_{3/4}} |Dv^{\alpha}_{j}|^{2} \leq C$. By the weak convergence in $W^{1,2}(B_{3/4})$ of 
	$v_j \equiv {\hat E}_{j}^{-1}u_{j}$ to $v$, the strong convergence in $L^2(B_{3/4})$, the fact that $\one_{\Omega_j}\to 1$ a.e.\ in $B_{3/4}$ and the uniform bound $\int_{B_{3/4}} |Dv_{j}^{\alpha}|^{2} \leq C$, we see that $\int_{B_{1} \cap \Omega_{j}}v_{j}^{\alpha} Dv_{j}^{\alpha} \cdot D\zeta \to \int_{B_{1}} v^{\alpha} Dv^{\alpha} \cdot D\zeta$ and $\int_{B_{1}}\zeta|Dv^{\alpha}|^{2} \leq \liminf_{j \to \infty} \int_{B_{1} \cap \Omega_{j}}  \zeta |Dv_{j}^{\alpha}|^{2}$. So  since 
	$$\sup_{X \in {\rm spt} \, \|V_{j}\| \cap ( \R \times B_{3/4})} \, |x^{1}| \to 0,$$ we can let $j \to \infty$ to get 
	$$\int_{B_1}\zeta|Dv|^2 \leq -\int_{B_1}\sum^Q_{\alpha=1}v^\alpha Dv^\alpha\cdot D\zeta$$
which is the first assertion. Since $v^{\alpha} = v_{f}^{\alpha} + v_{a}$, $\sum_{\alpha = 1}^{Q} v_{f}^{\alpha} = 0$ and 
$\int_{B_1}\zeta|Dv_{a}|^2 =  -\int_{B_1}v_{a} Dv_{a}\cdot D\zeta$ (by virtue of the fact that $v_{a}$ is harmonic), the second assertion follows immediately from this. The third and the fourth assertions are immediate by taking $\zeta^{2}$ in place of $\zeta$ in these inequalities, using the Cauchy--Schwartz inequality on the right hand side of the resulting inequalities, and then choosing $\zeta$ appropriately. 
\end{proof}

\subsection{An energy non-concentration estimate} An important fact we will use to establish monotonicity of the Almgren frequency function associated with the coarse blow-ups is (a variant of) the following energy non-concentration estimate  proved in \cite{BKW2021}. 
\begin{lemma}[Energy non-concentration estimate, \cite{BKW2021}]\label{noncon}
	For each $v\in \FB_Q$ and every $\delta >0$ we have:
	$$\int_{B_{1/2}}\sum^Q_{\alpha=1}\one_{\{|v_f^\alpha|<\delta\}}|Dv^\alpha_f(x)|^2\ \ext x \leq C\delta\left(\int_{B_1}|v|^2\right)^{1/2}\;,$$
	where $C = C(n,Q)>0.$
\end{lemma}

\textbf{Remark.} Lemma~\ref{noncon} holds in much more generality than stated here, including in arbitrary codimension and just under the assumption of stationarity of the sequence of 
varifolds (converging to $Q|\{0\} \times {\mathbb R}^{n}|$) that produces $v$. No additional regularity hypothesis on $v$ is necessary. See \cite{BKW2021}. 

In the present context, in fact a stronger version of this can be established in a more elementary fashion, based on continuity of $v$ (Lemma~\ref{continuity}) and property $(\FB4).$

\begin{lemma}[Alternative energy non-concentration estimate]\label{noncon2}
	For every $v\in \FB_Q$ and $\delta >0$ we have
	$$\int_{B_{\rho}}\sum^Q_{\alpha=1}\one_{\{|v^\alpha_f|<\delta\}}|D v^\alpha_f(x)|^2\ \ext x \leq C(1-\rho)^{-2}\delta\left(\int_{B_1}|v_{f}|^2\right)^{1/2}$$
	for every $\rho \in (0, 1)$ and some $C = C(n,Q)\in (0,\infty)$.
\end{lemma}

\begin{proof}
	Fix $\delta>0$.  For $\epsilon < \delta/2$, define $\eta_\epsilon:\R\to \R$ to be the odd extension to $\R$ of the function $t\mapsto \delta\one_{(\epsilon, \infty)}(t)\cdot \min\left\{\frac{t-\epsilon}{\delta - \epsilon}, 1\right\}$ for $t>0$. Let $\sigma \in (0, 1)$ and $\zeta\in C^1_c(B_{\sigma}).$ By Lemma~\ref{continuity}, $v$ is continuous in $B_{1}$, so for each $\alpha\in\{1,\dotsc,Q\}$, we have by $(\FB4)$ that $v_f^\alpha$ is harmonic on the open set $\{|v^\alpha_f|>0\}$, 
	so by an approximation argument we may take $\eta_{\epsilon}(v_f^\alpha)\zeta$ as a test function in the weak harmonicity identity for $v_f^\alpha$, i.e.,
	$$\int_{B_1}Dv^\alpha_f\cdot D(\eta_{\epsilon}(v^\alpha_f)\zeta) = 0.$$
Upon rearranging this becomes
$$\int_{B_1}\zeta\cdot \eta_{\epsilon}'(v_f^\alpha)\cdot|D v^\alpha_f|^2 = -\int_{B_1}\eta_{\epsilon}(v^\alpha_f)Dv^\alpha_f\cdot D \zeta,$$
from which we get
$$\frac{\delta}{\delta - \epsilon}\int_{B_1}\zeta\cdot\one_{\{\epsilon<|v^\alpha_f|<\delta\}}\cdot|D v^\alpha_f|^2 \leq \w_{n}^{1/2}\sup_\R|\eta_{\epsilon}| \, \sup_{B_{\sigma}} \, |D\zeta|\left(\int_{B_{\sigma}}|Dv_f^\alpha|^2\right)^{1/2}.$$
By Lemma~\ref{squash} we have that $\int_{B_{\sigma}}\sum_{\alpha = 1}^{Q}|D v_{f}^{\alpha}|^2 \leq C(1-\sigma)^{-2}\int_{B_{1}} \sum_{\alpha =1}^{Q} |v_{f}^{\alpha}|^2$
where $C = C(n)$, and so we get from the above, after summing over $\alpha$ and letting $\epsilon \to 0$ and noting that $D v_{f}^{\alpha} = 0$ a.e.\ on $\{v_{f}^{\alpha} = 0\}$, 
$$\int_{B_1} \sum_{\alpha = 1}^{Q}\zeta\cdot\one_{\{|v^\alpha_f|<\delta\}}\cdot|D v_{f}^{\alpha}|^2 \leq C\delta (1-\sigma)^{-1}\sup_{B_{\sigma}}|D\zeta|\left(\int_{B_{1}}|v_{f}|^2\right)^{1/2}.$$
For given $\rho \in (0, 1)$, choosing $\sigma = (1+\rho)/2$ and $\zeta$ such that $\zeta = 1$ on $B_{\rho}$, $|D\zeta|  \leq 2/(\sigma - \rho)$, this gives the desired conclusion.
\end{proof}

\subsection{Squeeze identity and frequency monotonicity for generalised-$C^{1}$ coarse blow-ups}\label{C1squeeze}

We start by proving the validity of the squeeze identity locally about classical singularities of coarse blow-ups.

\begin{lemma}[Squeeze identity at classical singular points]\label{squeeze_classical}
	Let $v\in \FB_Q$ and let $x_{0} \in \CC_{v}$. There is a number $r = r(v,x_0) \in (0,\frac{1}{2}(1-|x_0|))$ such that the following is true:
	\begin{enumerate}
		\item [(i)] $\Gamma_v\cap B_r(x_0)$ is equal to the graph of a $C^1$ function $\xi \, : \, B_r(x_0) \cap (x_{0} + L) \to L^{\perp}$ for some $(n-1)$-dimensional subspace $L \subset \R^{n}$, where $L^{\perp}$ is the orthogonal complement of $L$ in $\R^{n}$;
		\item [(ii)] $\Gamma_v\cap B_r(x_0) = \CC_v\cap B_r(x_0)$;
		\item [(iii)] \textnormal{(Squeeze Identity)} we have for every $\zeta\in C^1_c(B_r(x_0);\R^n)$,
		$$\int_{B_r(x_0)}\sum^Q_{\alpha=1}\sum^n_{i,j=1}\left(|Dv^\alpha|^2\delta_{ij} - 2D_i v^\alpha D_j v^\alpha\right)D_i\zeta^j = 0$$
		and
		$$\int_{B_r(x_0)}\sum^Q_{\alpha=1}\sum^n_{i,j=1}\left(|Dv_f^\alpha|^2\delta_{ij} - 2D_i v_f^\alpha D_j v_f^\alpha\right)D_i\zeta^j = 0.$$
	\end{enumerate}
\end{lemma}

\begin{proof}
	The first two conclusions hold for suitably chosen $r = r(v,x_0)$ by the definition of $\CC_v$ and the fact that, by Lemma~\ref{continuity}, we have $v(z) = Q\llbracket v_{a}(z)\rrbracket$ for every $z \in \Gamma_{v}$. The second conclusion in (iii) follows from the first by writing $v^\alpha = v^\alpha_f + v_a$ and noting that $\sum^Q_{\alpha=1}v^\alpha_f = 0$ and that 
	$$\int_{B_r(x_0)}\sum^n_{i,j=1}\left(|Dv_a|^2\delta_{ij} - 2D_iv_a D_j v_a\right)D_i\zeta^j = 0,$$
	which is easily seen by integrating by parts, since $v_a$ is harmonic.
	To prove the squeeze identity for $v$ (the first conclusion in (iii)), note first that by properties $(\FB5\textnormal{I})$, $(\FB5\textnormal{III})$ together with the fact that this identity holds for affine functions, we may assume without loss of generality (after choosing $r$ sufficiently small, translating, ``tilting'' and rescaling, i.e.\ considering 
	$\widetilde{v}(x) = \|v(x_{0} + r(\cdot)) - (v_{a}(0) + rDv_{a}(0) \cdot (\cdot))\|_{L^{2}(B_{1})}^{-1}(v(x_{0} + rx) - (v_{a}(0) + rDv_{a}(0) \cdot x))$ in place of $v$) that $x_{0}=0,$ $r=1$, 
	$\|v\|_{L^{2}(B_{1})} = 1$, $v_{a}(0)=0$ and $Dv_{a}(0) = 0,$ and that $\int_{B_{1}} \mathcal G(v, \phi)^{2} < \epsilon$ where $\phi \, : \, \R^{n} \to {\mathcal A}_{Q}(\R)$ is such that ${\bf v}(\phi) \in \CC_{Q}$ and $(\phi)_{a} \equiv 0$, and $\epsilon$ is as in Theorem~\ref{thm:B7}. Now let $(\widetilde{V}_k)_k\subset \S_Q$ be a sequence of varifolds generating $v.$ Then by the remark following Theorem~\ref{thm:B7}, we have that for sufficiently large $k$, Theorem~\ref{thm:fine_reg} is applicable with $V_{k} = \Gamma_{\#} \, \eta_{Z_{k}, 1 - |Z_{k}| \, \#} \, \widetilde{V}_{k}$ in place of $V$ and 
	${\BC}_{k} = {\mathbf v}(\hat{E}_{V_{k}, \{0\} \times \R^{n}} \psi \circ \gamma)$ in place of $\BC, $ where $\gamma \, : \, \R^{n} \to \R^{n}$ is a rotation taking $S(\phi)$ to $\{(0, 0)\} \times \R^{n-1}$ and $\Gamma(x) = \gamma(\pi (x))$ for $x \in \R^{n+1}$. This provides, for any $\alpha \in (0, 1)$ and for all sufficiently large $k$, the following: 
	\begin{itemize}
	\item[(a)] The set $\gamma_{k} = \{Z \in \R \times B_{1/4} \, : \, \Theta_{V_{k}}(Z) \geq Q\}$ is the graph of a $C^{1, \alpha}$ function over an $(n-1)$-dimensional subspace $L_{k}$ of $\R^{n+1}$, and $\pi \, (\gamma_{k}) \to \Gamma_{v} \cap B_{1/4}$ in Hausdorff distance (in fact in $C^{1, \alpha}$) where $\pi \, : \, \R^{n+1} \to \{0\} \times \R^{n}$ is the orthogonal projection; 
	\item[(b)] $V_{k} \res (\R \times B_{1/4}) = \sum_{j=1}^{Q} |{\rm graph} \, u_{k}^{j}|,$ where $u_{k}^{j} \, : \, B_{1/4} \to \R$, $u_{k}^{1} \leq u_{k}^{2} \leq \cdots \leq u_{k}^{Q},$  and for each $k$, if $\Omega_{k}^{\pm}$ denote the two components of $B_{1/4} \setminus \pi(\gamma_{k}),$ then $\left.u_{k}^{j}\right|_{\Omega_{k}^{\pm}} \in C^{1, \alpha}(\overline{\Omega_{k}^{\pm}} \cap B_{1/4})$ with $|u_{k}^{j}|_{1, \alpha; \Omega_{k}^{\pm}} 
	\leq C \hat{E}_{V_{k},\{0\} \times \R^{n}},$ where $C = C(n, Q) \in (0, \infty).$  
	\end{itemize}
	
	 Now fix $\zeta\in C^1_c(B_{1/4}(0);\R)$ and let $\tilde{\zeta}\in C^1_c(\R \times B_{1/4};\R)$ be the extension of $\zeta$ to 
	 $\R\times B_{1/4}$ such that $\tilde{\zeta}(x^1,x^2,\dotsc,x^{n+1}) = \zeta(x^2,\dotsc,x^{n+1})$. Taking the vector field $X = \tilde{\zeta}(x)e_p$, $p\in \{2,\dotsc,n\}$, in the first variation formula for $V_k$, to get
	$$\int \nabla^{V_k}x^p\cdot \nabla^{V_j}\tilde{\zeta}(x)\ \ext\|V_k\|(x) = 0.$$
	By a direct computation, this gives that
	$$\sum_{j=1}^{Q}\int_{B_{1/4}} \left(D_{p} \zeta  - \frac{(Du_{k}^{j} \cdot D\zeta)D_{p}u_{k}^{j}}{1 + |Du_{k}^{j}|^{2}}\right) 
	\sqrt{1 + |Du_{k}^{j}|^{2}} = 0$$
	which implies, since $\int_{B_{1/4}} D_{p} \zeta = 0$, that 
		$$\sum_{j=1}^{Q}\int_{B_{1/4}} \left(\frac{|Du_{k}^{j}|^{2}}{1 + \sqrt{1 + |Du_{k}^{j}|^{2}}}D_{p} \zeta  - \frac{(Du_{k}^{j} \cdot D\zeta)D_{p}u_{k}^{j}}{1 + |Du_{k}^{j}|^{2}}\right) 
	\sqrt{1 + |Du_{k}^{j}|^{2}} = 0.$$

	Given a vector field $\zeta  = (\zeta^{1}, \ldots, \zeta^{n}) \in C^{1}_{c}(B_{1/4}; \R^{n}),$ we can take $\zeta = \zeta^{p}$ in this, sum over $p$, divide the resulting identity by $\hat{E}_{V_{k}, \{0\} \times \R^{n}}^{2}$ and let $k \to \infty$ along a subsequence; using the 
	estimate and the convergence $\partial \, \Omega_{k}^{\pm}   = \pi(\gamma_{k}) \cap B_{1/4} \to \Gamma_{v} \cap B_{1/4}$ provided by (a) and (b) above, this gives that the desired identity (with $r/4$ in place of the original $r$).      
\end{proof}

Now we may prove that the squeeze identity holds for $v \in \FB_{Q}$ in a region possibly containing branch points, under the assumption that $v$ is of class generalised-$C^1$ in that region.

\begin{lemma}[Squeeze identity for $GC^1$ coarse blow-ups]\label{squeeze}
	Let $v\in \FB_Q$ and let $\Omega\subset B_{1/2}(0)$ be open. Suppose that $v$ is of class $GC^1$ in $\Omega$. Then for every $\zeta\in C^1_c(\Omega;\R^n)$,
	$$\int_\Omega\sum^Q_{\alpha=1}\sum^n_{i,j=1}\left(|Dv^\alpha|^2\delta_{ij}-2D_iv^\alpha D_j v^\alpha\right)D_i\zeta^j = 0$$ and 
	$$\int_\Omega\sum^Q_{\alpha=1}\sum^n_{i,j=1}\left(|Dv_{f}^\alpha|^2\delta_{ij}-2D_iv_{f}^\alpha D_j v_{f}^\alpha\right)D_i\zeta^j = 0.$$
\end{lemma}

\begin{proof} The first assertion follows by writing $v^{\alpha} = v_{f}^{\alpha} + v_{a}$ in the second, and noting that $\sum_{\alpha = 1}^{Q} v_{f}^{\alpha} = 0$ and that (as can easily be seen by integrating by parts, as $v_{a}$ is harmonic) 
	$$\int_\Omega\sum^n_{i,j=1}\left(|Dv_{a}|^2\delta_{ij}-2D_iv_{a}D_j v_{a}\right)D_i\zeta^j = 0.$$ 
	
	To prove the second assertion, let $v\in \FB_Q$. Let $\zeta\in C^1_c(\Omega;\R^n)$. For $\delta>0$ small, consider a smooth function $\gamma_\delta:\R\to \R$ such that $\gamma_\delta(t) = t+\frac{3}{4}\delta$ for $t<-\delta$, $\gamma_\delta(t)\equiv 0$ for $|t|<\delta/2$, and $\gamma_\delta(t) = t-\frac{3}{4}\delta$ for $t>\delta$, in such a way that $|\gamma'_\delta(t)|\leq 1$ and $|\gamma''_\delta(t)|\leq 3\delta^{-1}$ for all $t\in \R$. 
	
	We know from $(\FB4\text{II})$ that $\Gamma_v \cap \Omega \subset\{|v_f|=0\} \cap \Omega$ and that $\left.\gamma_\delta(v_f^\alpha)\right|_{\Omega}$ is a smooth function for each $\alpha=1,\dotsc,Q$. Thus by direct calculation and the dominated convergence theorem we have (using summation convention)
	\begin{align*}
		&\lim_{\delta\to 0}\int_{B_1(0)}\sum_{\alpha=1}^Q\left(|D\gamma_\delta(v_f^\alpha)|^2\delta_{ij} - 2D_i\gamma_\delta(v_f^\alpha)D_j\gamma_\delta(v_f^\alpha)\right)D_i\zeta^j \\
		&\hspace{2in}= \int_{B_1(0)}\sum^Q_{\alpha=1}\left(|Dv_f^\alpha|^2\delta_{ij} - 2D_i v_f^\alpha D_j v_f^\alpha\right)D_i\zeta^j.
	\end{align*}
	We now compute the limit on the left hand side in an alternative way. Let $\epsilon>0$. Since $v$ is generalised-$C^1$ in $\Omega$, we know that $|Dv_f| = 0$ on $\B_v \cap \Omega$. For each $x\in \CC_v \cap \Omega$, let $\rho_x>0$ denote the radius $r(v,x_0)>0$ from Lemma \ref{squeeze_classical}. Then since $\spt(\zeta)$ is compact we can pick finitely many points 
	$x_1, x_{2}, \ldots, x_{N} \in \CC_v \cap \Omega$ for which:
	$$\CC_v\cap \spt(\zeta)\cap B_\epsilon(\B_v)^c \subset \bigcup^N_{i=1}B_{\rho_{x_i}}(x_i).$$
	Then choose an open set $\mathcal{O}$ such that
	$$\dist(\mathcal{O},\CC_v\cup \B_v)>0\ \ \ \ \text{and}\ \ \ \ \spt(\zeta)\subset\mathcal{O}\cup B_{\epsilon}(\B_v)\cup\bigcup^N_{i=1}B_{\rho_{x_i}}(x_i).$$
	(e.g.\ take $\mathcal{O} = \Omega \setminus \overline{B_{\epsilon}(\B_{v}) \cup \cup_{i=1}^{N} B_{\rho_{x_{i}}}(x_{i})}$.)  
	Write $\mathcal{U}:= \{\mathcal{O}\}\cup \{B_\epsilon(\B_v)\}\cup\{B_{\rho_{x_i}}(x_i): i=1,\dotsc,N\}$; this is an open cover of $\spt(\zeta)$. Now let $(\phi_\beta)_{\beta\in A}$ be a smooth partition of unity subordinate to $\mathcal{U}$ (which of course depends on $\epsilon$ but is independent of $\delta$); note that:
	\begin{enumerate}
		\item [(I)] The indexing set $A$ has cardinality $|A|<\infty$ (i.e. is a finite number, depending on $\epsilon$);
		\item [(II)] $\mathcal{O}$ is contained in the set where $v_{f}$ is harmonic;
		\item [(III)] For each $\beta\in A$, there is some $B\in \mathcal{U}$ for which $\spt(\phi_\beta)\subset B$. This means that we can write $A$ as a disjoint union $A_{\CC}\cup A_{\B}\cup A_{\mathcal{O}}$ depending on the set in which the support of $\phi_\beta$ lies; indeed, $A_{\mathcal{O}} := \{\beta\in A: \spt(\phi_\beta)\subset \mathcal{O}\}$, $A_{\CC} := \{\beta\in A\setminus A_{\mathcal{O}}: \spt(\phi_\beta)\subset B_{\rho_{x_i}}(x_i)\text{ for some }i\in \{1,\dotsc,N\}\}$, and $A_\B := A\setminus (A_{\mathcal{O}}\cup A_{\CC})$;
		\item [(IV)] $\sum_{\beta\in A}\phi_\beta = 1$.
	\end{enumerate}
	Also, by Remark 4 following Definition~\ref{genC1}, since $v$ is generalised-$C^1$ and $|Dv_f| = 0$ on $\B_v \cap \Omega$ we have that
	$$\sup_{B_\epsilon(\B_v)\cap \spt(\zeta)\cap \CC_v^c}\sum_{\alpha=1}^Q|Dv_f^\alpha| \to 0 \ \ \ \ \text{as }\epsilon\to 0.$$
	Now consider the integral:
	\begin{equation}\tag{$\dagger$}\label{dagger}
	\int_{B_1(0)}\sum_{\alpha=1}^Q\left(|D\gamma_\delta(v_f^\alpha)|^2\delta_{ij} - 2D_i\gamma_\delta(v_f^\alpha)D_j\gamma_\delta(v_f^\alpha)\right)D_i(\phi_\beta \zeta^j)\ \ext x.
	\end{equation}
	If $\beta\in A_{\B}$, then by integrating by parts (which we can do as $\gamma_\delta(v_f^\alpha)$ is smooth) we have that this is equal to:
	\begin{align*}
		-\int_{B_1(0)}\sum_{\alpha=1}^Q\left[2D_\ell \gamma_\delta(v_f^{\alpha})D_{i\ell}\gamma_\delta(v_f^\alpha)\delta_{ij} - 2\Delta\gamma_\delta(v_f^\alpha)D_j\gamma_\delta(v_f^\alpha) - 2D_i\gamma_\delta(v_f^\alpha)D_{ij}\gamma_\delta(v_f^\alpha)\right]\phi_\beta\zeta^j
	\end{align*}
Since pointwise on $\spt(\gamma_\delta')$ we have $\Delta\gamma_\delta(v_f^\alpha) = \gamma''_\delta(v_f^\alpha)|Dv_f^\alpha|^2,$ we get:
	\begin{align*}
		\sum_{\beta\in A_{\B}}\int_{B_1(0)}\sum_{\alpha=1}^Q(|D\gamma_\delta(v_f^\alpha)|^2\delta_{ij}& - 2D_i\gamma_\delta(v_f^\alpha)D_j\gamma_\delta(v_f^\alpha))D_i(\phi_\beta\zeta^j)\\
		& = 2\int_{B_1(0)}\sum_{\alpha=1}^Q\gamma''_\delta(v_f^\alpha)|Dv_f^\alpha|^2D\gamma_\delta(v_f^\alpha)\cdot\zeta\cdot\left(\sum_{\beta\in A_{\B}}\phi_\beta\right).
	\end{align*}
	Now as $\spt(\gamma''_\delta)\subset[-\delta,\delta]$ and $\spt(\phi_\beta)\subset B_\epsilon(\B_v)$ for $\beta\in \A_\B$, the integral on the right hand side above (for each $\alpha$) only takes place over the set $\{|v_f^\alpha|<\delta\}\cap B_\epsilon(\B_v)$. Thus, since $|\gamma^{\prime}_{\delta}| \leq 1$, we can estimate it, in absolute value, by:
	$$\leq 2\sup_\R |\gamma_\delta''| \cdot \sup_{\spt(\zeta)\cap B_\epsilon(\B_v) \cap \CC_{v}^{c}}\sum_{\alpha=1}^Q|Dv_f^\alpha|\cdot\sup_{B_1(0)}|\zeta|\cdot\int_{B_{1/2}(0)}\sum_{\alpha=1}^Q \one_{\{|v_f^\alpha|<\delta\}} |Dv_f^{\alpha}|^2$$
	and so using the energy non-concentration estimate, Lemma \ref{noncon2}, and the fact that $\sup|\gamma_\delta''|\leq 2\delta^{-1}$, we see that for $\delta$ sufficiently small,
	$$\left|\sum_{\beta\in A_\B}\int_{B_1(0)}\sum^Q_{\alpha=1}\left(|D\gamma_\delta(v_f^\alpha)|^2\delta_{ij} - 2D_i\gamma_\delta(v_f^\alpha)D_j\gamma_\delta(v_f^\alpha)\right)D_i(\phi_\beta\zeta^j)\right| \leq C\sup_{\spt(\zeta)\cap B_{\epsilon}(\B_v)\cap \CC_v^c}\sum^Q_{\alpha=1}|Dv_f^\alpha|,$$
	where $C$ is independent of $\delta$ and $\epsilon$.	
	Now suppose $\beta\in A_{\mathcal{O}}$ in (\ref{dagger}). Then (\ref{dagger}) is equal to
	$$\int_{B_1(0)}\sum^Q_{\alpha=1}|\gamma'_\delta(v_f^\alpha)|^2\left(|Dv_f^{\alpha}|^2\delta_{ij} - 2D_i v_f^{\alpha}D_j v_f^{\alpha}\right)D_i(\phi_\beta \zeta^j)$$
	which as $\delta\downarrow 0$, converges to
	$$\int_{B_1(0)}\sum_{\alpha=1}^Q\left(|Dv_f^\alpha|^2\delta_{ij} - 2D_iv_f^\alpha D_jv_f^\alpha\right)D_i(\phi_\beta\zeta^j).$$
	By integrating by parts and making a pointwise calculation along $x\in \mathcal{O}$, this is equal to $0$ (note that the $v_f^\alpha$ are harmonic at such a point).
	
	Finally, suppose $\beta\in \A_\CC$ in (\ref{dagger}). Then we have:
	\begin{align*}
		&\int \left(|D\gamma_\delta(v_f^\alpha)|^2\delta_{ij} - 2D_i\gamma_\delta(v_f^\alpha)D_j\gamma_\delta(v_f^\alpha)\right)D_i(\zeta^{j}\phi_\beta)\\ 
		& \hspace{2in} = \int |\gamma'_\delta(v_f^\alpha)|^2\left(|Dv_f^\alpha|^2\delta_{ij} - 2D_iv_f^\alpha D_j v_f^\alpha\right)D_i(\zeta^j\phi_\beta)\\
		& \hspace{2.5in}\to \int\left(|Dv_f^\alpha|^2\delta_{ij} - 2D_i v_f^\alpha D_j v_f^\alpha\right)D_i(\zeta^j\phi_\beta)
	\end{align*}
	as $\delta\downarrow 0$, and moreover this final expression equals zero by Lemma \ref{squeeze_classical}(iii).
	
	Thus summing the above three expressions over $\beta\in A$, letting $\delta\to 0$ first, and then letting $\epsilon\downarrow 0$, we get the claimed identity.
\end{proof}

For  $v\in \FB_Q$ that is generalised-$C^{1}$ on an open set $\Omega \subset B_{1}$, and for $B_\rho(y)\subset \Omega$ we define:
$$N_{v;y}(\rho):= \frac{\rho^{2-n}\int_{B_\rho(y)}|Dv|^2}{\rho^{1-n}\int_{\del B_{\rho}(y)}|v|^2}.$$
This is well-defined whenever $v$ is not identically zero on $B_\rho(y)$; indeed, if $\int_{\del B_{\rho}(y)}|v|^2 = 0$, then the squash inequality, Lemma \ref{squash}, for $v$ would imply that $|Dv|=0$ on $B_\rho(y)$, and hence $v\equiv 0$ on $B_\rho(y)$.

By combining the squash inequality (Lemma \ref{squash}) and squeeze identity (Lemma \ref{squeeze}) we can prove in the standard fashion (see e.g. \cite{almgrenalmgren} or \cite{simon2016frequency}) that, where defined, $N_{v;y}(\rho)$ is a non-decreasing function of $\rho$. In particular for each $y\in \Omega$ either there is a $\rho>0$ for which $v|_{B_\rho(y)}\equiv 0$, or the limit
$$N_{v}(y):= \lim_{\rho\downarrow 0}N_{v;y}(\rho)$$
exists (in $[0, \infty)$). When defined, we call $N_v(y)$ the \textit{frequency} of $v$ at $y$. It is also a standard consequence of monotonicity of 
$N_{v; y}(\cdot)$ that if $\Omega$ is connected, then either $v \equiv 0$ in $\Omega$ or $v \not\equiv 0$ on each ball $B_{\rho}(y) \subset \Omega$; hence, if $\Omega$ is connected, then  
$N_{v;y}$ is well-defined  for each $y \in \Omega$ and $\rho \in (0, {\rm dist} \, (y, \partial \Omega))$ unless $v \equiv 0$ in $\Omega$. As the squash and squeeze identities also hold for $v_f,$ we can also define $N_{v_f;y}(\rho)$ and $N_{v_f}(y)$ in the same fashion, and the preceding facts hold with $v_{f}$ in place of $v$. Thus we have the following:

\begin{theorem}[Monotonicity of the frequency function]\label{frequency}
	Let $v\in \FB_Q$ and suppose that $v$ is of class $GC^1$ on a connected open set $\Omega \subset B_{1/2}$, and that $v \not\equiv 0$ in $\Omega$. Then $v \not\equiv 0$ on each ball 
	$B_{\rho}(y) \subset \Omega,$ $N_{v;y}(\rho)$ is well-defined  for each $y \in \Omega$ and $\rho \in (0, {\rm dist} \, (y, \partial \Omega)),$ and $N_{v; y}(\rho)$ is a monotonically non-decreasing function of $\rho;$ 
	in particular, the frequency $N_{v}(y) = \lim_{\rho \to 0} \, N_{v; y}(\rho)$ exists (as a number in $[0, \infty)$) for each $y \in \Omega$, and $N_{v}(y)$ is an upper semi-continuous function of $y$. Furthermore:
	\begin{itemize}
	\item[(i)] for each $y \in \Omega$ and  $0 < \sigma \leq \rho < {\rm dist} \, (y, \partial \, \Omega)$, we have that 
	$$\left(\frac{\sigma}{\rho}\right)^{2N_{v;y}(\rho)}\rho^{-n}\int_{B_\rho(y)}|v|^2 \leq \sigma^{-n}\int_{B_\sigma(y)}|v|^2 \leq\left(\frac{\sigma}{\rho}\right)^{2N_{v}(y)}\rho^{-n}\int_{B_\rho(y)}|v|^2;$$
  \item [(ii)] if $N_{v;y}(\rho)$ is constant for $\rho\in (\rho_1,\rho_2)$, then $v$ is a homogeneous function with respect to the variable 
  $|x-y|$ on the interval $(\rho_{1}, \rho_{2})$, with degree of homogeneity equal to the constant value of $N_{v;y}(\rho)$ on this interval.
	\end{itemize}
	
	Moreover, the same conclusions hold with $v_f$ in place of $v$ whenever $v_{f} \not\equiv 0$ in $\Omega$.
\end{theorem}

Armed with properties $(\FB1)-(\FB6)$ (subsection~\ref{initial-cb}) and the further properties established in Sections~\ref{continuity}-\ref{C1squeeze}, in this section we prove that any coarse blow-up $v\in \FB_Q$ is  of class $GC^{1,\alpha}$ for some fixed $\alpha = \alpha(n,Q) \in (0, 1),$ and that $v$ satisfies a uniform decay estimate. 
The first step of the proof is to classify the homogeneous degree 1 blow-ups as having graphs that are either hyperplanes or classical cones. This is done by first proving that if a homogeneous degree 1 coarse blow-up $v$ has a subspace of translation invariance (spine) of dimension $\leq n-2$, then $GC^1$ regularity must hold away from the spine; this ensures the validity of the squeeze identity, and consequently the monotonicity of the frequency function associated with $v$. Since the degree of homogeneity of $v$ is $1$, this implies that $\Gamma_{v}$ is contained in the spine, but by property $(\FB4\textnormal{II})$ this is impossible since the spine has zero 2-capacity. Thus the spine dimension of any homogeneous degree 1 coarse blow-up must be $ \geq n-1$, in which case it is easy to see that the classification must hold. With this classification at our disposal, in the second step we run again the regularity argument (that gives $GC^{1}$ regularity away from the spine for homogeneous degree 1 blow-ups in the first instance), but now for general coarse blow-ups. This will prove $GC^{1,\alpha}$ regularity of general coarse blow-ups, this time giving also a uniform decay estimate and a uniform $\alpha = \alpha (n, Q) \in (0, 1)$.

\subsection{Classification of homogeneous degree 1 coarse blow-ups}

\begin{theorem}[Classification of homogeneous degree 1 blow-ups]\label{classification}
If $v\in \FB_Q$ is homogeneous of degree 1 in $B_1$ then either there is a varifold $W \in {\mathcal C}_{Q}$ such that ${\mathbf v}(v) = W \res (\R \times B_{1})$  or 
 ${\mathbf v}(v) = Q|L| \res (\R \times B_{1})$ for some hyperplane $L$.
\end{theorem}

In the proof of this theorem we shall use the following elementary result, which is a slight variant of a well-known fact concerning limits of rescalings of homogeneous functions: 

\begin{lemma}\label{splitting}
For $i=1, 2, 3, \ldots,$ let $g_{i}  \, : \, B_{1} \to {\mathcal A}_{Q}(\R)$ be a continuous function such that $g_{i}$ is homogeneous of degree 1 in $B_{1}$, i.e.\ 
$g_{i}(\lambda \, x) = \lambda g_{i}(x)$ whenever $\lambda >0$ and $\lambda  x, x \in B_{1}$. Let $(z_{i})$ be a sequence of points in $B_{1}$ with  
$z_{i} \to z$ for some $z \in B_{1}$ and let $(\rho_{i})$ be a sequence of positive numbers with $\rho_{i} \to 0$. Let 
$h_{i}(x) = g_{i}(z_{i} + \rho_{i}x).$ If $h_{i} \to h$ locally uniformly on $B_{1}$, $h$ is homogeneous of degree 1 in $B_{1}$, and if $\widetilde{h}$ is the homogeneous degree 1 extension of $h$ to ${\mathbb R}^{n}$, then $\widetilde{h}(x + tz) = \widetilde{h}(x)$ for   all $x \in {\mathbb R}^{n}$ and $t \in \R$.
\end{lemma} 

\begin{proof} By homogeneity of $\widetilde{h}$, it suffices to verify that $\widetilde{h}(x + z) = \widetilde{h}(x)$ for each $x \in \R^{n}$. If $x, x+z \in B_{1}$, then for sufficiently large $i$ we have,
$h_{i}(x+ z) = g_{i}(z_{i} + \rho_{i}(x + z)) = g_{i}((1 + \rho_{i})z_{i} + \rho_{i}(z - z_{i}) + \rho_{i}x) = (1+ \rho_{i}) g_{i}(z_{i} + \rho_{i}(1+\rho_{i})^{-1}(z - z_{i} + x))
=(1 + \rho_{i}) h_{i}((1 + \rho_{i})^{-1}(z- z_{i} + x))$. Letting $i \to \infty$ in this, we get that $\widetilde{h}(x + z) = \widetilde{h}(x)$ if $x, x+z \in B_{1}.$ Replacing $x$ with $x-z$, we also have that $\widetilde{h}(x - z) = \widetilde{h}(x)$ if $x, x-z \in B_{1}.$ 
Now fix any   
$x \in B_{\frac{1}{2}(1 - |z|)}(0).$ Then for any $t \in (0, 1)$, we have that $|x| < (1 - |z|) \max\{t, 1-t\}$, i.e.\ either $\frac{x}{t} \in B_{1-|z|}(0)$ or $\frac{x}{1-t} \in B_{1-|z|}(0)$; if $\frac{x}{t} \in B_{1-|z|}(0),$ then 
$\widetilde{h}(x + t z) = t\widetilde{h}\left(\frac{x}{t} + z\right) = t \widetilde{h}\left(\frac{x}{t}\right) = \widetilde{h}(x)$, and if 
$\frac{x}{1-t} \in B_{1-|z|}(0)$ then $\widetilde{h}(x + t z) = \widetilde{h}(x + tz - z) = (1-t) \widetilde{h}\left(\frac{x}{1-t}  - z\right) =(1-t) \widetilde{h}\left(\frac{x}{1-t}\right) = \widetilde{h}(x).$ Thus we have shown that $\widetilde{h}(x + t z) = \widetilde{h}(x)$ for any $x \in B_{\frac{1}{2}(1 - |z|)}(0)$ and any $t \in [0, 1]$.  This also implies that for any $x \in \R^{n} \setminus 
B_{\frac{1}{2}(1 - |z|)}(0)$, setting $t= \frac{1}{4}(1 - |z|) |x|^{-1}$ and noting that $t \in [0, 1]$ and $tx \in B_{\frac{1}{2}(1 - |z|)}(0)$, we have $\widetilde{h}(x + z) = t^{-1} \widetilde{h}(t x + t z) =  t^{-1}\widetilde{h}(t x)= \widetilde{h}(x).$ Thus $\widetilde{h}(x + z) = \widetilde{h}(x)$ for all $x \in \R^{n}$ as required. \end{proof}

\begin{proof}[Proof of Theorem~\ref{classification}]
	First note that $v_a$ is harmonic (by property $(\FB3)$) and homogeneous of degree 1 on $B^n_1(0)$, and so is linear. Hence 
	if $v = Q\llbracket v_{a} \rrbracket$ then the conclusion holds with ${\bf v}(v) = Q|L| \res (\R \times B_{1})$ where $L = {\rm graph} \, v_{a}$ (a hyperplane). Else by $(\FB5\text{III})$ we  have that 
	$\|v - v_{a}\|_{L^{2}(B_{1}(0))} ^{-1}(v - v_{a})\in \FB_{Q}.$  So it suffices to establish the theorem (with the conclusion 
	${\bf v}(v)  = W \res (\R \times B_{1})$ for some $W \in {\mathcal C}_{Q}$) for $v \in \widetilde{\FB}_{Q}$
where 	$$\widetilde{\FB}_{Q} = \left\{v \in \FB_{Q} \, : \, \frac{\del(v/R)}{\del R} = 0 \; \mbox{a.e. in} \; B_1, \; v_{a} = 0, \; 
	\|v\|_{L^{2}(B_{1}(0))} = 1\right\}.$$
	 For each $v \in \widetilde{\FB}_{Q}$, 
	let $\widetilde{v}:\R^n\to {\mathcal A}_{Q}(\R)$ denote the homogeneous degree 1 extension of $v$ to $\R^{n}$. Denote by $S(\widetilde{v})$ the set of points $z \in \R^{n}$ such that $\widetilde{v}$ is invariant under translation by $z$ (i.e.\ $\widetilde{v}(x + z) = \widetilde{v}(x)$ for every $x \in \R^{n}$), and note that by homogeneity of $\widetilde{v}$, we have that $S(\widetilde{v})$ is a linear subspace of $\R^{n}$. So we can write $\widetilde{\FB}_{Q}= \cup_{k=0}^n\H_k$, where $\H_k:= \{v\in \widetilde{\FB}_{Q} \, : \, \dim(S(\widetilde{v})) = n-k\}$. Note $\H_0 = \emptyset$ since $\|v\|_{L^{2}(B_{1})} = 1$ and $v_a=0$ for each $v \in \widetilde{\FB}_{Q}$. If $v \in \H_{1}$  then by homogeneity we have $\mathbf{v}(\widetilde{v})\in \CC_Q$ and so the conclusion (with $W = {\mathbf v}(\widetilde{v})$) follows.
So to prove the theorem, we need to show that $\H_{k} = \emptyset$ for each $k=2, 3, \ldots, n$. Assuming this is false, let $d$ be the smallest integer in $\{2,3,\dotsc,n\}$ for which $\H_{d} \neq\emptyset$, and fix an element $v \in \H_{d}.$ 
	
	We claim the following: \emph{for any given compact subset $K\subset B_1\setminus S(\tilde{v})$, and any $\alpha \in (0, 1)$, there exists 
	$\epsilon = \epsilon(v,K,n,Q, \alpha)\in (0,{\rm dist} \, (K, \partial B_1 \cup S(\widetilde{v})))$ such that the following holds}: for each $z\in K\cap \Gamma_{v}$, each $\rho\in (0,\epsilon]$ and some fixed constant $C = C(n, Q, \alpha) \in (0, \infty)$ either: 
	\begin{itemize}
	\item[(a)] The conclusions of Theorem \ref{thm:B7} hold on $B_{3\rho/8}(z)$; in particular $v$ is a generalised-$C^{1,\alpha}$ function in $B_{3\rho/8}(z)$ and there is a function $\psi_{z} \, : \, \R^{n} \to {\mathcal A}_{Q}(\R)$ with ${\mathbf v}(\psi_{z}) \in \CC_{Q}$ such that
	\begin{equation}\label{E:Estimate}
	\widetilde{\rho}^{-n-2}\int_{B_{\widetilde{\rho}}(z)}{\mathcal G}(v(x), \psi_{z}(x))^{2} \, \ext x \leq C\left(\frac{\widetilde{\rho}}{\rho}\right)^{2\alpha}\cdot\rho^{-n-2}\int_{B_\rho(z)}|v|^2
	\end{equation}
	for all $0<\widetilde{\rho}\leq 3\rho/8$; or 
	\item[(b)] We have the \textit{reverse Hardt-Simon inequality}, i.e.,
	\begin{equation}\label{E:RHS}
	\sum^Q_{\alpha=1}\int_{B_\rho(z)\setminus B_{\rho/2}(z)}R_z^{2-n}\left(\frac{\del(v^\alpha/R_z)}{\del R_z}\right)^2 \geq \epsilon\rho^{-n-2}\int_{B_\rho(z)}|v|^2\;,
	\end{equation}
	where $R_z:= |x-z|.$  
	\end{itemize}
	To prove this we argue by contradiction, so suppose the claim is not true (with $C$ to be chosen depending only on $n,$ $Q$ and $\alpha$). Then for each $i=1,2,\dotsc$, there are numbers $\epsilon_{i} > 0$ with $\epsilon_i\to 0$, points $z,z_i\in K\cap \Gamma_v$ with $z_i\to z$, and radii 
	$\rho_{i} >0$ with $\rho_i\to 0$ such that assertion (a) with $\rho = \rho_{i}$ and $z= z_{i}$ fails for each $i,$ and also 
\begin{equation}\label{homog-control} 	
\sum_{\alpha=1}^Q\int_{B_{\rho_i}(z_i)\setminus B_{\rho_i/2}(z_i)}R_{z_i}^{2-n}\left(\frac{\del(v^\alpha/R_{z_i})}{\del R_{z_i}}\right)^2 < \epsilon_i\rho_i^{-n-2}\int_{B_{\rho_i}(z_i)}|v|^2.
\end{equation}
Set $w_i:= v_{z_i,\rho_i}$, and note that $w_{i} \in \FB_Q$ by $(\FB5\text{I}),$ and $w_{i}$ is continuous on $B_{1}$ by Lemma~\ref{continuity}. By $(\FB6)$ and Lemma~\ref{continuity}, we can find a subsequence (which we pass to) and an element $w_*\in \FB_Q$ such that  
	$w_i\to w_*$ locally uniformly and locally weakly in $W^{1,2}$ on $B_{1}$.  Moreover, since $v_{a} = 0$, we also have that 
	$(w_{\ast})_{a} = 0$.

Now note that for any function $u \in C^{1}(\overline{B}_{1}; \R^{Q})$, any $r, s \in [1/2, 1]$ and any $\omega \in {\mathbb S}^{n-1}$, we have that $\left|\frac{|u(r\omega)|}{r} - \frac{|u(s\omega)|}{s}\right| \leq \int_{1/2}^{1} \left|\frac{\ext(u(t\omega)/t)}{\ext t}\right|\ext t,$ which implies, by the triangle inequality and the Cauchy--Schwarz inequality, 
$$|u(r\omega)|^{2} \leq C\left(|u(s\omega)|^{2} + \int_{1/2}^{1}t^{n-1}\left|\frac{\ext (u(t\omega)/t)}{\ext t}\right|^{2} \ \ext t\right),$$
where $C = C(n),$ which in turn gives 
$$\int_{{\mathbb S}^{n-1}}|u(r\omega)|^{2} \, \ext\omega 
\leq C\left(\int_{{\mathbb S}^{n-1}} |u(s\omega)|^{2}  \, \ext\omega + \int_{B_{1} \setminus B_{1/2}} \left|\frac{\del (u/R)}{\del R}\right|^{2}\right),$$
where $R(x) = |x|$. Multiplying this by $r^{n-1}$ and integrating over $r \in [1/2, 1]$, and then multiplying the resulting inequality by $s^{n-1}$ and integrating it over $s \in [1/2, 3/4],$ we obtain, after rearranging terms, 
$$\int_{B_{1}}|u|^{2}  
\leq C\left(\int_{B_{3/4}} |u|^{2}  + \int_{B_{1} \setminus B_{1/2}} \left|\frac{\del (u/R)}{\del R}\right|^{2}\right),$$
where $C = C(n).$ By an approximation argument this holds for any $u \in W^{1, 2}(B_{1}; \R^{Q})$. Applying this with $u = w_{i}$  we get, by virtue of (\ref{homog-control}) and the fact that $\|w_i\|_{L^2(B_1)}=1,$ 
that $\int_{B_{3/4}} |w_{i}|^{2} > C^{-1}$ for all sufficiently large $i$, and hence that $w_*\not\equiv 0$. 

 Also we see from (\ref{homog-control}) that $w_*$ is homogeneous of degree 1 in $B_1\setminus B_{1/2},$ and hence by Lemma~\ref{unique-cont}  $w_{\ast}$ is homogeneous of degree 1 in $B_1.$ Let $\widetilde{w}_*$ denotes the homogeneous degree 1 extension of $w_{\ast}$ to $\R^{n}.$ By Lemma~\ref{splitting} (applied with $g_{i}(x) = \|v(z_{i} + \rho_{i} (\cdot))\|^{-1}_{L^{2}(B_{1})}v(x)$), we see that 
	$\{t z \, : \, t \in \R\} \subset S(\widetilde{w}_{\ast})$. Since we also have that 
	$S(\widetilde{v}) \subset   S(\widetilde{w}_{\ast}),$ $z \not\in S(\widetilde{v})$ (since $z \in K$) and that $S(\widetilde{w}_{\ast})$ is a linear subspace of $\R^{n}$, we must have   
	$\dim(S(\widetilde{w}_*))\geq d+1$; thus to prevent a contradiction to the definition of $d$, we must have 
	$\dim(S(\widetilde{w}_*)) \in \{n-1,n\}$. However as $(w_*)_a\equiv 0$ and $w_*\not\equiv 0$, we cannot have $\dim(S(\widetilde{w}_*)) = n$, and so we must have $\dim(S(\widetilde{w}_*)) = n-1$, and hence 
	$\mathbf{v}(w_*) = W_{\ast} \res (\R \times B_{1})$ for $W_{\ast}  = {\mathbf v}(\widetilde{w}_{\ast}) \in \CC_Q$. But then we can apply the $\epsilon$-regularity property, Theorem~\ref{thm:B7}, with $\psi = w_{\ast}$ and $v=(w_{i})_{0, 3/4} \equiv \|w_{i}(\frac{3}{4}(\cdot)\|_{L^{2}(B_{1})}^{-1}w_{i}(\frac{3}{4}(\cdot))$ for all sufficiently large $i$ to conclude that 
	alternative (a) above must hold with $z = z_{i}$, $\rho = \rho_{i},$ $\BC_{z_{i}} = {\bf v}(\|v(z_{i} + \rho_{i}(\cdot)\|_{L^{2}(B_{1})}w_{\ast})$ and with the constant $C = C(n, Q, \alpha)$ given by Theorem~\ref{thm:B7}. This is contrary to our assumption, so the dichotomy that (a) or (b) must hold is established.

	Combining (\ref{E:RHS}) with $(\FB4\text{I})$, we then get the following dichotomy: if $z\in K\cap \Gamma_v$ and $\rho\in(0,\epsilon]$ then either:
	\begin{enumerate}
		\item [(i)] The conclusions of Theorem \ref{thm:B7} hold on $B_{3\rho/8}(z)$; in particular $v$ is generalised-$C^{1,\alpha}$ on $B_{3\rho/8}(z)$ and there is a function $\psi_{z} \, : \, \R^{n} \to {\mathcal A}_{Q}(\R)$ with ${\mathbf v}(\psi_{z}) \in \CC_Q$ for which
		we have the estimate (\ref{E:Estimate}), or
		\item [(ii)] We have that (\ref{E:RHS}) holds and that 
		$$\sum^Q_{\alpha=1}\int_{B_{\rho/2}(z)}R_{z}^{2-n}\left(\frac{\del(v^{\alpha}/R_z)}{\del R_z}\right)^2 \leq \theta\sum^Q_{\alpha=1}\int_{B_\rho(z)}R_z^{2-n}\left(\frac{\del(v^\alpha/R_z)}{\del R_z}\right)^2,$$
	\end{enumerate}
	where $\theta = \theta(v,K,n,Q)\in (0,1)$. We claim that from this, the following dichotomy (I) or (II) follows for each $z \in K \cap \Gamma_{v}$: either 
		\begin{enumerate}
		\item [(I)] The conclusions of Theorem \ref{thm:B7} holds on some neighbourhood of $z$, and moreover there is a function $\psi_{z} \, : \, \R^{n} \to {\mathcal A}_{Q}(\R)$ with ${\mathbf v}(\psi_{z}) \in \CC_Q$ for which
		we have the estimate
		$$\rho^{-n-2}\int_{B_\rho(z)}{\mathcal G}(v(x), \psi_{z}(x))^2 \, \ext x \leq C\rho^{2\mu}\int_{B_\epsilon(z)}|v|^2$$
		for some $C = C(v,K,n,Q)$ and all $\rho\in (0,3\epsilon/8]$; or
		\item [(II)] We have that (\ref{E:RHS}) holds with $\rho = 2^{-i}\epsilon$ for each $i = 1, 2, 3, \ldots$, and hence 
		$$\sum^Q_{\alpha=1}\int_{B_\sigma(z)}R_z^{2-n}\left(\frac{\del(v^\alpha/R_z)}{\del R_z}\right)^2 \leq \beta\left(\frac{\sigma}{\rho}\right)^{2\mu}\sum^Q_{\alpha=1}\int_{B_\rho(z)}R_z^{2-n}\left(\frac{\del(v^\alpha/R_z)}{\del R_z}\right)^2$$
		for all $0<\sigma\leq \rho/2\leq\epsilon/2.$
	\end{enumerate}
	Here $\beta = \beta(v,K,n,Q)\in (0,\infty)$ and $\mu = \mu(v,K,n,Q)\in(0,1)$. 
	Indeed, for each fixed $z \in K \cap \Gamma_{v}$, the dichotomy (i) or (ii) above holds for $\rho = 2^{-i}\epsilon$, $i=0,1,2,\dotsc$. Let $I$ be the first time (i) holds (i.e.\ $I$ is the smallest integer $i \geq 0$ such that (i) holds with 
	$\rho = 2^{-i}\epsilon$). If $I= 0$ we have alternative (I); also if $I \geq 1$, then iterating (ii) for $i=0,1,\dotsc I-1$ and combining with the estimate provided in (i) as well as $(\FB4\text{I})$ and (\ref{E:RHS}), we again have (I). If 
	 if no such $I$ exists, i.e., if (ii) always holds, then iterating (ii) for all $i$ we get alternative (II).
	  
	Finally, in case (II) holds, we can again use $(\FB4\text{I})$ and (\ref{E:RHS}) in conjunction with the estimate in (II) to replace (II) with:
	\begin{enumerate}
		\item [(II)$^{\prime}$] For all $0<\sigma\leq \rho/2\leq \epsilon/4$ we have
		\begin{equation}\label{decay-on-branch-set}
		\sigma^{-n-2}\int_{B_\sigma(z)}|v|^2 \leq \beta\left(\frac{\sigma}{\rho}\right)^{2\mu}\rho^{-n-2}\int_{B_\rho(z)}|v|^2.
		\end{equation} 
	\end{enumerate}
	Note that the set of points $z \in \Gamma_{v} \cap {\rm int} \, K$ where alternative (I) holds is 
	${\mathcal C}_{\left. v\right|_{{\rm int} \, K}}$, and hence we have shown that for each point $z \in \Gamma_{v} \cap {\rm int} \, K \setminus 
	{\mathcal C}_{\left. v\right|_{{\rm int} \, K}}$, the estimate (\ref{decay-on-branch-set}) holds; from this it is straightforward to check (e.g.\ using the Campanato lemma \cite[Lemma 4.3]{wickramasekera2014general}) that $\left. v\right|_{{\rm int} \, K}$ is generalised-$C^{1,\mu}$ in ${\rm int} \, K$, where $\mu= \mu(v,K,n,Q)$. In particular as $K\subset B_1\setminus S(\widetilde{v})$ was an arbitrary compact set, we see that $v$ is generalised-$C^1$ in $B_1 \setminus S(\widetilde{v})$.
		
	We now claim that $\Gamma_{v}\subset S(\widetilde{v})$. Our method for showing this will rely on the frequency function, which now can be brought into play in view of the generalised-$C^1$ regularity of $v$ on $B_1\setminus S(\widetilde{v})$. Indeed, by Lemma~\ref{squeeze} we have that the squeeze identity 
	$$\int_{B_1(0)}\sum_{\alpha=1}^Q\left(|D\widetilde{v}^\alpha|^2\delta_{ij}-2D_{i}\widetilde{v}^\alpha D_j \widetilde{v}^\alpha\right)D_i\zeta^j\ \ext x = 0$$
holds for $\zeta^{j} \in C^{1}_{c}(B_{1} \setminus S(\widetilde{v}))$.  Since $S(\widetilde{v})$ is a linear subspace of dimension at most $n-2$, it has zero 2-capacity, and hence, since we also have that $D\widetilde{v}$ is bounded in $B_{1}\setminus S(\widetilde{v})$ by homogeneity of $\widetilde{v}$, we may perform a standard excision argument to see that in fact we can take $\zeta^{j} \in C^{1}_{c}(B_1(0)).$
		Armed with this squeeze identity and the squash inequality (Lemma \ref{squash}), we can use standard arguments (see e.g.\ \cite{simon2016frequency}) to show that (since $\widetilde{v} \not\equiv 0$ in $\R^{n}$) the frequency $N_{\widetilde{v}}(y)$ is well-defined at every point $y\in B_1(0),$  and in fact that all conclusions of Theorem~\ref{frequency} hold with $\widetilde{v}$ in place of $v$. In particular, as $\widetilde{v}$ is homogeneous of degree 1, we have that $N_{\widetilde{v}}(0) = 1$. It follows from upper semi-continuity of $N_{\widetilde{v}}$ and homogeneity of $\widetilde{v}$ that $N_{\widetilde{v}}(y) \leq N_{\widetilde{v}}(0) = 1$ for all $y\in B_1(0)$. Moreover frequency monotonicity and homogeneity of $\widetilde{v}$ give that if $N_{\widetilde{y}}(y) = N_{\widetilde{v}}(0) = 1$, then $\widetilde{v}$ is translation invariant along directions parallel to $y$, i.e. $y\in S(\widetilde{v})$, and so $S(\widetilde{v}) = \{y\in \R^n: N_{\widetilde{v}}(y) = N_{\widetilde{v}}(0) = 1\}$. We also have that for each $y\in B_1(0)$ and $0<\sigma\leq\rho$:
	\begin{equation}\label{E:in}
	\left(\frac{\sigma}{\rho}\right)^{2N_{\widetilde{v};y}(\rho)}\rho^{-n}\int_{B_\rho(y)}|\widetilde{v}|^2 \leq \sigma^{-n}\int_{B_\sigma(y)}|\widetilde{v}|^2 \leq\left(\frac{\sigma}{\rho}\right)^{2N_{\widetilde{v}}(y)}\rho^{-n}\int_{B_\rho(y)}|\widetilde{v}|^2.
	\end{equation}
	Now if $y\in \B_{v}\cap (B_1\setminus S(\widetilde{v}))$, then (II)$^{\prime}$ must hold for some $\epsilon = \epsilon(y) >0$ 
	(which can be taken to be the $\epsilon$  corresponding to some fixed compact set $K = K(y) \subset B_1\setminus S(\widetilde{v})$ with $y \in K$ in the argument leading to (I) and (II)$^{\prime}$); in particular, by combining (II)$^{\prime}$ with (\ref{E:in}), we have for fixed $\rho \in (0, \epsilon/2]$ and all $\sigma \in (0, \rho/2]$:
	$$\left(\frac{\sigma}{\rho}\right)^{2N_{\widetilde{v};y}(\rho)}\rho^{-n}\int_{B_\rho(y)}|\widetilde{v}|^2 \leq \sigma^{-n}\int_{B_\sigma(y)}|\widetilde{v}|^2 \leq \sigma^2\cdot\beta\left(\frac{\sigma}{\rho}\right)^{2\mu}\rho^{-n-2}\int_{B_\rho(y)}|\widetilde{v}|^2,$$
	and thus $\sigma^{1+\mu - N_{\widetilde{v};y}(\rho)} \geq \widetilde{C}>0$ for some $\widetilde{C} = \widetilde{C}(v,y,n,Q,\rho)$ and $\mu = \mu(v,y,n,Q)$. As we may take $\sigma\downarrow 0$, we see that $N_{\widetilde{v};y}(\rho)\geq 1+\mu$. As $\rho\in (0,\epsilon/2)$ was arbitrary, we can then take $\rho\downarrow 0$ to see that $N_{\widetilde{v}}(y)\geq 1+\mu>1 = N_{\tilde{v}}(0)$, which is a contradiction. Thus we must have $\B_{\widetilde{v}}\subset S(\widetilde{v})$. So if we have $\Gamma_{v}\not\subset S(\widetilde{v})$, then we can find $z\in B_1\setminus S(\widetilde{v})$ for which (I) holds. In particular, as each cone in $\CC_Q$ is determined by linear functions, we readily deduce from the estimate in (I) that $\rho^{-n}\int_{B_\rho(z)}|v|^2 \leq C\rho^{2}$ for all $\rho \in (0, \epsilon]$ and some $C$ independent of $\rho,$ so in the same way as above, we have $N_{\widetilde{v}}(z)\geq 1$, and thus $N_{\widetilde{v}}(z) = 1$, and so $z\in S(\widetilde{v})$, which is a contradiction. Thus we must have $\Gamma_{v}\subset S(\widetilde{v})$.
	
	Hence from $(\FB4)$ we know $\widetilde{v}$ is harmonic on $\R^n\setminus S(\widetilde{v})$. In particular as $d\equiv \dim(S(\widetilde{v}))$, this means that $\widetilde{v}$ is determined by a continuous function $f$, defined on $\R^{n-d}$, which is harmonic away from $0\in \R^{n-d}$, and thus, since $d \leq n-2$,  it follows that $0$ is a removable singularity of $f$, and so $\widetilde{v}$ is harmonic on all of $\R^n$ and hence is linear. By $(\FB2)$ and the fact that $\widetilde{v}(0) = 0$ (since $v$ is continuous) this implies that $\widetilde{v}^1\equiv \widetilde{v}^2\equiv\cdots \equiv \widetilde{v}^Q$ are all the same linear function, 
	contradicting the fact that $d\leq n-2$. We must therefore have $\H_{k} = \emptyset$ for all $k \in \{2, 3, \ldots, n\}$, and this completes the proof of the theorem.
\end{proof}

\subsection{Generalised-$C^{1, \alpha}$ regularity of coarse blow-ups}\label{sec:blow-up-reg}

Employing the classification of homogeneous degree 1 elements of $\FB_Q$, i.e., Theorem~\ref{classification}, and recycling its method of proof, we can now establish generalised-$C^{1,\alpha}$ regularity, together with a uniform decay estimate, for arbitrary elements of $\FB_Q$, for some fixed $\alpha = \alpha(n, Q) \in (0, 1)$. 

\begin{theorem}\label{coarse_reg}
	There exists $\alpha = \alpha(n,Q)\in (0,1)$ such that $\FB_Q\subset GC^{1,\alpha}(B_{1/2}(0);\A_Q(\R))$. Moreover
	 if $v\in \FB_Q$ and if $0 \in \Gamma_{v}^{\textnormal{HS}}$, then there exists $\phi:\R^{n} \to \A_Q(\R)$ with $\mathbf{v}(\phi)\in \CC_Q$ or 
	 ${\mathbf v}(\phi) = Q|L|$ for some hyperplane $L$ such that
	for every $\sigma$, $\rho$ with $0 < \sigma \leq \rho/2 \leq 3/16$, we have 
	$$\sigma^{-n-2}\int_{B_\sigma(0)}\G(v(x) - v_{a}(0),\phi(x))^2 \, \ext x\leq C\left(\frac{\sigma}{\rho}\right)^{2\alpha} \cdot \rho^{-n-2}\int_{B_{\rho}}|v|^2,$$
	where $C = C(n, Q) \in (0, \infty).$
\end{theorem}

\begin{proof}
In view of property $(\FB5\textnormal{I})$ and the fact that $\|v\|_{L^{2}(B_{1})} \leq 1$ for each $v \in \FB_{Q}$, it suffices to prove the claimed estimate for $\rho=3/8,$ so assume $\rho = 3/8$.  First note that we can repeat the first part of the argument of Theorem~\ref{classification}, leading to the dichotomy that (a) or (b) holds, to establish the existence of $\epsilon = \epsilon(n, Q) \in (0, 1/2)$ such that the same dichotomy ((a) or (b)) must hold with $z = 0$ and $\alpha = 1/2$ (say) for any $v \in \FB_{Q}$ 
such that $0 \in \Gamma_{v}^{\textnormal{HS}},$ provided we additionally assume that $v_{a}(0) = 0$ and $Dv_{a}(0) = 0$; specifically, there is $\epsilon = \epsilon(n, Q) \in (0, 1/2)$ such that for each $v \in \FB_{Q}$ with $0 \in \Gamma_{v}^{\textnormal{HS}}$, $v_{a}(0) = 0$ and $Dv_{a}(0) = 0,$ either: 
	\begin{itemize}
	\item[(a)] The conclusions of Theorem \ref{thm:B7} hold on $B_{9/64}(0)$; in particular $v$ is a generalised-$C^{1,1/2}$ graph on $B_{9/64}(0)$ and there is a function $\psi_{0} \, : \, \R^{n} \to {\mathcal A}_{Q}(\R)$ with ${\mathbf v}(\psi_{0}) \in \CC_{Q}$ such that 
	 we have the estimate
	\begin{equation}\label{E:Estimate2}
	\rho^{-n-2}\int_{B_{\rho}(0)}{\mathcal G}(v(x), \psi_{0}(x))^{2} \, \ext x \leq C\rho\int_{B_{3/8}(0)}|v|^2
	\end{equation}
	for all $0<\rho\leq 9/64$, where $C = C(n, Q) \in (0, \infty)$, or 
	\item[(b)] We have 
		\begin{equation}\label{E:RHS-general}
	\sum^Q_{\alpha=1}\int_{B_{3/8}(0)\setminus B_{3/16}(0)}R^{2-n}\left(\frac{\del(v^\alpha/R)}{\del R}\right)^2 \geq \epsilon\int_{B_{3/8}(0)}|v|^2,
	\end{equation}
	where $R(x)= |x|$.  
	\end{itemize}
To prove this we argue by contradiction exactly as in the proof of Theorem~\ref{classification}, but with $\|v_{i}\|_{L^{2}(B_{1})}^{-1}v_{i}$ taking the place of $w_{i} = v_{z_{i}, \rho_{i}}$ appearing in that argument, where $(v_{i}) \subset \FB_{Q}$  is a general sequence such that both options (a) and (b) with $v = v_{i}$ and $\epsilon = \epsilon_{i}$ are assumed to fail, with $\epsilon_{i} \to 0;$ note that this argument utilises the classification of homogeneous degree 1 elements provided by Theorem~\ref{classification} to reach a contradiction. 

Still subject to the conditions $v_{a}(0) = 0$, $Dv_{a}(0)=0$, this then leads to the final dichotomy, as in the proof of Theorem~\ref{classification} (and by the same argument), namely that either the statement (I) holds with $z= 0$ (and with the estimate 
$\rho^{-n-2}\int_{B_\rho(0)}{\mathcal G}(v(x), \psi_{0}(x))^2 \, \ext x \leq C\rho^{2\mu}\int_{B_{3/8}(0)}|v|^2$
		for  all $\rho\in (0,9/64]$), or the statement (II)$^{\prime}$ holds (with the estimate $\sigma^{-n-2}\int_{B_\sigma(0)}|v|^2 \leq \beta\left(\frac{\sigma}{\rho}\right)^{2\mu}\rho^{-n-2}\int_{B_\rho(0)}|v|^2$ for all $0<\sigma \leq \rho/2\leq 3/32$), 
where now the constants $C$, $\mu$, $\beta$ all depend only on $n$ and $Q$. In either case, this provides the desired estimate in the present theorem (with $L = \{0\} \times \R^{n}$, $\alpha = \mu$) in the special case $v_{a}(0) = 0$ and $Dv_{a}(0) = 0$. The claimed estimate in the general case (i.e.\ without the assumption $v_{a}(0) = 0$ and $Dv_{a}(0) = 0$) follows immediately from this special case in view of property $(\FB5\textnormal{III})$ and a standard derivative estimate for harmonic functions (applied to $v_{a}$). 

Finally, to see the claim that $v$ is of class $GC^{1, \alpha}$ in $B_{1/2},$ note that we can apply, for any $z \in \Gamma_{v} \cap B_{1/2}$, the estimate just proved with $v_{z, 1/2}$ in place of $v$ to obtain the corresponding decay estimate at base point $z$ and with some $\phi_{z} \,: \, \R^{n} \to \A_Q(\R)$ in place of $\varphi$ where $\mathbf{v}(\phi_{z})\in \CC_Q$ or 
	 ${\mathbf v}(\phi_{z}) = Q|L_{z}|$ for some hyperplane $L_{z}$. Then, setting  
$\CC_{v} = \{z \in \Gamma_{v} \cap B_{1/2} \, : \,  {\bf v}(\varphi_{z}) \in {\mathcal C}_{Q}\}$, 
${\mathcal B}_{v} = \Gamma_{v} \cap B_{1/2} \setminus {\mathcal C}_{v}$ and 
${\mathcal R}_{v} = B_{1/2} \setminus \Gamma_{v}$, we can employ property $(\FB4)$ to verify the requirement (A)(i) of the definition of generalised-$C^{1, \alpha}$ (i.e.\ Definition~\ref{genC1alpha}); Theorem~\ref{thm:B7} to verify the requirement (A)(ii) of the definition; and standard pointwise estimates for harmonic functions together with the decay estimate of the present theorem to verify the requirements A(iii) and (B) of the definition.  
\end{proof}

\subsection{Proof of Theorem~A}
We can now prove our main result in this work, Theorem~\ref{thm:A}. We will accomplish this by making use of two key ingredients we now have at our disposal: (i) the asymptotic decay estimates provided by Theorem~\ref{coarse_reg} for the coarse blow-ups, valid in uniform-sized neighbourhoods of a classical singularity or a branch point of a coarse blow-up; (ii) the estimate provided by Theorem~\ref{thm:fine_reg}, which gives decay of a varifold as in Theorem~\ref{thm:A} towards a unique classical tangent cone if and when it reaches, upon rescaling about a  point, a scale at which the degeneration of the varifold towards a hyperplane is not too rapid (i.e.\ a scale at which the fine excess relative to a classical cone is significantly smaller than the coarse excess relative to any hyperplane).

\begin{proof}[Proof of Theorem~\ref{thm:A}]
	We first claim the following: \emph{there exist $\epsilon = \epsilon(n, Q) \in (0, 1)$ and $\theta = \theta(n, Q) \in (0, 1)$ 
	such that if $P$ is a hyperplane of $\R^{n+1}$ with ${\rm dist}_{\H} \,(P \cap (\R \times B_{1}), \{0\} \times B_{1}) < \epsilon$, 
	and if $V \in \S_{Q}$ is such that 
	$\Theta_{V}(0) \geq Q$, $(\omega_{n}2^{n})^{-1}\|V\|(B_{2}^{n+1}(0)) < Q + 1/2$ and, 
	$\hat{E}_{V, P}^{2} \equiv \int_{\pi_{P}^{-1}(P \cap B_{1}^{n+1}(0))}  {\rm dist}^{2}(x, P) \, \ext\|V\| < \epsilon$, then either:}
	\begin{itemize}
	\item[(i)] \emph{there is a hyperplane $\widetilde{P}$ with 
	${\rm dist}_{\H} \, (\widetilde{P} \cap (\R \times B_{1}), P \cap (\R \times B_{1})) < C \hat{E}_{V, P}$, and  
	$\theta^{-n-2} \int_{\pi^{-1}_{\widetilde{P}}(\widetilde{P} \cap B_{\theta}^{n+1}(0))} {\rm dist}^{2}, (x, \widetilde{P}) \, \ext\|V\| \leq \frac{1}{2} \hat{E}_{V, P}^{2}$, or};
	\item[(ii)] \emph{there is a cone $\BC \in \CC_{Q}$ with ${\rm dist}_{\H} \, ({\rm spt} \, \|\BC\| \cap (\R \times B_{1}), P \cap (\R \times B_{1})) < C \hat{E}_{V, P}$, and $\rho^{-n-2}\int_{\R\times B_\rho}\dist^2(X,\spt\|\BC\|)\ \ext\|V\|\leq C\rho^{2\mu} \hat{E}^2_{V,P}$ for all $\rho\in (0,\theta/8]$.}
	\end{itemize}
\emph{Here $C = C(n, Q) \in (0, \infty)$ and $\mu  = \mu(n, Q) \in (0, 1)$}. 

To prove this, we argue by contradiction. So suppose we have a sequence of varifolds $(V_{k})_{k} \subset \S_{Q}$ and a sequence of hyperplanes $(P_{k})_{k}$ with ${\rm dist}_{\H} \,(P_{k} \cap (\R \times B_{1}), \{0\} \times B_{1}) \to 0$ such that 
$\Theta_{V_{k}}(0) \geq Q$, $(\omega_{n}2^{n})^{-1}\|V_{k}\|(B_{2}^{n+1}(0)) < Q + 1/2$ and 
	$\hat{E}_{V_{k}, P_{k}} \to 0$. It suffices to prove, with $\theta = \theta(n, Q) \in (0, 1)$, $\mu = \mu(n, Q) \in (0, 1)$ and 
	$C = C(n, Q) \in (0, \infty)$ to be chosen, that we have for infinitely many $k$, either:
	\begin{itemize}
		\item[(I)] there is a hyperplane $\widetilde{P}_{k}$ with 
	${\rm dist}_{\H} \, (\widetilde{P}_{k} \cap (\R \times B_{1}), P_{k} \cap (\R \times B_{1})) < C \hat{E}_{V_{k}, P_{k}}$ and 
	$\theta^{-n-2} \int_{\pi^{-1}_{\widetilde{P}_{k}}(\widetilde{P}_{k} \cap B_{\theta}^{n+1}(0))} {\rm dist}^{2}(X, \widetilde{P}_{k}) \, \ext\|V_{k}\|(X) \leq \frac{1}{2} \hat{E}_{V_{k}, P_{k}}^{2}$, or;
	\item[(II)] there is $\BC_{k} \in \CC_{Q}$ with ${\rm dist}_{\H} \, ({\rm spt} \, \|\BC_{k}\| \cap (\R \times B_{1}), P_{k} \cap (\R \times B_{1})) < C \hat{E}_{V_{k}, P_{k}}$ and $\rho^{-n-2}\int_{\R\times B_\rho}\dist^2(X,\spt\|\BC_k\|)\ \ext\|V_k\| \leq C\rho^{2\mu}{\hat E}^2_{V_{k}, P_{k}}$ for all $\rho \in (0, \theta/8]$.
	\end{itemize}
Let $\Gamma_{k} \, : \, \R^{n+1} \to \R^{n+1}$ be a rotation such that $\Gamma_{k}(P_{k}) = \{0\} \times \R^{n}$ and $\|\Gamma_{k} - {\rm Identity}\| \to 0$. Let $\widetilde{V}_{k} = 
(\Gamma_{k})_\# V_{k}$. Then $\hat{E}_{k} = \hat{E}_{\widetilde{V}_{k}, \{0\} \times \R^{n}} = \hat{E}_{V_{k}, P_{k}} \to 0$. Let 
$v \in \FB_{Q}$ be a coarse blow-up of $(\widetilde{V}_{k})$. Since $\Theta_{\widetilde{V}_{k}}(0) \geq Q$, we have that $0 \in \Gamma_{v}^{\textnormal{HS}}$ and $v_{a}(0) =0$ (by the remark following the list of properties $(\FB1)-(\FB6)$ in Section~\ref{initial-cb}), so by Theorem~\ref{coarse_reg}, there is $\phi \, : \, \R^{n} \to {\mathcal A}_{Q}(\R)$ with ${\bf v}(\phi) \in \CC_{Q}$ or ${\bf v}(\phi) = Q|L|$ for some hyperplane $L$ such that  
	for every $\sigma \in (0, 3/16]$, 
	\begin{equation}\label{linear-est}
	\sigma^{-n-2}\int_{B_\sigma(0)}\G(v(x),\phi(x))^2 \, \ext x\leq C_{1}{\sigma}^{2\alpha} \int_{B_{3/8}}|v|^2,
	\end{equation}
	where $C_{1} = C_{1}(n, Q) \in (0, \infty)$ and $\alpha = \alpha(n, Q) \in (0, 1)$. Since $\int_{B_{1}}|v|^{2} \leq 1$, this and homogeneity of $\varphi$ imply that 
	\begin{equation}\label{cone-bound}
	\int_{B_{1}} |\phi|^{2} \leq C_{2}
	\end{equation}
	for some $C_{2} = C_{2}(n, Q) \in (0, \infty).$ Also, note that (\ref{linear-est}) implies that 	
	$\sigma^{-n-2}\int_{B_\sigma(0)} |v_{a}(x) - \phi_{a}(x)|^2 \, \ext x\leq 2Q^{-2}C_{1}{\sigma}^{2\alpha} \int_{B_{3/8}}|v|^2$ for all $\sigma \in (0, 3/16]$, so since $v_{a}$ is harmonic, $v_{a}(0) = 0$, and $\phi$ is homogeneous of degree 1, it follows that 
	\begin{equation}\label{avgs}
	\phi_{a}(x) = Dv_{a}(0)\cdot x \;\; \mbox{for $x \in \R^{n}$}.
	  \end{equation}
	Choose $\theta = \theta(n, Q) \in (0, 1)$ such that 
	\begin{equation}\label{conditions}
	\max\{C_{1}(2\theta)^{\alpha}, (2\theta)^{\alpha}\} < \min\{3/16, \epsilon\}
	\end{equation}
	 where $\epsilon = \epsilon(n, Q, \alpha)$ is as in Theorem~\ref{thm:B7} (taken with $\alpha$ equal to the present value of $\alpha$ given by  (\ref{linear-est})). 

\emph{Case 1: $(2\theta)^{-n-2}\int_{B_{2\theta}} |v(x) - \phi_{a}(x)|^{2} \, \ext x < (2\theta)^{\alpha}$.}  In this case, set $\widetilde{P}_{k} = \Gamma_{k}^{-1} (\graph \, {\hat E}_{k}\phi_{a})$. Noting by (\ref{avgs}) that $\phi_{a}$ is linear, it is then straightforward to verify that option (I) holds for infinitely many $k$, with $C = C(n, Q)$.

\emph{Case 2: Case 1 fails.} In this case we have 
\begin{equation}\label{case2} 
(2\theta)^{-n-2}\int_{B_{2\theta}} |v(x) - \phi_{a}(x)|^{2} \, \ext x \geq (2\theta)^{\alpha};
\end{equation}
it follows that we must also have that ${\bf v}(\phi) \in \CC_{Q}$, for if not then ${\bf v}(\phi) = Q|L|$ for some hyperplane whence 
$L = {\rm graph} \, \phi_{a}(x)$ and $\phi^{j}(x) = \phi_{a}(x)$ for every $j=1, 2, \ldots, Q,$ so that  by (\ref{linear-est}),  
	$\sigma^{-n-2}\int_{B_\sigma(0)}|v(x) - \phi_{a}(x)|^2 \, \ext x\leq C_{1}{\sigma}^{2\alpha} \int_{B_{3/8}}|v|^2$ for all $\sigma \in (0, 3/16]$ and thus taking $\sigma = 2\theta$ and using (\ref{conditions}) we see that Case 1 must hold contrary to our assumption. Now note that there is a sequence of rotations 
	$\widetilde{\Gamma}_{k} \, : \, \R^{n+1} \to \R^{n+1}$ with $\widetilde{\Gamma}_{k}({\rm graph} \, \hat{E}_{k}\phi_{a}) = \{0\} \times \R^{n}$ 
	and $\|\widetilde{\Gamma}_{k} - {\rm Identity}\| \to 0$ such that 
	$$\widetilde{v}(x) \equiv \|v(2\theta (\cdot)) - \phi_{a}(2\theta (\cdot))\|_{L^{2}(B_{1})}^{-1}(v(2\theta x) - \phi_{a}(2\theta x))$$ is the coarse blow-up of a subsequence of $((\eta_{0, 2\theta})_\# \widetilde{\Gamma}_{k \, \#} \, \widetilde{V}_{k})_{k}$.  Then we have by (\ref{linear-est}), 
$$\int_{B_{1}} {\mathcal G}(\widetilde{v}, \widetilde{\phi})^{2} < C_{1}(2\theta)^{2\alpha}(2\theta)^{2}\|v(2\theta (\cdot)) - \phi_{a}(2\theta (\cdot))\|_{L^{2}(B_{1})}^{-2} \leq C_{1} (2\theta)^{\alpha} < \epsilon,$$
where $\widetilde{\phi}(x) = \|v(2\theta (\cdot)) - \phi_{a}(2\theta (\cdot))\|_{L^{2}(B_{1})}^{-1}\phi_{f}(2\theta x).$ Since $\|\widetilde{v}\|_{L^{2}(B_{1})} = 1$ and $(\widetilde{\phi})_{a} \equiv 0,$ we see, by the choice of $\epsilon$ and Theorem~\ref{thm:B7}(i), that for infinitely many $k$,  the hypotheses of Theorem~\ref{thm:fine_reg} are satisfied with $W_{k} = (\eta_{0, 2\theta})_\# \widetilde{\Gamma}_{k \, \#} \, \widetilde{V}_{k}$ in place of $V$ and $\hat{\BC}_{k} \equiv {\bf v}({\hat E}_{W_{k}, \{0\} \times \R^{n}} \widetilde{\phi})$ in place of $\BC.$ Hence by applying Theorem~\ref{thm:fine_reg}, and noting that by (\ref{case2}) and (\ref{cone-bound}) we have $\int_{B_{1}} |\widetilde{\phi}|^{2} \leq (2\theta)^{-\alpha}C_{2} \equiv C = C(n, Q),$ we see that option (II) must hold for infinitely many $k$ (with ${\bf C}_{k} = \left(\widetilde{\Gamma}_{k} \circ \Gamma_{k}\right)^{-1}_{\#} \, \widetilde{\BC}_{k},$ where $\widetilde{\BC}_{k}$ is the cone $\widetilde{\BC}$ provided by Theorem~\ref{thm:fine_reg} when that theorem is 
applied with $W_{k}$ in place of $V$, $\hat{\BC}_{k}$ in place of $\BC$, and $\mu = \alpha$). This establishes the claim asserted at the beginning of the proof. 

Now we can apply the claim iteratively to deduce that for any hyperplane $P$ with ${\rm dist}_{\H} \,(P \cap (\R \times B_{1}), \{0\} \times B_{1}) < \epsilon$, one of the following must hold:

\begin{itemize} 
\item[(i)$^{\prime}$] \emph{there is a sequence of hyperplanes $(P_{k})_{k}$ with $P_{1} = P,$    
	${\rm dist}_{\H} \, ({P}_{k+1} \cap (\R \times B_{1}), P_{k} \cap (\R \times B_{1})) < C \hat{E}_{(\eta_{0, \theta^{k}})_\# V, P_{k}}$ and 
	$\hat{E}_{(\eta_{0, \theta^{k}})_\# V, P_{k+1}}^{2}\leq \frac{1}{2} \hat{E}_{(\eta_{0, \theta^{k-1}})_\# V, P_{k}}^{2}$ for all $k \geq 1$, or};
	
	\item[(ii)$^{\prime}$] \emph{there is an integer $I \geq 1$ and a finite sequence of hyperplanes $P_{1} = P, P_{2}, \ldots, P_{I}$ such that \textnormal{(i)$^{\prime}$} holds for for $k=1, 2, \ldots, I-1$ (if $I \geq 2$), and there is $\BC \in \CC_{Q}$ with 
	${\rm dist}_{\H} \, ({\rm spt} \, \|\BC\| \cap (\R \times B_{1}), P_{I} \cap (\R \times B_{1})) < C \hat{E}_{(\eta_{0, \theta^{I-1}})_\#V, P_{I}}$ and $$(\rho\theta^{I-1})^{-n-2}\int_{\R\times B_{\rho\theta^{I-1}}}\dist^2(X,\spt\|\BC\|)\ \ext\|V\| \leq C\rho^{2\mu}{\hat E}^2_{(\eta_{0, \theta^{{I-1}}})_\#V, P_{I}}$$ for all $\rho \in (0, \theta/8]$}.
\end{itemize} 

From these, it is standard to deduce that there are constants $\beta = \beta(n, Q) \in (0, 1)$ and $C = C(n, Q)\in (0, \infty)$ such that  for any hyperplane $P$ with ${\rm dist}_{\H} \,(P \cap (\R \times B_{1}), \{0\} \times B_{1}) < \epsilon$, we have either: 
\begin{itemize}
\item[(A)]  there is a (unique) hyperplane $P_{0}$  with ${\rm dist}_{\H} \, ({\rm spt} \, P_{0} \cap (\R \times B_{1}), P \cap (\R \times B_{1})) < C \hat{E}_{V, P}$ and $\hat{E}_{(\eta_{0, \rho})_\# V, P_{0}} \leq C\rho^{\beta} \hat{E}_{V, P}$ for all $\rho \in (0, \theta/8]$, or; 
\item[(B)] there is a (unique) cone $\BC_{0} \in \CC_{Q}$ with ${\rm dist}_{\H} \, ({\rm spt} \, \|\BC_{0}\| \cap (\R \times B_{1}), P \cap (\R \times B_{1})) < C \hat{E}_{V, P}$ and $\rho^{-n-2}\int_{\R\times B_\rho}\dist^2(X,\spt\|\BC\|)\ \ext\|V\| \leq C\rho^{2\beta}{\hat E}^2_{V, P}$ for all $\rho \in (0, \theta/8]$.
\end{itemize}

Indeed, (A) holds if  (i)$^{\prime}$ holds, and (B) holds if (ii)$^{\prime}$ holds. In particular, $V$ has a unique tangent cone at $0$ which is $\BC_{0}$ if (B) holds, and $k|P_{0}|$ for some constant integer $k$ if (A) holds in which case $k=Q$ since $\Theta_{V}(0) \geq Q$ and 
$(\omega_{n}2^{n})^{-1}\|V\|(B_{2}^{n+1}(0)) \leq Q + 1/2$; thus we must have that $\Theta_{V}(0) = Q$. 

To complete the proof, note that if the hypotheses of Theorem~\ref{thm:A} are satisfied with  $\epsilon = \epsilon(n, Q)$ sufficiently small, then for any $Z \in \R \times B_{3/4}$ with $\Theta_{V}(Z) \geq Q$,  we may repeat the argument leading to (A) or (B) with $(\eta_{Z, 1/4})_\#V$ in place of $V$.  This gives $\{Z \in \R \times B_{3/4} \, : \, \Theta_{V}(Z) \geq Q\} = 
\{Z \in \R \times B_{3/4} \, : \, \Theta_{V}(Z) = Q\}$. Set
\begin{align*}
{\mathcal B}_{V} = \left\{Z \in \R \times B_{3/4} \, : \, \Theta_{V}(Z) = Q, \;\mbox{and (A) holds  with $(\eta_{Z, 1/4})_\# V$ in place of $V$}\right.\\ 
&\hspace{-2in}\left.\mbox{and a hyperplane $P_{Z}$ in place of $P_{0}$}\right\},
\end{align*} 
\begin{align*}
\CC_{V} = \left\{Z \in \R \times B_{3/4} \, : \, \Theta_{V}(Z) = Q, \; \mbox{and (B) holds  with $(\eta_{Z, 1/4})_\# V$ in place of $V$}\right.\\ 
&\hspace{-2.5in}\left.\mbox{and a cone $\BC_{Z} \in \CC_{Q}$ in place of $\BC_{0}$}\right\}, \;\; {\rm and}
\end{align*} 
$$\Omega= B_{1/2}(0)\setminus \pi(\B_{V} \cup \CC_{V}),$$
where $\pi:\R\times \R^n\to \{0\}\times\R^n$ is the orthogonal projection. Then 
$\B_{V} \cup \CC_{V} = \{Z \in B_{3/4} \, : \, \Theta_{V}(Z) \geq Q\},$ and in particular $\B_{V} \cup \CC_{V}$ is relatively closed in $\R \times B_{3/4}$ so $\Omega$ is an open subset of $B_{1/2}$. We have, by Theorem~\ref{sheeting} and the decay estimates provided by (A) and (B)  that for each $Y\in \Omega$ and 
$\rho_Y = \frac{1}{2}\dist(Y,\pi(\B_V \cup \CC_V))$, the varifold $V \res (\R \times B_{\rho_{Y}}(Y))$ is the sum of multiplicity 1 varifolds corresponding to $Q$ embedded, ordered graphs of smooth solutions to the minimal surface equation on $B_{\rho_{Y}}(Y)$ with small 
gradient;  this gives functions $u^{j} \, : \, \Omega \to \R$ with $u^{1} \leq u^{2} \leq \cdots \leq u^{Q}$ solving the minimal surface equation and with small gradient, such that $V \res \R \times \Omega  = \sum_{j=1}^{Q} |{\rm graph} \, u_{j}|$. Moreover, since for each $z\in B_{1/2} \setminus\Omega$, $\pi^{-1}(z) \cap \spt\|V\|$ is a single-point $Z \in {\mathcal B}_{V} \cup \CC_{V}$, we can extend these functions to obtain a function $u \, : \, B_{1/2} \to {\mathcal A}_{Q}(\R)$ with  $u^{1} \leq u^{2}\leq \cdots \leq u^{Q}$. 
Setting $\CC_{u} = \pi(\CC_{V})$, ${\mathcal B}_{u} = \pi({\mathcal B}_{V})$ and ${\mathcal R}_{u} = \Omega$, the desired properties for $u$ to be generalised-$C^{1,\alpha}$ in $B_{1/2}$ follow by using: (i) the decay estimate provided by alternative (A) together with Theorem~\ref{thm:B} (or Theorem~\ref{thm:fine_reg}) to give the local description of $u$ near points in $\CC_{u}$; (ii) the decay estimate provided by alternative (B) to verify differentiability of $u$ at points in $\B_{u};$ and (iii) standard elliptic estimates together with the Hausdorff distance estimates between $P_{Z}$ and $P$, and between $C_{Z}$ and $P,$ as provided by (A) and (B) for any hyperplane $P$ close to $\{0\} \times \R^{n}$, to verify that $Du$ is in $C^{0, \alpha}(K)$ for any compact $K \subset {\mathcal R}_{u} \cup \B_{u}$.
\end{proof}

\appendix
\section{Hausdorff dimension bound for the branch set of coarse blow-ups}\label{app:A}

Here we show that the results in Part~\ref{main-thm-proof} above can be used to bound the Hausdorff dimension of the branch set $\B_v$ for any $v \in \FB_{Q}$, i.e.\ we prove the following result:
\begin{theoremA}\label{AP1} $\dim_\H(\B_v)\leq n-2$ for every $v\in \FB_Q$. 
\end{theoremA}

\textbf{Remark:} When $Q=2$, given Theorem \ref{thm:A} this follows immediately from \cite{simon2016frequency}; moreover, from \cite{krummel2013fine} we in fact know that $\B_v$ is countably $(n-2)$-rectifiable, with $v-v_a$ having unique blow-ups at $\H^{n-2}$-a.e. point in $\B_v$.

Before starting the proof of this, we point out the following $C^{0,1}$ estimate for the average-free part of each $v\in \FB_Q$, which we shall use in the proof.

\begin{lemmaA}\label{AL1}
	If $v\in \FB_Q$ then: 
	\begin{itemize}
	\item[(i)] for each $z\in B_{1/2}(0) \cap \Gamma_{v}^{\textnormal{HS}}$, and for each $\sigma, \rho$ with $0<\sigma\leq\rho/2<3/32$, we have that
	$$\sigma^{-n}\int_{B_\sigma(z)}|v_f|^2 \leq \left(\frac{\sigma}{\rho}\right)^{2}\rho^{-n}\int_{B_\rho(z)}|v_f|^2;$$
	\item[(ii)] for each $\rho \in (0, 1)$ we have 
	$$\|v_f\|_{C^{0,1}(B_\rho)}:= \sup_{B_{\rho}} \, |v_{f}| + \sup_{x_{1}, x_{2} \in B_{\rho}:\, x_{1} \neq x_{2}} \,  \frac{|v_{f}(x_{1}) - v_{f}(x_{2})|}{|x_{1} - x_{2}|} \leq 
 C\left(\int_{B_1}|v_f|^2\right)^{1/2},$$
where $C = C(n, Q, \rho) \in (0, \infty).$
\end{itemize}
\end{lemmaA}

\begin{proof}
	Fix $v\in \FB_Q$. Then for each $z\in B_{1/2}(0) \cap \Gamma_{v}^{\textnormal{HS}}$, we know from Theorem~\ref{coarse_reg} that there is a function $\phi:B_1(0)\to \A_Q(\R),$ with $\mathbf{v}(\phi)\in \CC_Q$ or ${\mathbf v}(\phi) = Q|L|$ for some hyperplane $L$, for which
	$$\sigma^{-n-2}\int_{B_\sigma(z)}\G(v - v_{a}(z),\phi)^2 \leq C\left(\frac{\sigma}{\rho}\right)^{2\alpha}\cdot\rho^{-n-2}\int_{B_\rho(z)}|v|^2$$
	for every $0<\sigma\leq \rho/2\leq 3/32$, where $\alpha = \alpha(n,Q)\in (0,1)$ and $C = C(n,Q)\in (0,\infty)$. Writing 
	$v^{j} = v_f^{j} + v_a$ and $\phi^{j} = \phi_{f}^{j} + \phi_{a},$ and noting that $\sum_{j=1}^{Q} v_{f}^{j} = \sum_{j=1}^{Q} \phi_{f}^{j} = 0$ 
	pointwise, 
	we have that $\G(v - v_{a}(z),\phi)^2 = \G(v_{f}, \phi_{f})^{2} + Q|v_{a} - v_{a}(z) - \phi_{a}|^{2}$, so it follows from the above inequality that
	\begin{align*}
	\sigma^{-n-2}\int_{B_\sigma(z)}|v_f|^2 & \leq 2\sigma^{-n-2}\int_{B_\sigma(z)}|\phi_f|^2 
	+ 2\sigma^{-n-2}\int_{B_\sigma(z)}\G(v - v_{a}(y),\phi)^2\\
	& \leq C\int_{B_{1/2}(z)}|\phi_f|^2  + C\int_{B_{1/2}(z)}|v|^2.\\
	\end{align*}
	Note that from Theorem~\ref{frequency}, we know that the frequency function associated with $v_{f}$ is well-defined and is monotone. In the same manner as in the proof of Proposition \ref{classification}, the preceding inequality then shows that the frequency 
	$N_{v_f}(z)\geq 1$ for every $z\in B_{1/2}(0) \cap \Gamma_{v}^{\textnormal{HS}}$. Thus from frequency monotonicity, we therefore see that for each such $z$ and all $0<\sigma\leq\rho <1/2$:
	$$\sigma^{-n}\int_{B_\sigma(z)}|v_f|^2 \leq \left(\frac{\sigma}{\rho}\right)^2\rho^{-n}\int_{B_\rho(z)}|v_f|^2$$
	which is the desired estimate in part~(i). To see part~(ii), first consider the case $\rho = 1/2$. Notice that if $z\in B_{1/2}(0) \setminus \Gamma_{v}^{\textnormal{HS}}$, then from $(\FB4\text{II})$ we see that, setting $\rho_{z} = \dist(z,\Gamma_v\cup\del B_1(0)),$ 
	$v_f$ is harmonic on $B_{\rho_{z}}(z)$ (i.e.\ $\left. v_{f}^{j}\right|_{B_{\rho_{z}}(z)}  
	\equiv \left. (v^{j} - v_{a})\right|_{B_{\rho_{z}}(z)} \, : \, B_{\rho_{z}}(z) \to \R$ is harmonic for each $j=1, 2, \ldots, Q$). 
	 Applying standard estimates for harmonic functions, we thus see that for such $z$ and all $\sigma, \rho$ with $0<\sigma \leq \rho< \rho_{z}$ and any constant $b \in {\mathcal A}_{Q}(\R)$, 
	$$\sigma^{-n}\int_{B_\sigma(z)}|v_f - v_{f}(z)|^2 \leq C\left(\frac{\sigma}{\rho}\right)^2\rho^{-n}\int_{B_\rho(z)} |v_f - b|^2,$$
where $C = C(n).$ From here we may apply standard Campanato-style arguments (see e.g.\ \cite[Lemma 4.3]{wickramasekera2014general} or \cite{minter2021campanato}) to reach the desired conclusion for $\rho = 1/2$. In view of property $(\FB5\textnormal{I})$, the claimed estimate for arbitrary $\rho \in (0, 1)$ follows from the case $\rho = 1/2.$ 
\end{proof}

\begin{proof}[Proof of Theorem \ref{AP1}]
	Fix $v\in \FB_Q$ and assume that there is some $j \in \{1, \ldots, Q\}$ such that $v^{j} \not\equiv v_{a}$ on $B_{1}$. Then it follows from Theorem~\ref{frequency} that for every ball $B_{\rho}(y) \subset B_{1}$, $v_{f}$ is not identically zero on $B_{\rho}(y).$  Set $w:=v_f$. We know from Theorem~\ref{coarse_reg} that $w$ is generalised-$C^{1,\alpha}$ for some $\alpha = \alpha(n,Q)$, and from Theorem~\ref{frequency} we know that at every $x_0\in \B_v$ the frequency function $N_{w;x_0}$ is monotone and that the frequency $N_{w}(x_0)$ is well-defined. From the bounds provided by Theorem~\ref{coarse_reg} and the monotonicity of $N_{w;x_0}$ we can readily check that $N_{w}(x_0)\geq 1+\alpha$ at every $x_0\in \B_v$. So set:
	$$\CF:= \{x\in B_1(0):N_w(x)\geq 1+\alpha\}.$$
	The above tells us that $\B_v\subset \CF$, and so it suffices to show $\dim_\H(\CF)\leq n-2$. For $y\in \CF$ and $\rho>0$ set $w_{y,\rho}(x):= \|w(y+\rho x)\|_{L^2(B_1(0))}^{-1}w(y+\rho x)$. Now from $(\FB5\text{I})$, we know that $v_{y,\rho}\in \FB_Q$, and moreover
	$$\|v(y+\rho x)\|_{L^2(B_1(0))}\cdot (v_{y,\rho})_f = \|w(y+\rho x)\|_{L^2(B_1(0))}\cdot w_{y,\rho}.$$
	In particular, as $(v_{y,\rho})_f$ and $w_{y,\rho}$ only differ by a multiplicative constant, we have that the squash inequality (Lemma \ref{squash}) and squeeze identity (Lemma \ref{squeeze}) hold for $w_{y,\rho}$ from the corresponding results for $(v_{y,\rho})_f$, i.e.\
	$$\int_{B_1}|Dw_{y,\rho}|^2\zeta \leq -\int_{B_1(0)}\sum^Q_{\alpha=1}w_{y,\rho}^\alpha Dw_{y,\rho}^\alpha \cdot D\zeta$$
	$$\int_{B_1}\sum^Q_{\alpha=1}\sum^n_{i,\ell=1}\left(|Dw_{y,\rho}^\alpha|^2\delta_{i\ell} - 2D_iw^\alpha_{y,\rho} D_{\ell}w^\alpha_{y,\rho}\right)D_i\zeta^\ell = 0$$
	for any $\zeta, \zeta^{\ell} \in C^{1}_{c}(B_{1})$. For any $r \in (0,1)$, the squash inequality provides a bound of the form $\|Dw_{y,\rho}\|_{L^2(B_{r})} \leq C(1-r)^{-1}\|w_{y,\rho}\|_{L^2(B_{1})} = C(1-r)^{-1}$ where $C = C(n)$. This tells us that for each $r \in (0, 1)$, we have a uniform bound on $\|w_{y,\rho}\|_{W^{1,2}(B_{r})}$  depending on $r$ but independent of $\rho$, and so given any sequence $(\rho_{j})$ with $\rho_{j} \to 0^{+}$, we may pass to a subsequence without relabelling (using a diagonal argument) to ensure that $w_{y,\rho_j}\to w_* \in W^{1,2}_{\rm loc}(B_{1})$, where the convergence is locally strongly in $L^2$ and locally weakly in $W^{1,2}.$ Moreover, by  frequency monotonicity, we have for $0 < r < 1$ the doubling condition 
	$\rho^{-n}\|v_{f}\|^{2}_{L^{2}(B_{\rho}(y))} \leq C(v, r) (r \rho)^{-n}\|v_{f}\|^{2}_{L^{2}(B_{r \rho}(y))}$  and so $w_{*} \not\equiv 0$ on any ball in $B_{1}$.
	
	We now claim that in fact the convergence $w_{y,\rho_j}\to w_*$ is in the strong $W^{1,2}$ topology locally on $B_{1}$. By standard results it suffices to show that $\|Dw_{y,\rho_j}\|_{L^2(B_{r})}\to \|Dw_*\|_{L^2(B_{r})}$ for each $r \in (0, 1)$. So fix $r \in (0, 1)$. We know from the local weak convergence in $W^{1,2}$ that
	$$\|Dw_*\|_{L^2(B_{r})} \leq \liminf_{j\to\infty}\|Dw_{y,\rho_j}\|_{L^2(B_{r})}$$
	and so it suffices to show that $\limsup_{j\to\infty}\|Dw_{y,\rho_j}\|_{L^2(B_{r})} \leq \|Dw_*\|_{L^2(B_{r})}$. To show this, first note that applying the energy non-concentration estimate, Lemma \ref{noncon2}, to $v_{y,\rho_j}$, with $\delta$ replaced by $\delta\cdot\|v(y+\rho_j x)\|_{L^2(B_1(0))}^{-1}\cdot\|w(y+\rho_j x)\|_{L^2(B_1(0))}$, we get for any $\delta>0$
	\begin{equation}\label{A:0}
	\int_{B_{r}}\sum^Q_{\alpha=1}\one_{\{|w^\alpha_{y,\sigma}|<\delta\}}|Dw^\alpha_{y,\sigma}|^2 \leq C(1-r)^{-2}\delta,
	\end{equation}
	for some $C = C(n,Q)$. We now claim that we in fact have local uniform convergence $w_{y,\rho_j}\to w_*$ in $B_{1}$. Indeed, applying Lemma \ref{AL1} to $v_{y,\rho_j}$ we get local uniform bounds on the $C^{0,1}$ norm of $w_{y,\rho_j}$ for all $j$, and thus by the Arzel\`a-Ascoli theorem, we can pass to a subsequence to ensure that in fact $w_{y,\rho_j}\to w_*$ locally uniformly, and in particular $w_*$ is in $C^{0,1}$. Moreover this shows that for any $\delta>0$  and any compact $K \subset B_{1}$ we have, for all $j$ sufficiently large (depending on $\delta$ and $K$), 
		$$\{|w_{y,\rho_j}|\geq \delta\} \cap K \subset \{|w_*|\geq 3\delta/4\} \cap K\subset \{|w_*|\geq \delta/2\} \cap K \subset \{|w_{y,\rho_j}| \geq\delta/4\} \cap K$$
	and thus from property $(\FB4)$, $w_{y,\rho_j}$ is harmonic on $\{|w_*| > \delta/2\}$, and we can upgrade the convergence to $C^2$ convergence on $K\cap \{|w_*|\geq 3\delta/4\}$; in particular we have $Dw_{y,\rho_j}\to Dw_*$ strongly in $L^2$ on this set.

Now fix $\epsilon>0$ and choose $\delta\in (0,\epsilon)$ such that $C(1-r)^{-2}\delta<\epsilon$, where $C$ is as in (\ref{A:0}). Then we have
	\begin{align*}
	\int_{B_{r}}|Dw_{y,\rho_j}|^2 & =  \int_{B_{r}} \sum_{\alpha=1}^Q\one_{\{|w^\alpha_{y,\rho_j}|<\delta\}}|Dw^\alpha_{y,\rho_j}|^2 + 
	\int_{B_{r}} \sum^Q_{\alpha=1}\one_{\{|w^\alpha_{y,\rho_j}|\geq\delta\}}|Dw^\alpha_{y,\rho_j}|^2 \\
	& \leq C(1-r)^{-2}\delta + \int_{B_{r}\cap \{|w_*|\geq 3\delta/4\}}|Dw_{y,\rho_j}|^2\\
	& < \epsilon + \epsilon + \int_{B_{r}\cap \{|w_*|\geq3\delta/4\}}|Dw_*|^2\\
	& \leq 2\epsilon + \int_{B_{r}}|Dw_*|^2\; ,
	\end{align*}
	where the second and third inequalities holds for all $j$ sufficiently large. Thus taking $\epsilon\downarrow 0$ we get
	$$\limsup_{j\to\infty}\int_{B_{r}}|Dw_{y,\rho_j}|^2 \leq \int_{B_{r}}|Dw_*|^2$$
	as desired; thus $w_{y,\rho_j}\to w_*$ strongly in $W^{1,2}(B_{r})$ for every $r \in (0, 1)$.
	
	In particular, we are now able to take $\rho = \rho_j$ in the squash inequality and squeeze identity for $w_{y,\rho}$ and take $j\to\infty$ to see that
	$$\int_{B_1(0)}|Dw_*|^2\zeta\leq -\int_{B_1(0)}\sum^Q_{\alpha=1}w^\alpha_*Dw^\alpha_*\cdot D\zeta$$
	$$\int_{B_1(0)}\sum^Q_{\alpha=1}\sum^n_{i,j=1}\left(|D^\alpha w_*|^2 \delta_{ij} - 2D_i w^\alpha_* D_j w^\alpha_*\right) D_i\zeta^j = 0$$
	for each $\zeta, \zeta^{\ell} \in C^1_c(B_1(0))$, i.e., the squash and squeeze identities hold for $w_*$. In the same way as in Theorem \ref{frequency}, we can define a frequency function for $w_*$. Moreover from the strong $W^{1,2}_{\text{loc}}$ convergence, we can now show that: $N_{w_*; 0}(\rho) = N_w(y)\geq 1+\alpha$ for every $\rho \in (0, 1)$ and thus $w_*$ is homogeneous of degree $N_{w_*}(0)$, and extends to $\R^{n}$ as homogeneous degree $N_{w_*}(0)$ function for which the above squash and squeeze identities hold for all $\zeta, \zeta^{\ell} \in C^{1}_{c}(\R^{n})$. Then the homogeneity of $w_*$ and frequency monotonicity give in the usual way that $N_{w^*}(y) \leq N_{w^*}(0)$ for all $y\in \R^n$ and that the \textit{spine} $S_{w_*}:= \{x_0\in \R^n: N_{w_*}(0) = N_{w_*}(x_0)\}$ is a linear subspace, along which $w_*$ is translation invariant. Note that the frequency is upper semi-continuous with respect to both the function and spatial variables when the convergence is strong in $W^{1,2}.$

	Clearly we cannot have $\dim(S_{w_*}) =n$; if this were the case $w_*$ would be a constant, and hence zero, but we have shown that $w_*$ is non-zero. Moreover we cannot have $\dim(S_{w_*}) = n-1$; for if this were the case then we can find an $(n-1)$-dimensional subspace $L$ along which $w_*$ is translation invariant, and so $w_*$ is determined by a function on $\R$ with values in $\R^{Q}$ which is harmonic on $\R\setminus\{0\}$ and homogeneous of degree $N_{w_*}(0) \geq 1+\alpha$, but no such function exists. So we have that 
$\dim(S_{w_*})\leq n-2$. 
	
	Now the above analysis was completed at an arbitrary point of $\CF$. To prove that $\dim_\H(\CF)\leq n-2$, we shall follow the dimension reduction argument established in \cite{almgrenalmgren} (and revisited in \cite{simon1996theorems}). Indeed, setting $\eta:\R^n\to \R^n$ to be $\eta_{y,\rho}(x):= \rho^{-1}(x-y)$, we claim that for each $y\in \CF$ and $\delta>0$, there is an $\epsilon = \epsilon(v,y,\delta)$ such that for $\rho\in (0,\epsilon]$:
	\begin{equation}\label{E:neighbourhood}
	\eta_{y,\rho}\left(\{x\in B_\rho(y): N_{w}(x)\geq N_{w}(y)-\epsilon\}\right)\subset \text{the $\delta$-neighbourhood of }L_{y,\rho}
	\end{equation}
	for some $(n-2)$-dimensional subspace $L_{y,\rho}$ of $\R^n$. Indeed, if this were false, we could find $\delta>0$ and $y\in \CF$ where it fails, i.e., there are sequences $0<\rho_k<\epsilon_k\downarrow 0$ such that
	$$\{x\in B_1(0): N_{w_{y,\rho_k}}(x) \geq N_w (y) - \epsilon_k\}\not\subset \text{the $\delta$-neighbourhood of }L$$
	for every $(n-2)$-dimensional subspace $L$ of $\R^n$. But we know that (up to a subsequence) $w_{y,\rho_k}\to w_*$ for some $w_*$ as above, with $N_{w_*}(0) = N_w(y)$. In particular we know $\dim(S_{w_*})\leq n-2$, and so as $S_{w_*}$ is the set of points where the frequency of $w_*$ takes the maximal value $\Theta_{w_*}(0)$, we know that there is a $(n-2)$-dimensional subspace $L_0\supset S(w_*)$ and $\alpha>0$ such that
	$$N_{w_*}(x)<N_{w_*}(0) -\alpha\ \ \ \ \text{for all }x\in \overline{B}_1(0)\text{ with }\dist(x,L_0)\geq \delta.$$
	Then we must have by upper semi-continuity of the frequency, for all $k$ sufficiently large,
	$$\{x\in B_1(0): N_{w_y,\rho_k}(x)\geq N_w(0) - \alpha\}\subset \{x:\dist(x,L_0)<\delta\}$$
	which is a contradiction to the original assumption, and so we have established (\ref{E:neighbourhood}).
	
	Now fix $\delta>0$. Define $\CF_i$, $i\in \{1,2,\dotsc\}$, to be the set of points $y\in \CF$ for which (\ref{E:neighbourhood}) holds with $\epsilon = i^{-1}$. Then by (\ref{E:neighbourhood}) we know that $\CF = \cup_{i=1}^\infty \CF_i$. Next for each $q\in \{1,2,\dotsc\}$ set
	$$\CF_{i,q}:= \left\{y\in \CF_i: N_{w}(y) \in \left(\frac{q-1}{i},\frac{q}{i}\right]\right\}$$
	and note that clearly $\CF = \cup_{i,q} \CF_{i,q}$. For any $y\in \CF_{i,q}$ we trivially have by definition that
	$$\CF_{i,q}\subset \{x: N_{w}(x) > N_{w}(y) - i^{-1}\}$$
	and thus by (\ref{E:neighbourhood}), for each $\rho\leq i^{-1}$,
	$$\eta_{y,\rho}(\CF_{i,q}\cap B_\rho(y))\subset \text{the $\delta$-neighbourhood of }L_{y,\rho}$$
	for some $(n-2)$-dimensional subspace $L_{y,\rho}$ of $\R^{n}$.  From this \lq\lq$\delta$-approximation property'', the proof can then be concluded by applying \cite[Section 3.4, Lemma 3]{simon1996theorems}.
\end{proof}
	
\bibliographystyle{alpha} 
\bibliography{references}

\end{document}